\documentclass[11pt,reqno]{amsart}
\usepackage{amssymb,latexsym,graphicx,amscd}
\usepackage{lscape}
\usepackage[all]{xy}
\usepackage{color}
\usepackage{epic,eepic}
\usepackage{xspace}
\usepackage{mathrsfs}
\usepackage{setspace}
\usepackage{cases}
\usepackage{bbm}         
\newtheorem{Thm}{Theorem}[section]
\newtheorem{Lem}[Thm]{Lemma}
\newtheorem{Cor}[Thm]{Corollary}
\newtheorem{Prop}[Thm]{Proposition}
\newtheorem{Conj}[Thm]{Conjecture}

\theoremstyle{definition}

\newcommand{\Z}{\mathbb{Z}}

\newcommand{\N}{\mathbb{N}}
\newcommand{\C}{\mathbb{C}}

\newcommand{\df}{\colon}

\newcommand{\cB}{{\mathcal B}}

\newcommand{\cF}{{\mathcal F}}
\newcommand{\cG}{{\mathcal G}}

\newcommand{\cM}{{\mathcal M}}

\newcommand{\cO}{{\mathcal O}}
\newcommand{\cP}{{\mathcal P}}

\newcommand{\cS}{{\mathcal S}}

\newcommand{\cU}{{\mathcal U}}

\newcommand{\cW}{{\mathcal W}}

\newcommand{\g}{\mathfrak{g}}
\newcommand{\n}{\mathfrak{n}}
\newcommand{\h}{\mathfrak{h}}

\newcommand{\bM}{{\mathbf M}}

\newcommand{\bp}{{\mathbf p}}
\newcommand{\bq}{{\mathbf q}}
\newcommand{\br}{{\mathbf r}}
\newcommand{\bs}{{\mathbf s}}

\newcommand{\bU}{{\mathbf U}}

\newcommand{\bv}{{\mathbf v}} 

\newcommand{\Serre}{{\mathcal I}}

\newcommand{\vpi}{\varpi}
\newcommand{\vph}{\varphi}
\newcommand{\vep}{\varepsilon}
\newcommand{\la}{\lambda}

\newcommand{\ad}{\operatorname{ad}}
\newcommand{\rk}{\operatorname{rank}}
\newcommand{\wt}{\operatorname{wt}}
\newcommand{\mx}{{\rm max}}

\newcommand{\nil}{{\rm nil}}

\newcommand{\rep}{\operatorname{rep}}

\newcommand{\dimv}{\underline{\dim}}

\newcommand{\sgn}{\operatorname{sgn}} 

\newcommand{\tp}{\operatorname{top}}
\newcommand{\sub}{\operatorname{sub}}
\newcommand{\fac}{\operatorname{fac}}
\newcommand{\Inj}{\operatorname{Inj}} 

\newcommand{\Hom}{\operatorname{Hom}}
\newcommand{\Ext}{\operatorname{Ext}}
\newcommand{\End}{\operatorname{End}}

\newcommand{\Ima}{\operatorname{Im}}
\newcommand{\Ker}{\operatorname{Ker}}
\newcommand{\Coker}{\operatorname{Cok}}

\newcommand{\length}{\operatorname{length}}

\newcommand{\ov}{\overline}

\newcommand{\Span}{\operatorname{Span}}

\newcommand{\bil}[1]{\langle #1\rangle}

\newcommand{\out}{\rm out}
\newcommand{\inn}{\rm in}

\newcommand{\Irr}{\operatorname{Irr}}
\newcommand{\supp}{\operatorname{supp}}

\newcommand{\bsm}{\begin{smallmatrix}}
\newcommand{\esm}{\end{smallmatrix}}

\newcommand{\bbsm}{\left[\begin{smallmatrix}}
\newcommand{\besm}{\end{smallmatrix}\right]}

\newcommand{\bbm}{\begin{matrix}}
\newcommand{\ebm}{\end{matrix}}

\newcommand{\diag}{{\rm diag}}

\newcommand{\wti}{\widetilde}

\newcommand{\GL}{\operatorname{GL}}
\newcommand{\Aut}{\operatorname{Aut}}
\newcommand{\Gr}{\operatorname{Gr}} 

\newcommand{\vp}{{\rm l.f.}}

\newcommand{\cry}{{\rm cr}}

\newcommand{\ra}{\rightarrow}

\newcommand{\ttheta}{\wti{\theta}}
\newcommand{\tM}{\wti{\cM}}
\newcommand{\tF}{\wti{\cF}}

\begin{document}

\date{24.02.2017}

\title[Quivers with relations for symmetrizable Cartan matrices IV]
{Quivers with relations for symmetrizable Cartan matrices IV:
Crystal graphs and semicanonical functions}

\author{Christof Gei{\ss}}
\address{Christof Gei{\ss}\newline
Instituto de Matem\'aticas\newline
Universidad Nacional Aut{\'o}noma de M{\'e}xico\newline
Ciudad Universitaria\newline
04510 M{\'e}xico D.F.\newline
M{\'e}xico}
\email{christof.geiss@im.unam.mx}

\author{Bernard Leclerc}
\address{Bernard Leclerc\newline
LMNO, Univ. de Caen\newline
CNRS, UMR 6139\newline
F-14032 Caen Cedex\newline
France}
\email{bernard.leclerc@unicaen.fr}

\author{Jan Schr\"oer}
\address{Jan Schr\"oer\newline
Mathematisches Institut\newline
Universit\"at Bonn\newline
Endenicher Allee 60\newline
53115 Bonn\newline
Germany}
\email{schroer@math.uni-bonn.de}



\begin{abstract}
We generalize Lusztig's nilpotent varieties, 
and 
Kashiwara and Saito's geometric construction of crystal graphs from
the symmetric to the symmetrizable case.
We also construct semicanonical functions in the 
convolution algebras of generalized preprojective algebras.
Conjecturally these functions 
yield semicanonical bases of the
enveloping algebras of the positive part of symmetrizable Kac-Moody algebras. 
\end{abstract}

\maketitle

\setcounter{tocdepth}{1}
\numberwithin{equation}{section}
\tableofcontents

\parskip2mm


\section{Introduction and main results}\label{sec:intro}


\subsection{Introduction}
There is a remarkable geometric universe relating the representation theory
of quivers and preprojective algebras with the representation theory of
symmetric Kac-Moody algebras. 
This includes the realization of the enveloping algebra $U(\n)$ of the positive part $\n$ of
a symmetric Kac-Moody algebra $\g$ as an algebra of constructible
functions on varieties of modules over path algebras \cite{S} and
over preprojective algebras \cite{L1,L2}.
The latter leads to the construction of a semicanonical basis $\cS$ of $U(\n)$
due to Lusztig \cite{L2}.
The elements of $\cS$ are parametrized by the irreducible
components of varieties of modules over preprojective algebras.
Furthermore, closely linked with varieties of modules over preprojective
algebras, there is a 
geometric realization of the crystal graph $B(-\infty)$ of the quantized enveloping algebra $U_q(\n)$ due to
Kashiwara and Saito \cite{KS}.
This crystal graph controls the decompositions of tensor products
of irreducible integrable highest weight $\g$-modules, and it
encodes all crystals graphs and characters of these modules.

Many geometric constructions for symmetric Kac-Moody algebras, 
especially the construction of Lusztig's
semicanonical basis, do not exist for non-symmetric Kac-Moody algebras.
Nandakumar and Tingley \cite{NT} recently
realized $B(-\infty)$ in the symmetrizable case via varieties of modules over 
preprojective algebras associated with species.
In the non-symmetric cases, their construction cannot be carried out over algebraically closed fields, especially not over $\C$.
There exists also a folding technique, which sometimes allows
to transfer results from the symmetric cases to the non-symmetric ones.

In our setting, symmetric and symmetrizable cases are dealt with uniformly.
We generalize Lusztig's nilpotent varieties, 
and 
Kashiwara and Saito's geometric construction of the crystal graph $B(-\infty)$ from
the symmetric to the symmetrizable case.
We also construct semicanonical functions in the 
convolution algebras of generalized preprojective algebras.
Conjecturally these functions 
yield semicanonical bases of the enveloping algebras
$U(\n)$.
In the symmetric cases with minimal symmetrizer, we recover as a special
case Lusztig's semicanonical basis, and Kashiwara and Saito's  construction of $B(-\infty)$.

\subsection{Main results}
We now describe our results in more detail.
Let $C \in M_n(\Z)$ be a symmetrizable generalized Cartan matrix,
and let $D$ be a symmetrizer of $C$. 
Let $\Pi = \Pi(C,D)$ be the associated preprojective algebra as defined
in \cite{GLS3}.
We assume throughout that our ground field $K$ is algebraically closed.
For $d \in \N^n$,
let $\nil_E(\Pi,d)$ be the variety of $E$-filtered $\Pi$-modules with dimension 
vector $d$.

Let $G(d)$ be the product of linear groups, which acts on $\nil_E(\Pi,d)$
by conjugation.
For $d = (d_1,\ldots,d_n)$ and $D = \diag(c_1,\ldots,c_n)$ define
$d/D := (d_1/c_1,\ldots,d_n/c_n)$.
Let $q_{DC}$ be the quadratic form associated with $1/2DC$.

\begin{Thm}\label{introthm:dim}
For each irreducible component $Z$ of $\nil_E(\Pi,d)$ we have
$$
\dim(Z) \le \dim G(d) - q_{DC}(d/D).
$$
\end{Thm}

Let $\Irr(\nil_E(\Pi,d))^\mx$ be the set of irreducible components 
of $\nil_E(\Pi,d)$ of maximal dimension $\dim G(d) - q_{DC}(d/D)$. 

Assume that $C$ is symmetric and $D$ is the identity matrix.
Then $\Pi$ is a classical
preprojective algebra associated with a quiver $Q$, the $\nil_E(\Pi,d)$
are Lusztig's nilpotent varieties, $\dim G(d) - q_{DC}(d/D)$ is the dimension 
of the affine space of
representations of the quiver $Q$ with dimension vector $d$,
and all irreducible components of $\nil_E(\Pi,d)$ have the same dimension
$\dim G(d) - q_{DC}(d/D)$.

Let $\n(C)$ be the positive part of the symmetrizable Kac-Moody algebra
$\g(C)$ associated with $C$.
Let $B(-\infty)$ be the crystal graph of the quantized enveloping algebra
$U_q(\n(C))$.

The following theorem is our first main result.

\begin{Thm}\label{introthm:crystal}
Let $\Pi = \Pi(C,D)$, and set
$$
\cB := \bigsqcup_{d \in \N^n} \Irr(\nil_E(\Pi,d))^\mx.
$$
Then there are isomorphisms of crystals
$$
(\cB,\wt,\tilde{e}_i,\tilde{f}_i,\vph_i,\vep_i) \cong
(\cB,\wt,\tilde{e}_i^*,\tilde{f}_i^*,\vph_i^*,\vep_i^*) \cong B(-\infty).
$$
\end{Thm}

The operators and maps $\wt,\tilde{e}_i,\tilde{f}_i,\vph_i,\vep_i$ (and their $*$-versions)
appearing in Theorem~\ref{introthm:crystal}
are defined in a module theoretic way in the fashion of
Kashiwara and Saito's \cite{KS}
geometric realization of $B(-\infty)$, see also Nandakumar and Tingley \cite{NT}.
Kashiwara and Saito only work with symmetric Kac-Moody algebras ($C$
symmetric and $D$ the identity matrix), and Nandakumar and Tingley need
to work over fields which are not algebraically closed in case $C$ is
non-symmetric.
For $C$ symmetric and $D$ the identity matrix, Theorem~\ref{introthm:crystal}
coincides with Kashiwara and Saito's result.

For $K = \C$ the field of complex numbers,
let $\tF(\Pi)$ be the convolution algebra of constructible functions
on the representation varieties $\rep(\Pi,d)$, and let
$$
\tM(\Pi) = \bigoplus_{d \in \N^n} \tM(\Pi)_d
$$ 
be the subalgebra generated by the characteristic
functions $\{ \ttheta_i := 1_{E_i} \mid 1 \le i \le n \}$.
(Here $E_i$ is a free $K[X]/(X^{c_i})$-module of rank $1$, which
can be seen as a $\Pi$-module in a natural way.)
We assume that all constructible functions are constant on $G(d)$-orbits.
The elements in $\tM(\Pi)_d$ are constructible functions
$\nil_E(\Pi,d) \to \C$.
In general, the functions $\ttheta_i$ do not satisfy the Serre relations.
For a constructible function $f\df \nil_E(\Pi,d) \to \C$ and
an irreducible component $Z$ of $\nil_E(\Pi,d)$ let $\rho_Z(f)$
be the generic value of $f$ on $Z$.

\begin{Thm}\label{introthm:semican}
For $K = \C$ and $\Pi = \Pi(C,D)$,
the convolution algebra $\tM(\Pi)$ contains 
a set 
$$
\wti{\cS} := \{ \wti{f}_Z \mid Z \in \cB \}
$$ 
of constructible functions 
such that for each $Z' \in \cB$ we have
$$
\rho_{Z'}(\wti{f}_Z) =
\begin{cases}
1 & \text{if $Z = Z'$},\\
0 & \text{otherwise}.
\end{cases}
$$
\end{Thm}

Define 
$$
\cM(\Pi) := \tM(\Pi)/\Serre
$$
where $\Serre$ is the ideal generated by the Serre relations
$\{ \ttheta_{ij} \mid 1 \le i,j \le n \text{ with } c_{ij} \le 0 \}$
where
$$
\ttheta_{ij} := \ad(\ttheta_i)^{1-c_{ij}}(\ttheta_j).
$$
Let 
$$
\theta_i := \ttheta_i + \Serre
\text{\;\;\; and \;\;\;}
f_Z := \wti{f}_Z + \Serre
$$
be the residue classes of $\ttheta_i$ and $\wti{f}_Z$ in $\cM(\Pi)$.

For a constructible function $f\df \nil_E(\Pi,d) \to \C$ let 
$$
\supp(f) := \{ M \in \nil_E(\Pi,d) \mid f(M) \not= 0 \}
$$
be the \emph{support} of $f$.
By Theorem~\ref{introthm:dim} we have 
$\dim \supp(f) \le \dim G(d) - q_{DC}(d/D)$.

\begin{Conj}\label{mainconj}
Let $K = \C$ and $\Pi = \Pi(C,D)$.
For $0 \not= f \in \tM(\Pi)_d \cap \Serre$ we have 
$$
\dim \supp(f) < \dim G(d) - q_{DC}(d/D).
$$
\end{Conj}

The conjecture above is supported
by Corollary~\ref{cor:strictly4}.
The examples discussed in
Section~\ref{subsub:nonmaxsupp} illustrate certain subtleties.

The next theorem is our second main result.

\begin{Thm}\label{introthm:main}
Let $K = \C$, $\Pi = \Pi(C,D)$ and $\n = \n(C)$.
Assume that Conjecture~\ref{mainconj} is true.
Then the following hold:
\begin{itemize}

\item[(i)]
There is a Hopf algebra isomorphism
$$
\eta_\Pi\df U(\n) \to \cM(\Pi)
$$
defined by $e_i \mapsto \theta_i$.

\item[(ii)]
Via the isomorphism $\eta_\Pi$, the set
$$
\cS := \{ f_Z \mid Z \in \cB \}
$$
is a $\C$-basis of $U(\n)$.

\item[(iii)]
For $0 \not= f \in \tM(\Pi)_d$ the following are equivalent:
\begin{itemize}

\item[(a)]
$f \in \Serre$;

\item[(b)]
$\dim \supp(f) < \dim G(d) - q_{DC}(d/D)$.

\end{itemize}

\end{itemize}
\end{Thm}

We have $\Serre \not= 0$ if and only if 
$c_{ij} < 0$ and $c_i \ge 2$ for some $1 \le i,j \le n$.

One should expect that $\cS$ (seen as a subset of $U(\n)$ via
$\eta_\Pi$) does not depend on the symmetrizer $D$.

Suppose that $C$ is symmetric and that $D$ is the identity matrix.
Then $\tM(\Pi) = \cM(\Pi)$ and the Hopf algebra isomorphism $U(\n) \to \cM(\Pi)$
can be obtained by combining \cite[Lemma~12.11]{L1} with either \cite{KS} 
or \cite{S}, see \cite{L2}.
Furthermore, $\wti{\cS} = \cS$ is exactly Lusztig's \cite{L2}
semicanonical basis.

\subsection{}
The paper is organized as follows.
In 
Section~\ref{sec:recallpreproj}
we recall definitions and results on
preprojective algebras and their representation varieties.
In 
Section~\ref{sec:bundles}
we generalize Lusztig's construction of certain fibre bundles
from the classical nilpotent varieties to our more general setup.
The proof of Theorem~\ref{introthm:dim} is contained in
Section~\ref{sec:maxcomp}.
We also show that generically the modules in maximal irreducible components
are crystal modules.
(These modules are defined in Section~\ref{subsec:crystalmodules}.)
Section~\ref{sec:crystal} contains the proof of Theorem~\ref{introthm:crystal}.
The convolution algebra $\cM(\Pi)$ is defined in 
Section~\ref{sec:convolution}.
Section~\ref{sec:semican} contains the proof of 
Theorems~\ref{introthm:semican} and \ref{introthm:main}.
Assuming that Conjecture~\ref{mainconj} is true,
we also show that 
the semicanonical bases of the enveloping algebras $U(\n)$ induce semicanonical bases of all irreducible integrable highest weight modules.
Section~\ref{sec:examples} contains the classification
of maximal irreducible components for the Dynkin cases, and
also
examples of Dynkin type $A_2$, $B_2$ and $G_2$.

\subsection{Notation}
By a \emph{module} we mean a finite-dimensional left module, unless mentioned otherwise.
For maps $f\df X \to Y$ and $g\df Y \to Z$ the composition is denoted 
by $gf\df X \to Z$.
A module $M$ over an algebra $A$ is \emph{rigid} if $\Ext_A^1(M,M) = 0$.
For a module $M$, let $M^m$ be the direct sum of $m$ copies of $M$.

For a constructible subset $X$ of a quasi-projective variety, 
let $\Irr(X)$ be the set of irreducible components of $X$.

Let $\N$ be the natural numbers, including $0$.


\section{Quivers with relations associated with symmetrizable Cartan matrices}\label{sec:recallpreproj}


In this section, we recall some definitions and results from \cite{GLS3}.

\subsection{The preprojective algebras $\Pi(C,D)$}\label{subsec:DefPi}
A matrix $C = (c_{ij}) \in M_n(\Z)$ is a  
\emph{symmetrizable generalized Cartan matrix}
provided the following hold:
\begin{itemize}

\item[(C1)]
$c_{ii} = 2$ for all $i$;

\item[(C2)] 
$c_{ij} \le 0$ for all $i \not= j$;

\item[(C3)]
$c_{ij} \not= 0$ if and only if $c_{ji} \not= 0$;

\item[(C4)]
There is a diagonal integer matrix $D = \diag(c_1,\ldots,c_n)$ with
$c_i \ge 1$ for all $i$ such that
$DC$ is symmetric. 

\end{itemize}
The matrix $D$ appearing in (C4) is called a \emph{symmetrizer} of $C$.
The symmetrizer $D$ is \emph{minimal} if $c_1 + \cdots + c_n$ is
minimal.

From now on, let $C = (c_{ij}) \in M_n(\Z)$ be a symmetrizable generalized Cartan matrix.
Throughout, let
$$
I := \{ 1,\ldots,n \}.
$$

An \emph{orientation of} $C$ is a subset 
$\Omega \subset  I \times I$
such that for all $(i,j) \in I \times I$ the following are equivalent:
\begin{itemize}

\item[(i)]
$\{ (i,j),(j,i) \} \cap \Omega \not= \varnothing$;

\item[(ii)]
$|\{ (i,j),(j,i) \} \cap \Omega| = 1$;

\item[(iii)]
$c_{ij}<0$.

\end{itemize}
The \emph{opposite orientation} of an orientation $\Omega$ is defined as
$\Omega^* := \{ (j,i) \mid (i,j) \in \Omega \}$.
Let
$\overline{\Omega} := \Omega \cup \Omega^*$.
Define
$$
\ov{\Omega}(i) := \{ j \in I \mid (i,j) \in \ov{\Omega} \}
= \{ j \in I \mid (j,i) \in \overline{\Omega} \}
= \{ j \in I \mid c_{ij} < 0 \}.
$$
For $(i,j) \in \overline{\Omega}$ set
$$
\sgn(i,j) :=
\begin{cases}
1 & \text{if $(i,j) \in \Omega$},\\
-1 & \text{if $(i,j) \in \Omega^*$}.
\end{cases}
$$

For all $c_{ij} < 0$ define
\[
g_{ij} := |\gcd(c_{ij},c_{ji})|,\qquad
f_{ij} := |c_{ij}|/g_{ij}.
\]
Let
$\ov{Q} := \ov{Q}(C) := (I,\ov{Q}_1,s,t)$ be the quiver with the
set of vertices $I = \{ 1,\ldots,n \}$ and 
with the set of arrows 
$$
\ov{Q}_1 := \{ \alpha_{ij}^{(g)}\df j \to i \mid (i,j) \in \ov{\Omega}, 1 \le g \le g_{ij} \}
\cup \{ \vep_i\df i \to i \mid i \in I \}.
$$
(Thus we have $s(\alpha_{ij}^{(g)}) = j$ and $t(\alpha_{ij}^{(g)}) = i$ and
$s(\vep_i) = t(\vep_i) = i$, where $s(a)$ and $t(a)$ denote the
starting and terminal vertex of an arrow $a$, respectively.)
If $g_{ij}=1$, we also write $\alpha_{ij}$ instead of $\alpha_{ij}^{(1)}$.

For $\ov{Q} = \ov{Q}(C)$ and a symmetrizer $D = \diag(c_1,\ldots,c_n)$ of $C$, we define an algebra
$$
\Pi := \Pi(C,D,\Omega) := K\overline{Q}/\overline{I}
$$ 
where $K\ov{Q}$ is the path algebra of $\ov{Q}$ and
$\overline{I}$ is the ideal defined by the following
relations:
\begin{itemize}

\item[(P1)]
For each $i$ we have
\[
\vep_i^{c_i} = 0.
\]

\item[(P2)]
For each $(i,j) \in \overline{\Omega}$ and each $1 \le g \le g_{ij}$ we have
\[
\vep_i^{f_{ji}}\alpha_{ij}^{(g)} = \alpha_{ij}^{(g)}\vep_j^{f_{ij}}.
\]

\item[(P3)]
For each $i$ we have
\[
\sum_{j\in \ov{\Omega}(i)} \sum_{g=1}^{g_{ji}} \sum_{f=0}^{f_{ji}-1} \sgn(i,j)
\vep_i^f\alpha_{ij}^{(g)} \alpha_{ji}^{(g)} \vep_i^{f_{ji}-1-f} = 0.
\]

\end{itemize}
We call $\Pi$ a \emph{preprojective algebra} of type $C$.
These algebras generalize the classical preprojective algebras associated
with quivers, see \cite{GLS3} for details.
Up to isomorphism,
the algebra $\Pi := \Pi(C,D) := \Pi(C,D,\Omega)$ does not depend on
the orientation $\Omega$ of $C$.
Let $\rep(\Pi)$ be the category of finite-dimensional $\Pi$-modules.

Define bilinear forms 
$$
\bil{-,-}\df \Z^n \times \Z^n \to \Z
$$ 
by $\bil{\alpha_i,\alpha_j} := c_{ji}$, 
and 
$$
(-,-)\df \Z^n \times \Z^n \to \Z
$$ 
by $(\alpha_i,\alpha_j) := c_ic_{ij}$.
(Here $\alpha_1,\cdots,\alpha_n$ denotes the standard basis of $\Z^n$.)
Let 
$$
q_{DC}\df \Z^n \to \Z
$$
be the associated quadratic form defined by
$q_{DC}(x) := (x,x)/2$.

For $i \in I$ let
$S_i$ be the $1$-dimensional simple $\Pi$-module
associated with the vertex $i$, and let
$E_i$ be the $c_i$-dimensional uniserial $\Pi$-module 
associated with $i$.
Let 
$$
H_i := K[\vep_i]/(\vep_i^{c_i}),
$$ 
and let
$e_i \in \Pi$ be the idempotent associated with $i$.
For each $\Pi$-module $M$ the space $e_iM$ is naturally an
$H_i$-module.
We have $E_i = e_iE_i$, and $e_iE_i$ is free of rank $1$ as an $H_i$-module.
A $\Pi$-module $M$ is \emph{locally free} if $e_iM$ is a free
$H_i$-module for all $i$.
The rank of a free $H_i$-module $M_i$ is denoted by $\rk(M_i)$.
For a locally free $\Pi$-module $M$ let $\rk(M) := (\rk(e_1M),\ldots,\rk(e_nM))$
be the \emph{rank vector} of $M$.
A $\Pi$-module $M$ is \emph{$E$-filtered} (resp. $S$-filtered)
if there exists a chain
$$
0 = U_0 \subseteq U_1 \subseteq \cdots \subseteq U_t = M
$$
of submodules $U_i$ of $M$ such that for each $1 \le k \le t$ we have 
$U_k/U_{k-1} \cong E_{i_k}$ (resp. $U_k/U_{k-1} \cong S_{i_k}$) for some 
$i_k \in I$.
Let $\nil_E(\Pi) \subseteq \rep(\Pi)$ be the subcategory of
$E$-filtered $\Pi$-modules. 
Note that each $E$-filtered $\Pi$-module is locally free. 
The converse of this statement is in general wrong.
We refer to \cite{GLS3} for further details.

\subsection{Representation varieties (quiver version)}
Let $\Pi = \Pi(C,D)$.
For a dimension vector $d = (d_1,\ldots,d_n)$  
let
$$
\ov{H}(d) :=
\prod_{a \in \ov{Q}_1} 
\Hom_K(K^{d_{s(a)}},K^{d_{t(a)}}),
$$
and let $\rep(\Pi,d)$ be the varieties of
$\Pi$-modules with dimension vector $d$.
By definition, the points in $\rep(\Pi,d)$ are
tuples
$$
(M(a))_a \in \ov{H}(d)
$$
satisfying the equations
\begin{align*}
M(\vep_i)^{c_i} &= 0, &
M(\vep_i)^{f_{ji}}M(\alpha_{ij}^{(g)}) &= M(\alpha_{ij}^{(g)})M(\vep_j)^{f_{ij}},
\end{align*}
\begin{align*}
\sum_{j\in \ov{\Omega}(i)} \sum_{g=1}^{g_{ji}} \sum_{f=0}^{f_{ji}-1} \sgn(i,j)
M(\vep_i)^fM(\alpha_{ij}^{(g)})M(\alpha_{ji}^{(g)})M(\vep_i)^{f_{ji}-1-f} &= 0
\end{align*}
for all $i \in I$, $(i,j) \in \ov{\Omega}$ and $1 \le g \le g_{ij}$.
The group
$$
G(d) := \prod_{i \in I} \GL_K(d_i)
$$
acts on $\rep(\Pi,d)$ by conjugation.
For a module $M \in \rep(\Pi,d)$ 
let $\cO(M) := G(d)M$ be its $G(d)$-orbit.
The $G(d)$-orbits are in bijection with the isomorphism classes of modules
in $\rep(\Pi,d)$.
For $M \in \rep(\Pi,d)$ we have
$$
\dim \cO(M) = \dim G(d) - \dim \End_\Pi(M).
$$
Let $\rep_\vp(\Pi,d) \subseteq \rep(\Pi,d)$ be the subvarieties of 
locally free modules, and let
$\nil_E(\Pi,d) \subseteq \rep_\vp(\Pi,d)$ be the subset of $E$-filtered $\Pi$-modules.
Using the same technique as in the proof of \cite[Theorem~1.3(i)]{CBS},
one shows that $\nil_E(\Pi,d)$ is a constructible subset of
$\rep_\vp(\Pi,d)$.

\subsection{Representation varieties (species version)}\label{subsec:species}
Let $\Pi = \Pi(C,D) = \Pi(C,D,\Omega)$.
For a tuple $\bM = (M_1,\ldots,M_n)$ with $M_i \in \rep(H_i)$
let
$$
H(\bM) := \prod_{(i,j) \in \Omega} 
\Hom_{H_i}({_i}H_j \otimes_j M_j,M_i)
$$
and
$$
\ov{H}(\bM) := \prod_{(i,j) \in \ov{\Omega}} 
\Hom_{H_i}({_i}H_j \otimes_j M_j,M_i).
$$
Here ${_i}H_j$ are the $H_i$-$H_j$-bimodules defined in \cite{GLS3}.
Using the results in \cite{GLS3} we see that $(\dim \ov{H}(\bM))/2 = 
\dim H(\bM)$.

For $M \in \ov{H}(\bM)$ let 
$$
M_{ij}\df {_iH_j} \otimes_j M_j \to M_i
$$ 
be the corresponding homomorphisms in $\Hom_{H_i}({_iH_j} \otimes_j M_j,M_i)$.

Let 
$$
G(\bM) := \prod_{i \in I} \GL_{H_i}(M_i)
$$
where $\GL_{H_i}(M_i)$ is the group of $H_i$-linear
automorphisms of $M_i$.
The group $G(\bM)$ acts by conjugation on $\ov{H}(\bM)$.
We call $\bM$ \emph{locally free} if each $M_i$ is a free $H_i$-module.
In this case, let
$$
\rk(\bM) := (\rk(M_1),\ldots,\rk(M_n))
$$
be the \emph{rank vector} of $\bM$.
The \emph{total rank} of $\bM$ is defined as $\rk(M_1)+ \cdots + \rk(M_n)$.

For $M \in \ov{H}(\bM)$ let
$$
M_{i,\inn} := (\sgn(i,j)M_{ij})_j\df \bigoplus_{j \in \ov{\Omega}(i)}
{_i}H_j \otimes_j M_j \to M_i
$$
and
$$
M_{i,\out} := (M_{ji}^\vee)_j\df 
M_i \to \bigoplus_{j \in \ov{\Omega}(i)} {_i}H_j \otimes_j M_j
$$
be defined as in \cite[Section~5]{GLS3}.
Let
$$
\rep(\Pi,\bM) := \{ M \in \ov{H}(\bM) \mid 
M_{i,\inn} \circ M_{i,\out} = 0 \}.
$$ 
We can see $\rep(\Pi,\bM)$ as the affine variety of $\Pi$-modules
$M$ with $e_iM = M_i$ for $i \in I$.
The $G(\bM)$-action on $\ov{H}(\bM)$ restricts to $\rep(\Pi,\bM)$.
The isomorphism classes of $\Pi$-modules $M$ with $e_iM = M_i$ for all
$i$ are
in bijection with the $G(\bM)$-orbits in $\rep(\Pi,\bM)$.
For a module $M \in \rep(\Pi,\bM)$ 
let 
$\cO(M) := G(\bM)M$ 
be its $G(\bM)$-orbit.
For $M \in \rep(\Pi,\bM)$ we have
$$
\dim \cO(M) = \dim G(\bM) - \dim \End_\Pi(M).
$$
For $\bM$ locally free, let
$$
\nil_E(\Pi,\bM) := \Pi(\bM)
$$ 
be the subset of $E$-filtered modules in $\rep(\Pi,\bM)$.
This is a constructible subset of $\rep(\Pi,\bM)$.

For a rank vector $\br = (r_1,\ldots,r_n)$ define
$\bM(\br) := (H_1^{r_1},\ldots,H_n^{r_n})$ and set
\begin{align*}
H(\br) &:= H(\bM(\br)),
\\
\ov{H}(\br) &:= \ov{H}(\bM(\br)).
\\
\rep(\Pi,\br) &:= \rep(\Pi,\bM(\br)),
\\
\nil_E(\Pi,\br) &:= \Pi(\br) := \Pi(\bM(\br)).
\end{align*}
Set $G(\br) := G(\bM(\br))$.
(We always denote rank vectors in bold letters, like $\br$, and
dimension vectors in ordinary letters, like $d$.)

Obviously, each variety $\Pi(\bM)$ is isomorphic to
$\Pi(\br)$ where $\br = \rk(\bM)$.
We sometimes just identify $\Pi(\bM)$ and
$\Pi(\br)$.

\subsection{Relating the quiver version and the species version}
We have the obvious projection
$$
\rep(\Pi,d) \xrightarrow{\varepsilon_\Pi} \prod_{i \in I} \rep(H_i,d_i). 
$$
For $\bM = (M_1,\ldots,M_n) \in \prod_{i \in I} \rep(H_i,d_i)$
we have
$$
\varepsilon_\Pi^{-1}(\bM) \cong \rep(\Pi,\bM).
$$
This follows from the considerations in \cite[Section~5]{GLS3}.
We see that $\vep_\Pi$ is a fibre bundle.
We identify the fibre $\varepsilon_\Pi^{-1}(\bM)$ with $\rep(\Pi,\bM)$.
For $M \in \rep(\Pi,\bM)$ we have
$$
G(\bM)M = G(d)M \cap \rep(\Pi,\bM).
$$
Assume now that $\bM$ is locally free.
Recall that $d/D = (d_1/c_1,\ldots,d_n/c_n)$, and note 
that $\rk(\bM) = d/D$.
An easy calculation shows that
$$
\dim H(\bM) = \dim G(\bM) - q_{DC}(d/D).
$$
For a closed $G(\bM)$-stable subset $Z$ of $\rep(\Pi,\bM)$ of
dimension $\dim G(\bM) + m$ for some $m \in \Z$, the correspond
subset $G(d)Z$ of $\rep(\Pi,d)$ has dimension 
$\dim G(d) +m$.

\subsection{Convolution algebras}\label{subsec:convolution}
In this section, assume that $K = \C$.
Let 
$$
\tF(\Pi) := \bigoplus_{d \in \N^n} \tF(\Pi)_d
$$ 
be the convolution algebra associated with $\Pi$, where
$\tF(\Pi)_d$ is the $\C$-vector space of constructible
functions $\rep(\Pi,d) \to \C$.
Recall that a map 
$$
f\df \rep(\Pi,d) \to \C
$$ 
is a \emph{constructible function} if the following hold:
\begin{itemize}

\item[(i)]
$\Ima(f)$ is finite;

\item[(ii)]
For each $m \in \C$,
the preimage
$f^{-1}(m)$ is a constructible subset of $\rep(\Pi,d)$;

\item[(iii)]
$f$ is constant on $G(d)$-orbits.

\end{itemize}
For $M \in \rep(\Pi)$ define $1_M \in \tF(\Pi)$ by
$$
1_M(N) :=
\begin{cases}
1 & \text{if $M \cong N$},
\\
0 & \text{otherwise}.
\end{cases}
$$

For $f,g \in \tF(\Pi)$ the product $f * g$ is defined by
$$
(f*g)(M) := \sum_{m \in \C} m \chi(\{ U \subseteq M \mid f(U)g(M/U) = m \})
$$
where $M \in \rep(\Pi)$ and $\chi$ denotes the topological Euler characteristic. 

For $i \in I$ let 
$$
\ttheta_i := 1_{E_i}.
$$
Let 
$$
\tM(\Pi) = \bigoplus_{d \in \N^n} \tM(\Pi)_d
$$
be the subalgebra of $\tF(\Pi)$ generated by $\{ \ttheta_i \mid i \in I \}$,
where
$$
\tM(\Pi)_d := \tF(\Pi)_d \cap \tM(\Pi).
$$
For $f \in \tM(\Pi)_d$ let 
$$
\supp(f) := \{ M \in \rep(\Pi,d) \mid f(M) \not= 0 \}
$$
be the \emph{support} of $f$.
We have $\supp(f) \subseteq \rep_\vp(\Pi,d)$.
Using the same arguments as in the proof of 
\cite[Proposition~4.7]{GLS4}, we get that $\tM(\Pi)$
is a Hopf algebra, which is isomorphic to the
enveloping algebra $U(\cP(\tM(\Pi)))$ of the Lie algebra
$\cP(\tM(\Pi))$ of primitive elements in $\tM(\Pi)$.
A constructible function $f \in \tM(\Pi)_d$ is in $\cP(\tM(\Pi))$
if and only if $\supp(f)$ consists just of indecomposable modules.
The comultiplication in $\tM(\Pi)$ is given by
$\ttheta_i \mapsto \ttheta_i \otimes 1 + 1 \otimes \ttheta_i$.

For a dimension vector $d$ with $\rep_\vp(\Pi,d) \not= \varnothing$ let
$\br := d/D$ be the associated rank vector.
Alternatively, we can define $\tM(\Pi)$ 
using constructible
functions 
$\rep(\Pi,\br) \to \C$.
(Condition (iii) in the definition of a constructible function is
replaced by demanding that $f$ is constant on $G(\br)$-orbits.)
It is straightforward to check that the two definitions yield canonically isomorphic algebras.
In this article, 
we mainly work with the varieties $\rep(\Pi,\br)$ (the species version) instead of the varieties $\rep(\Pi,d)$ (the quiver version).
Our main results in the introduction 
and also the examples collection in Section~\ref{sec:examples}
are formulated using the quiver version, whereas the rest of the article
(especially the proofs) are based on the more convenient species 
version.

\subsection{Hom-Ext formulas}\label{subsec:3.1}
The following result is proved in \cite[Theorem~12.6]{GLS3}.
It generalizes \cite[Lemma~1]{CB2}.

\begin{Lem}\label{lem:3.7a}
For $M,N\in \rep_\vp(\Pi)$ the following hold:
\begin{itemize}

\item[(i)]
$\dim \Ext_\Pi^1(M,N) = \dim \Ext_\Pi^1(N,M)$;

\item[(ii)]
$
\dim \Ext_\Pi^1(M,N) = \dim \Hom_\Pi(M,N) + \dim \Hom_\Pi(N,M)
- (\rk(M),\rk(N))$.

\end{itemize}
\end{Lem}

\begin{Cor}\label{cor:3.8}
For $M \in \rep_\vp(\Pi)$ and $i \in I$ we have
$$
\dim \Ext_\Pi^1(E_i,M) = 
\dim \Hom_\Pi(E_i,M) + \dim \Hom_\Pi(M,E_i) - 
c_i\bil{\rk(M),\alpha_i}. 
$$
\end{Cor}

\begin{proof}
Let $\rk(M) = (m_1,\ldots,m_n)$.
We have 
$$
\bil{\rk(M),\alpha_i} = 2m_i + \sum_{j \in \ov{\Omega}(i)} c_{ij}m_j
$$
and
$$
(\rk(M),\alpha_i) = 2c_im_i + \sum_{j \in \ov{\Omega}(i)} c_ic_{ij}m_j.
$$
Now the result follows from Lemma~\ref{lem:3.7a}(ii).
\end{proof}

Let $M \in \rep(\Pi)$.
For each $i \in I$ let
$$
\wti{M}_i := \bigoplus_{j \in \ov{\Omega}(i)} {_i}H_j \otimes_j M_j.
$$
As before, let
$$
M_{i,\inn} := (\sgn(i,j)M_{ij})_j\df \bigoplus_{j \in \ov{\Omega}(i)}
{_i}H_j \otimes_j M_j \to M_i
$$
and
$$
M_{i,\out} := (M_{ji}^\vee)_j\df 
M_i \to \bigoplus_{j \in \ov{\Omega}(i)} {_i}H_j \otimes_j M_j.
$$

For $M \in \rep(\Pi)$ and $i \in I$ let $\sub_i(M)$ (resp.
$\fac_i(M)$) be the largest submodule (resp. factor module) of
$M$ such that each composition factor of $\sub_i(M)$ (resp. 
$\fac_i(M)$) is isomorphic to $S_i$.

Let $\tp(M)$ be the largest semisimple factor module of $M$,
and let $\tp_i(M)$ be the largest semisimple factor module of $M$
such that each composition factor of $\tp_i(M)$ is isomorphic to $S_i$.

\begin{Lem}\label{lem:3.7}
For $M \in \rep(\Pi)$ the following hold:
\begin{itemize}

\item[(i)]
$\dim \Hom_\Pi(E_i,M) = \dim \Ker(M_{i,\out}) = \dim \sub_i(M)$;

\item[(ii)]
$\dim \Hom_\Pi(M,E_i) = \dim(\Coker(M_{i,\inn})) = \dim \fac_i(M)$;

\item[(iii)]
If $M$ is locally free, then
\begin{align*}
\dim \Ext_\Pi^1(M,E_i) &= \dim(\wti{M}_i) - 
\dim \Ima(M_{i,\inn}) - \dim \Ima(M_{i,\out})\\
&= 
\dim(\Ker(M_{i,\inn})/\Ima(M_{i,\out})).
\end{align*}

\end{itemize}
\end{Lem}

\begin{proof}
We have $\Ker(M_{i,\out}) = \sub_i(M)$ and $\Coker(M_{i,\out}) = \fac_i(M)$.
The $H_i$-module $H_i = E_i$ is indecomposable projective-injective
in $\rep(H_i)$.
Thus $\dim \Hom_\Pi(E_i,M)$ and $\dim \Hom_\Pi(M,E_i)$ are the dimensions
of $\sub_i(M)$ and $\fac_i(M)$, respectively.
This proves (i) and (ii).

For $M \in \rep(\Pi)$ we have
a sequence 
$$
0 \to \sub_i(M) \xrightarrow{\iota} M_i \xrightarrow{M_{i,\out}}
\wti{M}_i \xrightarrow{M_{i,\inn}} M_i \xrightarrow{\pi} \fac_i(M) \to 0
$$
of $H_i$-linear maps, where
$\iota$ is a monomorphism, $\pi$ is an epimorphism, 
$\Ima(\iota) = \Ker(M_{i,\out})$,
$\Ima(M_{i,\inn}) = \Ker(\pi)$ and $\Ima(M_{i,\out}) \subseteq 
\Ker(M_{i,\inn})$.
Observe that
$$
\dim(\wti{M}_i) - 
\dim \Ima(M_{i,\inn}) - \dim \Ima(M_{i,\out})
= \dim(\Ker(M_{i,\inn})/\Ima(M_{i,\out})).
$$
We have 
\begin{align*}
\dim \Ima(M_{i,\out}) &:= \dim(M_i) - \dim(\sub_i(M)),
\\
\dim \Ima(M_{i,\inn}) &:= \dim(M_i) - \dim(\fac_i(M)).
\end{align*}
It follows that 
$$
\dim(\Ker(M_{i,\inn})/\Ima(M_{i,\out}))  = 
\dim(\wti{M}_i) - 2 \dim(M_i) + \dim \sub_i(M) + \dim \fac_i(M).
$$
We know that 
$$
\dim \wti{M}_i = \sum_{j \in \ov{\Omega}(i)} c_j|c_{ji}|m_j
$$
where $(m_1,\ldots,m_n) = \rk(M)$.
Here we used that 
$$
{_i}H_j \otimes_j M_j \cong H_j^{|c_{ji}|m_j}.
$$
We have
$$
(\rk(M),\rk(E_i)) = (\rk(M),\alpha_i) = 2\dim(M_i) - \dim(\wti{M}_i).
$$
Combining the above equalities with (i) and (ii) and with
Lemma~\ref{lem:3.7a}(ii) we get the formula (iii).
􏰭\end{proof}


\section{Lusztig's bundle construction}\label{sec:bundles}


\subsection{Partitions and $H_k$-modules}
For $m \ge 0$ let $\cP_m$ be the set of partitions 
with entries bounded by $m$.
(These are tuples $\bp = (p_1,\ldots,p_t)$ of integers with
$m \ge p_1 \ge \cdots \ge p_t \ge 0$. 
We identify $(p_1,\ldots,p_t,0,\ldots,0)$ with $(p_1,\ldots,p_t)$.)
For a partition $\bp = (p_1,\ldots,p_t)$ and $k \ge 0$ let 
$\bp(k) := |\{ 1 \le i \le t \mid p_i = k \}|$
be the number of entries equal to $k$, and let
$\length(\bp) :=  |\{ 1 \le i \le t \mid p_i \not= 0 \}|$.
For $p,m \ge 0$ we also write $(p^m) = (p,\ldots,p)$ for
the partition with $m$ entries equal to $p$.
Similarly, for a partition $(p_1,\ldots,p_t)$ and $m_1,\ldots,m_t \ge 0$
we define
$(p_1^{m_1},\ldots,p_t^{m_t})$ in the obvious way.

For $k \in I$ the isomorphism classes of finite-dimensional 
$H_k$-modules can be parametrized by $\cP_{c_k}$ in the obvious way.
For $\bp \in \cP_{c_k}$,
let $H_k^\bp$ be an $H_k$-module corresponding to $\bp$.
Vice versa, for $M \in \rep(H_k)$ let $\bp(M) \in \cP_{c_k}$ be the partition
associated with $M$.

\subsection{Stratifications of $\Pi(\bM)$}\label{subsec:strata}
Let $\Pi = \Pi(C,D)$ and let 
$\bM = (M_1,\ldots,M_n)$ be locally free.

Recall that for
$M \in \rep(\Pi)$, $\fac_k(M)$ 
is the largest factor module $M/U$ 
such that each composition factor of $M/U$  
is isomorphic to $S_k$. 
Similarly, $\sub_k(M)$ is the largest submodule $U \subseteq M$ 
such that each composition factor of $U$ 
is isomorphic to $S_k$. 

Recall that
$$
\Pi(\bM) = \nil_E(\Pi,\bM).
$$
For $\bp \in \cP_{c_k}$ let 
$$
\Pi(\bM)^{k,\bp} := \{ M \in \Pi(\bM) \mid \fac_k(M) \cong H_k^\bp \}
$$
and
$$
\Pi(\bM)_{k,\bp} := \{ M \in \Pi(\bM) \mid \sub_k(M) \cong H_k^\bp \}.
$$
For the special case $\bp = (c_k^p)$ we define
$$
\Pi(\bM)^{k,p} := \Pi(\bM)^{k,\bp}
\text{\;\;\; and \;\;\;}
\Pi(\bM)_{k,p} := \Pi(\bM)_{k,\bp}.
$$

In the following, we prove some results involving the varieties 
$\Pi(\bM)^{k,\bp}$.
We leave it as an easy exercise to formulate and prove the 
corresponding dual results for $\Pi(\bM)_{k,\bp}$.

\begin{Lem}\label{lusztig0}
The following hold:
\begin{itemize}

\item[(a)]
$\Pi(\bM)^{k,\bp}$ is a locally closed $G(\bM)$-stable subvariety of $\Pi(\bM)$. 

\item[(b)]
$\Pi(\bM)^{k,0}$ is open in $\Pi(\bM)$.

\item[(c)]
For $\bM \not= 0$ we have
$$
\Pi(\bM) = \bigcup_{\substack{k \in I \\\bp \in \cP_{c_k}\\\bp(c_k) \ge 1}} 
\Pi(\bM)^{k,\bp}.
$$

\end{itemize}
\end{Lem}

\begin{proof}
(a): 
For $1 \le i \le c_k$, recall that $H_k^{(i)}$ denotes the uniserial $H_k$-module
of length $i$. 
For $M \in \Pi(\bM)$ the numbers
$$
\dim \Hom_\Pi(M,H_k^{(i)})
$$
with $1 \le i \le c_k$ 
determine $\fac_k(M)$.
It follows that $\Pi(\bM)^{k,\bp}$ is a finite intersection of locally closed
sets.
This yields the result.

(b): This follows directly from the upper semicontinuity of
the map $\dim \Hom_H(M,-)$.

(c): By definition each non-zero $M \in \Pi(\bM)$ has a chain
$$
0 = U_0 \subset U_1 \subset \cdots \subset U_t = M
$$
such that for $1 \le k \le t$ we have 
$U_j/U_{j-1} \cong E_{i_j}$ for some $i_j \in I$.
Wit $k := i_t$ we get
$$
M \in \Pi(\bM)^{k,\bp}
$$
where $\bp$ is a partitition of the form $\bp = (c_k,\ldots)$.
This proves (c).
\end{proof}

By upper semicontinuity, for each $Z \in \Irr(\Pi(\bM))$
there exists a dense open subset $U_Z \subseteq Z$ such that
for all $k \in I$ and all $M,N \in U$ we have
$\sub_k(M) \cong \sub_k(N)$ and $\fac_k(M) \cong \fac_k(N)$.
Let $\sub_k(Z) := \sub_k(M)$ and $\fac_k(Z) := \fac_k(M)$ for 
some $M \in U$. 
(This is well defined up to isomorphism.)

Again, by upper semicontinuity it follows that
for each $Z \in \Irr(\Pi(\bM))$ and $k \in I$ there exists a unique
$\bp,\bq \in \cP_{c_k}$ such that 
$$
Z^{k,\bp} := Z \cap \Pi(\bM)^{k,\bp} 
\text{\;\;\; and \;\;\;}
Z_{k,\bq} := Z \cap \Pi(\bM)_{k,\bq} 
$$ 
are open and dense in $Z$.

We say that
a $\Pi$-module $M$ is \emph{generic in} $Z$, if $M$ is contained in 
a sufficiently small dense open subset of $Z$ defined by a finite set
of suitable open conditions.
The context will always imply which conditions are meant.
For example, we often demand that
$M \in Z$ with
$\sub_k(M) \cong \sub_k(Z)$ and $\fac_k(M) \cong \fac_k(Z)$ for all $k$.

\subsection{Fibre bundles and principal $G$-bundles} 
\label{ssec:bdl-basic}
All varieties considered are algebraic varieties over the 
algebraically closed field $K$, and our topology is the Zarisky topology. 
In particular, we use freely elementary concepts from algebraic geometry like
dimension, irreducible components and morphisms between varieties.
We recall some classical concepts from topology in our setting. 

A morphism between varieties 
$$
\pi\df B \ra V
$$ 
is a \emph{fibre bundle} with 
\emph{fibre} $F$, if $V$ has an open covering $(V_i)_{i\in I}$ together
with isomorphisms 
$$
\tau_i\df V_i\times F \ra \pi^{-1}(V_i)
$$ 
such that $\pi\tau_i(v,f) = v$ for all $(v,f) \in V_i \times F$. 
In particular, we have $\pi^{-1}(v) \cong F$ for all $v \in V$, and our fibre bundles are always locally trivial in the Zarisky topology. 
Thus, if $F$ is irreducible, there is
a natural bijection between the irreducible components of $V$ and the 
irreducible components of $B$, and we have 
$$
\dim(B) = \dim(V) + \dim(F).
$$

Let $\phi\df U \ra V$ be another morphism of varieties, then the \emph{pullback}
\[
\phi^*(B) := \{ (u,b) \in U\times B\mid \phi(u)=\pi(b) \},
\]
together with the projection 
$$
\phi^*(\pi)\df \phi^*(B) \ra U
$$ 
defined by $(u,b) \mapsto u$ is again a fibre bundle. 
In particular, it is easy to see how to trivialize
$(\phi^*(B),\phi^*(\pi))$ over the open subsets 
$\phi^{-1}(V_i) \subseteq U$ with fibre $F$.

Let now $G$ be an algebraic group which acts (algebraically) on 
$B$ from 
the right, such that $\pi(b\cdot g) = \pi(b)$ for all $b \in B$ and $g \in G$.
We say that the fibre bundle $\pi\df B \ra V$ is a 
\emph{principal $G$-bundle} if $G$ acts 
freely and transitively on the fibres of $\pi$. 
In this case, all fibres $\pi^{-1}(v)$ are isomorphic to $G$ as a variety. 
Again, it is easy to see that the pullback of a principal $G$-bundle is again a principal $G$-bundle.

\subsection{Grassmannians of submodules of fixed type} 
\label{ssec:GrassFT}
In this section, we fix some $k \in I$, and set  $c := c_k$.
Let 
$$
A := H_k = K[\vep_k]/(\vep_k^c).
$$
For a partition $\bp \in \cP_c$ with
we define the $A$-module
\[
A^\bp := H_k^\bp.
\]

For $A$-modules $M$ and $U$ let
$$
\Gr_U(M) := \{ V \subseteq M \mid V \cong U \}
$$
be the quasi-projective variety of $A$-submodules $V$
of $M$ which are isomorphic to $U$.
Similarly, let
$$
\Gr^U(M) := \{ V \subseteq M \mid M/V \cong U \}
$$
be the quasi-projective variety of $A$-factor modules
$M/V$ of $M$ which are isomorphic to $U$.
(Factor modules $M/V$ are defined via submodules $V$,
so we can think of $\Gr^U(M)$ as a variety of factor modules.)

Consider the open subset 
\[
\Inj_A(U,M) := \{ f \in \Hom_A(U,M) \mid f 
\text{ is injective} \}
\subset \Hom_A(U,M).
\]

Following Haupt \cite[Section~3.1]{H}, we consider
\[
\Inj_A(U,M) \to \Gr_U(M),\qquad 
f \mapsto \Ima(f).
\]
It is easy to see that this is a principal $\Aut_A(U)$-bundle with
$\Aut_A(U)$ acting on $\Inj_A(U,M)$ by precomposition. 
Now, $\Aut_A(U)$ and $\Inj_A(U,M)$ are, as open subsets in a vector space, smooth and irreducible. 
If $\Gr_U(M)$ is non-empty, then
$\Gr_U(M)$ is smooth and irreducible,
and we have
$$
\dim \Gr_U(M) = \dim \Hom_A(U,M) - \dim \End_A(U),
$$
see \cite[Theorem~3.1.1]{H}.
Similarly, if $\Gr^U(M)$ is non-empty, then $\Gr^U(M)$ is smooth and irreducible with
$$
\dim \Gr^U(M)= \dim \Hom_A(M,U) - \dim \End_A(U).
$$

For the special case $U = A^\bp$ and $M = A^b$ we have
$\Gr_U(M) \not= \varnothing$ if and only if $b \ge \length(\bp)$.
In this case, we get
\begin{equation}\label{eq:DimGr}
d(\bp,b) := \dim \Gr_U(M) 
= \sum_{i=1}^t p_i(b+1-2i)
\end{equation}
where $\bp = (p_1,\ldots,p_t)$.

\subsection{Two-step flags of submodules as fibre bundles}
For $A$-modules $U_1,U_2,M$ let
$$
\Gr_{U_1}^{U_2}(M) := \{ (V_1,V_2) \in \Gr_{U_1}(M) \times \Gr^{U_2}(M) \mid
V_1 \subseteq V_2
\}
$$
be the variety of 2-step chains 
$(0 \subseteq V_1 \subseteq V_2 \subseteq M)$
of submodules of $M$ with 
$V_1 \cong U_1$ and $M/V_2 \cong U_2$.
This is a closed subset of $\Gr_{U_1}(M) \times \Gr^{U_2}(M)$.

Let $\bp = (p_1,\ldots,p_t)$ with $p_t \ge 1$, and let $b \ge t$.
Let $U \in \Gr_{A^\bp}(A^b)$.
We obviously get $A^b/U \cong A^{\bq}$ for 
\begin{equation}
\bq := (c^{b-t}, c-p_t,c-p_{t-1},\ldots,c-p_d) \text{ with }
d := \min\{ 1\le i \le t \mid p_i < c \}.
\label{eq:def-bq}
\end{equation}
If $\bp = (c^t)$, we have just $\bq = (c^{b-t})$. 
We have an obvious isomorphism
$$
\Gr^{A^r}(A^\bq) 
\cong 
\{ V \subseteq A^b \mid U \subseteq V \text{ and } A^b/V \cong A^r \}.
$$
Clearly, 
$\Gr^{A^r}(A^\bq)$ is non-empty if and only if $r \le b-t$. 
In this case, 
it is smooth and irreducible of dimension
\[
\dim \Gr^{A^r}(A^\bq) = \dim \Hom_A(A^\bq,A^r) - \dim \End_A(A^r) = \dim \Hom_A(V,A^r)
\]
for any $V \in \Gr^{A^r}(A^\bq)$.
In view of~\cite[Theorem 3.1.1]{H} we only need to show the last equality. 
For each $V \in \Gr^{A^r}(A^\bq)$ there is a short exact sequence
\[
0 \to V \to A^\bq \to A^r \to 0
\]
of $A$-modules.
Since $A^r$ is a projective $A$-module, this sequence splits.
Applying the functor $\Hom_A(-,A^r)$ to this sequence yields the result.

\begin{Lem} \label{prp:flag-bdl}
The restriction of the projection
$$
\Gr_{A^\bp}(A^b) \times
\Gr^{A^r}(A^b) \ra \Gr_{A^\bp}(A^b) 
$$
defined by $(U,V) \mapsto U$ to $\Gr_{A^\bp}^{A^r}(A^b)$
yields a fibre bundle
\[
\pi\df \Gr_{A^\bp}^{A^r}(A^b) \ra \Gr_{A^\bp}(A^b) 
\]
with fibre $\Gr^{A^r}(A^\bq)$ with $\bq$ as 
in \eqref{eq:def-bq}. 
In particular, the fibre is smooth and irreducible.
\end{Lem}

Note, that our claim about the type of the fibre is clear, however the 
local triviality seems not to be so obvious. We will see this in the
next section.

\subsection{Proof of Lemma~\ref{prp:flag-bdl}}

\subsubsection{Notation} 
Let us write the partition $\bp$ as
\[
\bp=(c^{p'_0}, (c-1)^{p'_1}, \ldots, 1^{p'_{c-1}}) \text{ and set } p'_c := b-l. 
\]
 With this notation we can rewrite \eqref{eq:DimGr} as
\[
d(\bp,b) = \sum_{1\leq i< j\leq b} (p_i-p_j) = 
\sum_{0\leq i<j\leq c} (j-i)p'_i p'_j,
\] 
Recall, that this is the dimension of $\Gr_{A^\bp}(A^b)$.

Next, we define
\[
p''_j :=\sum_{i=0}^{j-1} p'_i\ \text{ for } 0 \le j \le c.
\]
Thus in particular $p''_0 = 0$ and $p''_c = t$, and we have 
\[
p''_{c-p_j}< j\leq p''_{c-p_j+1} \text{ for all } 1 \le j \le t.
\]
Finally we set $j_+ := p''_{c-p_j+1}+1$ for all $1 \le j \le t$.

\subsubsection{Affine charts for $\Gr_{A^\bp}(A^b)$}
Let $U \in \Gr_{A^\bp}(A^b)$. 
For an appropriate $A$-basis $\bv := (v_1,\ldots,v_b)$ of $A^b$ we have  
$$
U = \bigoplus_{j=1}^t A\vep^{c-p_j} v_j.
$$
We may set
\[
V_i := \bigoplus_{j=p_i''+1}^b A v_j \text{ for } 1 \le i \le c, 
\]
and consider the open subset
\[
\cO^\bv_U := \{ U' \in \Gr_{A^\bp}(A^b) \mid \vep^{c-i}U' \cap V_i = 0
\text{ for } 1 \le i \le c \}
\]
of $\Gr_{A^\bp}(A^b)$, which clearly contains $U$. 

Imitating the description of the open Schubert cells in ordinary Grassmannians
we see that each element $U'\in\cO^\bv_U$ has a unique set of generators 
in normal form with respect to the chosen basis:
\begin{align*}
g'_i    &=\vep^{c-p_i} v_i(U') \text{ for } 1 \le i \le t, \text{ where}
\\
v_i(U') &= v_i +  \sum_{j=i_+}^b
\underbrace{\left(\sum_{k=0}^{p_i-p_j-1} a_{ji}^{(k)}(U') \vep^k\right)}_{:=a_{ji}(U')\in A} v_j 
\end{align*}
Altogether we showed:

\begin{Lem}
With 
$$
I(\bp,b) := \{ (i,j,k) \in \Z^3 \mid
1\leq i\leq l, \ i_+\leq j\leq b,\ 0\leq k\leq p_i-p_j-1 \}
$$ 
we have an isomorphism of varieties
\[
\cO^\bv_U\ra K^{I(\bp,b)}
\]
defined by
\[
U' \mapsto (a_{ji}^{(k)}(U'))_{(i,j,k)\in I(\bp, b)}.
\]
\end{Lem}

We leave it as an exercise to verify directly that $I(\bp,b)$ has exactly 
$d(\bp,b)$ elements.

\subsubsection{Local trivialization}
For $U'\in\cO_U^\bv$ we define $g_{U'}\in\Aut_A(A^b)$ by 
\[
g_{U'}(v_i) := 
\begin{cases} 
v_i(U') &\text{ if } 1\leq i\leq l,
\\
v_i     &\text{ if } l< i\leq b.
\end{cases}
\]
Note that $\Aut_A(A^b)$ acts naturally on $\Gr_{A^\bp}(A^b)$ and on
$\Gr_{A^\bp}^{A^r}(A^b)$ as an algebraic group,
and we trivially have $g_{U'}(U) = U'$ for all $U' \in \cO_U^\bv$. 
Thus, we obtain the required local trivialization of 
$$
\pi\df \Gr_{A^\bp}^{A^r}(A^b) \ra \Gr_{A^\bp}(A^b)
$$ 
on the open neighbourhood $\cO_U^\bv$ by
\[
\cO_U^\bv \times \pi^{-1}(U) \ra \pi^{-1}(\cO_U^\bv),
\qquad (U',V)\mapsto g_{U'}(V).
\]
Here it is clear, that the map $\cO_U^\bv \ra \Aut_A(A^b)$
defined by $U' \mapsto g_{U'}$ 
is a morphism of varieties.

\subsection{Bundle construction}\label{subsec:bundlemain}
Let $\Pi = \Pi(C,D)$, and let 
$\bM = (M_1,\ldots,M_n)$ with $M_i$ a free $H_i$-module for all $i$.

We fix now some $k \in I$ and $\bU = (U_1,\ldots,U_n)$
with $U_i \subseteq M_i$ a free $H_i$-submodule of $M_i$ with
$U_i = M_i$ for all $i \not= k$.
Let 
$$
J := \prod_{i \in I} \Hom_{H_i}(U_i,M_i)
\text{\;\;\; and \;\;\;}
J_0 := \{ (f_i)_i \in J \mid f_i \text{ is injective for all } i \in I \}.
$$

With $\bM$ and $\bU$ defined as above, let $\bp$ and $\bq$ be partitions
in $\cP_{c_k}$.
We assume that $\bp = (c_k^r,q_1,\ldots,q_t)$ and
$\bq = (q_1,\ldots,q_t)$ with $r \ge 1$.
Assume that $M_k/U_k \cong E_k^{r}$.

We fix a
direct sum decomposition 
$$
M_k = U_k \oplus T_k
$$ 
of $H_k$-modules. 
Such a decomposition exists, since $U_k$ is by assumption free.
Note that $T_k$ is also a free $H_k$-module.

Let
$$
Y := Y^{k,\bq,\bp}
$$ 
be the variety of all triples 
$$
(U,M,f) \in \Pi(\bU)^{k,\bq} \times \ov{H}(\bM) \times J_0
$$
such that for all $(i,j) \in \ov{\Omega}$ the diagram
$$
\xymatrix{
{_iH_j} \otimes_j U_j \ar[r]^>>>>>{U_{ij}} \ar[d]^{1 \otimes f_j} & U_i \ar[d]^{f_i}
\\
{_iH_j} \otimes_j M_j \ar[r]^>>>>>{M_{ij}} & M_i
}
$$
commutes and such that for all $i \in I$ we have
$M_{i,\inn} \circ M_{i,\out} = 0$.
Note that for $(U,M,f) \in Y$ we have $M \in \Pi(\bM)$.

On $Y$ we have a free $G(\bU)$-action
defined by
$$
g \cdot (U,M,f) := ((g_iU_{ij}(1 \otimes g_j^{-1}))_{(i,j) \in \ov{\Omega}},
M,(f_i \circ g_i^{-1})_i).
$$
We define a diagram
$$
\xymatrix{
& Y \ar[dl]_{p'} \ar[dr]^{p''}\\
\Pi(\bU)^{k,\bq} \times J_0 && \Pi(\bM)^{k,\bp}
}
$$
by
$p'(U,M,f) := (U,f)$ and $p''(U,M,f) := M$.
The maps $p'$ and $p''$ are of central importance. 
We apply now the findings of the previous sections to describe them
in more detail.

\begin{Lem}\label{lusztig1}
With the notation above,
$p'$ is a vector bundle with fibres isomorphic to
$K^m$ with
$$
m = \sum_{(j,k) \in \ov{\Omega}} \dim \Hom_{H_k}(T_k,{_kH_j} \otimes_j M_j) - \dim \Hom_{H_k}(T_k,U_k').
$$
\end{Lem}

\begin{proof}
The canonical projection
$$
\Pi(\bU)^{k,\bq} \times \ov{H}(\bM) \times J_0 \to \Pi(\bU)^{k,\bq} \times J_0
$$
is obviously a vector bundle.
One also checks easily that $Y$ is a closed subset of 
$\Pi(\bU)^{k,\bq} \times \ov{H}(\bM) \times J_0$.
We fix $(U,f) \in \Pi(\bU)^{k,\bq} \times J_0$.
Let
$$
\cF := \{
M \in \ov{H}(\bM) \mid (U,M,f) \in Y 
\} = p''((p')^{-1}((U,f))).
$$
We have to show that $\cF \cong K^m$ for some $m$ which is
independent of $(U,f)$.
Set $U_k' := \Ima(U_{k,\inn})$.
Note that $\bp(U_k/U_k') = \bq$ and $\bp(M_k/U_k') = \bp$.
In other words, we have $M_k/U_k' \cong H_k^\bp$. 

Define 
$$
\eta\df \bigoplus_{j \in \ov{\Omega}(k)} 
\Hom_{H_k}(T_k,{_kH_j} \otimes_j M_j) \to
\Hom_{H_k}(T_k,f_k(U_k'))
$$
by 
$$
(T_{jk}^\vee)_j \mapsto 
\sum_{(k,j) \in \ov{\Omega}} \sgn(k,j) U_{kj}T_{jk}^\vee
$$
and let 
$$
\cF' := \Ker(\eta).
$$
Recall that $U_j = M_j$ for all $j \not= k$.
Since $f_k$ is a monomorphism, we have 
$$
\dim \Hom_{H_k}(T_k,f_k(U_k')) = \dim \Hom_{H_k}(T_k,U_k').
$$ 
Clearly, $\eta$ is $K$-linear.
Since $T_k$ is a free $H_k$-module (and therefore projective as an
$H_k$-module) we get that $\eta$ is surjective.
Thus we get $\cF' \cong K^m$ with 
$$
m := \sum_{(j,k) \in \ov{\Omega}} \dim \Hom_{H_k}(T_k,{_kH_j} \otimes_j M_j) - \dim \Hom_{H_k}(T_k,U_k').
$$
Let
$$
\mu\df \cF \to \cF'
$$
be defined by
$$
M \mapsto (M_{jk}^\vee|_{T_k})_{(j,k) \in \ov{\Omega}}.
$$
This is obviously an isomorphism of $K$-vectorspaces.
\end{proof}

We need one more auxiliary variety
\[
Y'' := \{ (M_k',M) \in \Gr^{T_k}(M_k) \times \Pi(\bM)^{k,\bp} \mid 
M'_k\supseteq \Ima(M_{k,\inn}) \}
\]
where we recall that $\Gr^{T_k}(M_k)$ is the Grassmannian of 
$H_k$-submodules $U$ of $M_k$ such that 
$M_k/U \cong T_k \cong H_k^r$.
We have two natural morphisms
\begin{alignat*}{4}
p''_1\df &Y  &&\ra Y'',           & \qquad (U,M,f) &\mapsto (\Ima(f_k), M),
\\
p''_2\df & Y'' &&\ra \Pi(\bM)^{k,\bp},&\qquad (M'_k,M)&\mapsto M.
\end{alignat*}
Obviously, we have $p'' = p''_2 \circ p''_1$.

\begin{Lem}\label{lem:aux1} 
With the above notation we have:
\begin{itemize}

\item[(a)]
$p''_1\df Y \ra Y''$ is a $G(\bU)$-principal bundle.

\item[(b)]
$p''_2\df Y'' \ra \Pi(\bM)^{k,\bp}$ is a fibre bundle with fibre
$\Gr^{T_k}(H_k^\bp)$. 
In particular, this fibre is smooth, irreducible and of dimension 
$$
\dim \Hom_{H_k}(H_k^\bp,T_k) - \dim_{H_k} \End_{H_k}(T_k).
$$

\item[(c)]
$p''$ is a fibre bundle with fibres isomorphic to 
$$
G(\bU) \times \Gr^{T_k}(H_k^\bp).
$$

\end{itemize}
\end{Lem}

\begin{proof}
(a): 
It is easy to see that 
\[
\Ima_k\df J_0 \ra \Gr^{T_k}(M_k),\qquad 
(f_i)_{i\in I}\mapsto \Ima(f_k)
\]
is a principal $G(\bU)$-bundle since 
$\Gr^{T_k}(M_k) \cong \Gr_{U_k}(M_k)$,
see also Section~\ref{ssec:GrassFT}.
Now, consider the morphism
\[
\phi_1\df Y'' \ra \Gr^{T_k}(M_k),\qquad (M'_k,M)\mapsto M'_k.
\]
We observe that $Y\cong \phi_1^*(J_0)$, the pullback of a 
$G(\bU)$-principal bundle, see Section~\ref{ssec:bdl-basic}. 
In fact, it follows directly from the 
definitions that $\phi_1^*(J_0)$ can be identified with
\[
Y' := \{ (M,f) \in \Pi(\bM)^{k,\bp} \times J_0 \mid \Ima(f_k) \supset 
\Ima(M_{k,\inn}) \}.
\]
Clearly, for each $(M,f) \in Y'$ there exists a unique $U \in \Pi(\bU)^{k,\bq}$
with $f \in \Hom_\Pi(U,M)$.

(b): 
There exists a unique partition $\bp^*$ such that 
$\Ima(M_{k,\inn})\cong H_k^{\bp^*}$ for all $M \in \Pi(\bM)^{k,\bp}$. 
With this notation we  consider the fibre bundle
\[
\pi\df \Gr_{H_k^{\bp^*}}^{T_k}(M_k) \ra \Gr_{H_k^{\bp^*}}(M_k)
\]
with fibre $\Gr^{T_k}(H_k^\bp)$, see Lemma~\ref{prp:flag-bdl}.
By construction, we have the natural morphism
\[
\phi_2\df \Pi(\bM) \ra \Gr_{H_k^{\bp^*}}^{H_k}(M_k),\qquad 
M \mapsto \Ima(M_{k,\inn}).
\]
It follows directly from the definitions that 
$$
\phi_2^*(\Gr_{H_k^{\bp^*}}^{T_k}(M_k)) = Y''.
$$
Thus 
$$
p''_2\df Y'' \ra \Pi(\bM)^{k,\bp}
$$ 
is a fibre bundle with the requested type of fibre.
This proves (b).

Part (c) is a direct consequence of (a) and (b) and the fact that
$p'' = p''_2 \circ p''_1$.
\end{proof}

\begin{Lem}\label{lusztig2}
The following hold:
\begin{itemize}

\item[(a)]
If $Z'$ is an irreducible component of $\Pi(\bU)^{k,\bq}$, then 
$$
Z := p''((p')^{-1}(Z' \times J_0))
$$
is an irreducible component of $\Pi(\bM)^{k,\bp}$.

\item[(b)]
The map $Z' \mapsto Z$ defines a bijection
$$
\Irr(\Pi(\bU)^{k,\bq}) \to \Irr(\Pi(\bM)^{k,\bp}).
$$

\item[(c)]
We have
$$
\dim(Z) - \dim(Z') = \dim H(\bM) - \dim H(\bU).
$$

\end{itemize}
\end{Lem}

\begin{proof}
Recall that we have two maps
$$
\xymatrix{
& Y \ar[dl]_{p'} \ar[dr]^{p''}\\
\Pi(\bU)^{k,\bq} \times J_0 && \Pi(\bM)^{k,\bp}
}
$$
defined by
$p'(U,M,f) := (U,f)$ and $p''(U,M,f) := M$.
The statements (a) and (b) follow immediately from combining 
Lemma~\ref{lusztig1} and Lemma~\ref{lem:aux1}(c).
Also from these lemmas we get that
$$
\dim(Z) + \dim G(\bU) + \dim \Hom_{H_k}(T_k,U_k/U_k')
= \dim(Z') + \dim(J_0) + m
$$ 
with 
$$
m = \sum_{(j,k) \in \ov{\Omega}} \dim \Hom_{H_k}(T_k,{_kH_j} \otimes_j M_j) - \dim \Hom_{H_k}(T_k,U_k').
$$
One easily checks that
\begin{align*}
\dim(J_0) - \dim G(\bU) &= \dim \Hom_{H_k}(U_k,T_k)\\ 
&= \dim \Hom_{H_k}(T_k,U_k)\\
&= \dim \Hom_{H_k}(T_k,U_k') +\dim \Hom_{H_k}(T_k,U_k/U_k').
\end{align*}
Furthermore, we have
\begin{align*}
\dim H(\bM) &= \sum_{(j,i) \in \Omega} \dim \Hom_{H_j}({_jH_i} \otimes_j 
M_i,M_j) = 
\sum_{(j,i) \in \Omega} \dim \Hom_{H_i}(M_i,{_iH_j} \otimes_j M_j),
\\
\dim H(\bU) &= \sum_{(j,i) \in \Omega} \dim \Hom_{H_j}({_jH_i} \otimes_j 
U_i,U_j) = 
\sum_{(j,i) \in \Omega} \dim \Hom_{H_i}(U_i,{_iH_j} \otimes_j U_j).
\end{align*}
Thus we have
and
$$
\dim H(\bM) - \dim H(\bU) = \sum_{(j,k) \in \ov{\Omega}} \dim \Hom_{H_k}(T_k,{_kH_j} \otimes_j M_j).
$$
Combining the above equalities we obtain
Thus we get
\begin{align*}
\dim(Z) - \dim(Z') &= \dim(J_0) - \dim G(\bU) - 
\dim \Hom_{H_k}(T_k,U_k/U_k') + m
\\
&=
\sum_{j \in \ov{\Omega}(k)} 
\dim \Hom_{H_k}(T_k,{_kH_j} \otimes_j M_j)
\\
&= \dim H(\bM) - \dim H(\bU).
\end{align*}
This proves (c).
\end{proof}

\subsection{Comparision to Lusztig's bundle construction}
\label{subsec:comparison}
In the classical case ($C$ symmetric and $D$ the identity matrix),
Lusztig constucted bundles
$$
\xymatrix{
& Y \ar[dl]_{p'}\ar[dr]^{p''}
\\
\Pi(\bU)^{k,0} \times J_0 && \Pi(\bM)^{k,p}
}
$$
with $p \ge 1$ and $\rk(\bM/\bU) = p\alpha_k$.
(Here $p$ stands for the partition $(c_k^p)$.)
Lusztig does not consider the situation
$$
\xymatrix{
& Y \ar[dl]_{p'}\ar[dr]^{p''}
\\
\Pi(\bU)^{k,q} \times J_0 && \Pi(\bM)^{k,p}
}
$$
with $p > q \ge 1$ and $\rk(\bM/\bU) = (p-q)\alpha_k$.
Thus for the classical case, one can see our construction as a 
refinement of Lusztig's bundle construction.
Another important difference is that in our setup
$$
\xymatrix{
& Y \ar[dl]_{p'}\ar[dr]^{p''}
\\
\Pi(\bU)^{k,\bq} \times J_0 && \Pi(\bM)^{k,\bp}
}
$$
from Section~\ref{subsec:bundlemain},
the closures (in $\Pi(\bM)$) of the irreducible
components of $\Pi(\bM)^{k,\bp}$ are in general not irreducible components of $\Pi(\bM)$.
In general, this will only be the case for maximal components of $\Pi(\bM)^{k,\bp}$.
Some examples of this kind can be found in Section~\ref{subsub:bundle}.

\subsection{The maps $e_{k,r},f_{k,r},e_{k,r}^*,f_{k,r}^*$}\label{subsec:3.3} 
Let 
$\bp,\bq \in \cP_{c_k}$ be partitions of the form
$$
\bp = (c_k^r,q_1,\ldots,q_t)
\text{\;\;\; and \;\;\;}
\bq = (q_1,\ldots,q_t).
$$ 
with $r \ge 1$.
Then Lemma~\ref{lusztig2}(b) and its dual yield bijections
$$
\xymatrix{
\bigsqcup_{\br \in \N^n} \Irr(\Pi(\br)^{k,\bp}) \ar@<1ex>[rr]^{f_{k,r}^*} &&
\bigsqcup_{\br \in \N^n} \Irr(\Pi(\br-r\alpha_k)^{k,\bq}) \ar@<1ex>[ll]^{e_{k,r}^*} 
}
$$
and
$$
\xymatrix{
\bigsqcup_{\br \in \N^n} \Irr(\Pi(\br)_{k,\bp}) \ar@<1ex>[rr]^{f_{k,r}} &&
\bigsqcup_{\br \in \N^n} \Irr(\Pi(\br-r\alpha_k)_{k,\bq}) \ar@<1ex>[ll]^{e_{k,r}} 
}
$$
with $f_{k,r}^* = (e_{k,r}^*)^{-1}$ and $f_{k,r} = (e_{k,r})^{-1}$.

The following lemma is a straightforward consequence of
Lemma~\ref{lusztig2}.

\begin{Lem}\label{lem:updown}
We have 
$$
(f_{k,1}^*)^r = f_{k,r}^*, \qquad (e_{k,1}^*)^r = e_{k,r}^*, \qquad
(f_{k,1})^r = f_{k,r}, \qquad (e_{k,1})^r = e_{k,r}.
$$
\end{Lem}

\subsection{The functions $\vph_i$ and $\vph_i^*$}\label{subsec:3.2} 
For $M \in \rep(H_i)$ let $[M:E_i]$ be the maximal number $p \ge 0$
such that there exists a direct sum decomposition $M = U \oplus V$ with $U \cong E_i^p$.
Define functions
$$
\vph_i,\vph_i^*\df \Pi(\br) \to \Z
$$ 
by
$$
\vph_i(M) := [\sub_i(M):E_i]
\text{\;\;\; and \;\;\;}
\vph_i^*(M) := [\fac_i(M):E_i].
$$
We obviously get $\vph_i(M) = p$ for all $M \in \Pi(\br)_{i,p}$, and
$\vph_i^*(M) = p$ for all $M \in \Pi(\br)^{i,p}$.


\section{Maximal irreducible components and crystal modules}
\label{sec:maxcomp}


\subsection{Maximal irreducible components}

\begin{Thm}\label{thm:dimbound}
For each $Z \in \Irr(\Pi(\bM))$ we have
$$
\dim(Z) \le \dim H(\bM).
$$
\end{Thm}

\begin{proof}
Let $Z \in \Irr(\Pi(\bM))$.
There exists some $k \in I$ with $\vph_k^*(Z) > 0$.
Thus there is a partition $\bp = (c_k^r,q_1,\ldots,q_t)$ 
with $r \ge 1$ such that $Z^{k,\bp} = Z \cap \Pi(\bM)^{k,\bp}$ is
dense in $Z$. 
Furthermore, we have
$Z^{k,\bp} \in \Irr(\Pi(\bM)^{k,\bp})$.
Let $Z'$ be the corresponding component of 
$\Pi(\bU)^{k,\bq}$, where $\bq := (q_1,\ldots,q_t)$ and
$\bU$ is defined as in Section~\ref{subsec:bundlemain}.
By induction we know that $\dim(Z') \le \dim H(\bU)$.
By Lemma~\ref{lusztig2}(c) we know that
$$
\dim(Z) - \dim(Z') = \dim H(\bM) - \dim H(\bU).
$$
This implies 
$$
\dim(Z)-\dim H(\bM) = \dim(Z') - \dim H(\bU) \le 0.
$$
This finishes the proof.
\end{proof}

An irreducible component $Z$ of $\Pi(\bM)$ is
\emph{maximal} if $\dim(Z) = \dim H(\bM)$.
We denote
the set of maximal irreducible components of 
$\Pi(\bM)$ by $\Irr(\Pi(\bM))^\mx$. 
Similarly, let
$\Irr(\Pi(\bM)^{k,\bp})^\mx$ and $\Irr(\Pi(\bM)_{k,\bp})^\mx$
be the sets of irreducible components of 
$\Pi(\bM)^{k,\bp}$ and $\Pi(\bM)_{k,\bp}$ of dimension 
$\dim H(\bM)$, respectively.

We can embed $H(\bM)$ into $\Pi(\bM)$ in the obvious way. 
By Theorem~\ref{thm:dimbound}, $H(\bM)$ is then a maximal
irreducible component of $\Pi(\bM)$.
Thus $\Irr(\Pi(\bM))^\mx$ is non-empty.
However, the sets
$\Irr(\Pi(\bM)^{k,\bp})^\mx$ and $\Irr(\Pi(\bM)_{k,\bp})^\mx$ 
can be empty, depending on the partition $\bp$.

\subsection{Crystal modules}\label{subsec:crystalmodules}
For $M \in \rep(\Pi)$ and $i \in I$ there are
canonical short exact sequences
$$
0 \to K_i(M) \to M \to \fac_i(M) \to 0
$$
and
$$
0 \to \sub_i(M) \to M \to C_i(M) \to 0.
$$
Here $K_i(M)$ is the unique submodule of $M$ with $M/K_i(M) \cong \fac_i(M)$, and
$C_i(M) = M/U$ is the unique factor module of $M$ with 
$U \cong \sub_i(M)$.

We say that $M \in \nil_E(\Pi)$ is a \emph{crystal module}
if $\fac_i(M)$ and $\sub_i(M)$ are locally free for all $i$, 
and if $K_i(M)$ and $C_i(M)$ are crystal modules for all $i \in I$.
By definition the trivial module $0$ is a crystal module.
 
Clearly, if $M \in \nil_E(\Pi)$ is a crystal module, then we have
$\dim \sub_i(M) = c_i\vph_i(M)$ and $\dim \fac_i(M) = c_i\vph_i^*(M)$.

For $i,j \in I$
there is a canonical homomorphism $f\df \sub_j(M) \to \fac_i(M)$
defined by $u \mapsto u + K_i(M)$.

\begin{Lem}\label{lem:strictlya}
Let $M \in \rep(\Pi)$.
For $i \in I$ and each submodule $U \subseteq M$ we have
$$
\fac_i(M/U) \cong \fac_i(M)/(U+K_i(M)).
$$
\end{Lem}

\begin{proof}
We have canonical short exact sequences
$$
0 \to K_i(M) \to M \to \fac_i(M) \to 0
$$
and
$$
0 \to K_i(M/U) \to M/U \to \fac_i(M/U) \to 0.
$$
We use that submodules of a factor module $M/V$ are in
bijection with submodules $W$ of $M$ with $V \subseteq W \subseteq M$.
In this way, we can interpret $U+K_i(M)$ as a submodule of $\fac_i(M)$,
and $K_i(M/U)$ as a submodule of $M$.

We get the obvious inclusions displayed in the following diagram:
$$
\xymatrix{
& M \ar@{-}[dr]\ar@{-}[dl]
\\
K_i(M/U) \ar@{-}[dr] && U + K_i(M)  \ar@{-}[dr]\ar@{-}[dl]
\\
& U  && K_i(M) 
}
$$
There is an epimorphism
$$
\pi\df M/U \to \fac_i(M)/(U+K_i(M))
$$
defined by $m+U \mapsto m+ (U+K_i(M))$.
Since all composition factors of the image of $\pi$ are
isomorphic to $S_i$, the epimorphism $\pi$ factors through
$\fac_i(M/U)$.
This yields an epimorphism 
$$
\pi'\df \fac_i(M/U) \to \fac_i(M)/(U+K_i(M)).
$$
We obviously have $U \subseteq K_i(M/U)$.
We also have $K_i(M) \subseteq K_i(M/U)$, since all
composition factors of 
$M/K_i(M/U)$ are isomorphic to $S_i$, and 
$M/K_i(M)$ is the unique maximal factor module
with this property.
It follows that $U+K_i(M) \subseteq K_i(M/U)$.
Thus for dimension reasons, $\pi'$ has to be an isomorphism.
\end{proof}

\begin{Lem}\label{lem:strictly0}
For $i,j \in I$ and $M \in \nil_E(\Pi)$ a crystal module the following 
hold:
\begin{itemize}

\item[(i)]
If the canonical homomorphism
$$
f\df \sub_j(M) \to \fac_i(M)
$$
is non-zero, then $i = j$ and $M$ has a direct summand isomorphic to 
$E_i$.

\item[(ii)]
If the canonical homomorphism
$$
f\df \sub_j(M) \to \fac_i(M)
$$
is zero, then 
$$
\fac_i(M) \cong \fac_i(C_j(M)) 
\text{\;\;\; and \;\;\;}
\sub_j(M) = \sub_j(K_i(M)).
$$

\end{itemize}
\end{Lem}

\begin{proof}
We first prove (i).
If $i \not= j$, then the canonical homomorphism $f\df \sub_j(M) \to \fac_i(M)$
is obviously zero.
Thus assume that $i=j$ and that $f\df \sub_i(M) \to \fac_i(M)$
is non-zero.
We know that $\sub_i(M)$ and $\fac_i(M)$ are free $H_i$-modules.
Let $p\df \fac_i(M) \to \tp(\fac_i(M))$ be the canonical projection
of $\fac_i(M)$ onto its top. 
Note that $\tp(\fac_i(M)) = \tp_i(M)$.
By Lemma~\ref{lem:strictlya} we have 
$$
\fac_i(M/\sub_i(M)) \cong \fac_i(M)/f(\sub_i(M)).
$$
Now suppose that $pf = 0$.
Then we get 
$$
\tp_i(\fac_i(M)) \cong \tp_i(\fac_i(M)/f(\sub_i(M))). 
$$
Together with our assumption that $f(\sub_i(M)) \not= 0$, this implies
that $\fac_i(M/\sub_i(M))$ is not free, a contradiction to our assumption that
$M$ is a crystal module.
Thus we proved that $pf \not= 0$.
This implies that 
there is a submodule $U$ of $\sub_i(M)$
with $U \cong E_i$ and 
$f(U) \cong E_i$.
This yields a homomorphism $g\df \fac_i(M) \to U$ with $gf\iota_U = 1_U$,
where $\iota_U\df U \to \sub_i(M)$ denotes the inclusion.
We have $f = f_2f_1$ with the obvious homomorphisms
$f_1\df \sub_i(M) \to M$ and $M \to \fac_i(M)$.
We get $(gf_2)(f_1\iota_U) = 1_U$.
This shows that $f_1\iota_U\df U \to M$ is a split monomorphism.
It follows that
$M$ has a direct summand isomorphic to $E_i$.
This finishes the proof of (i).
Part (ii) is straightforward.
\end{proof}

Let 
$$
\nil_E^\cry(\Pi,\bM) = \Pi(\bM)^\cry
$$ 
be the subset of crystal modules in $\nil_E(\Pi,\bM) = \Pi(\bM)$.
 
An irreducible component $Z \in \Irr(\Pi(\bM))$ is a
\emph{crystal component} if it contains a dense open subset 
of crystal modules.

\begin{Prop}\label{prop:strictly1}
For $Z \in \Irr(\Pi(\bM))$ the following are equivalent:
\begin{itemize}

\item[(i)]
$Z$ is maximal.

\item[(ii)]
$Z$ is a crystal component.

\end{itemize}
\end{Prop}

\begin{proof}
(ii) $\implies$ (i):
Let $\bM = (M_1,\ldots,M_n)$ be locally free.
Suppose $Z \in \Irr(\Pi(\bM))$ is
a crystal component.
By Lemma~\ref{lusztig0}(c) there exists some $k \in I$ and some $p > 0$
such that $\vph_k^*(Z) = p$.
Now choose $\bU = (U_1,\ldots,U_n)$ such that
$U_i = M_i$ for all $i \not= k$, and $U_k$ is a free $H_k$-submodule of $M_k$ such that $M_k/U_k \cong E_k^p$.
Let $(Z^{k,p})' \in \Irr(\Pi(\bU)^{k,0})$ be the irreducible component corresponding to 
$Z^{k,p} := Z \cap \Pi(\bM)^{k,p}$ under
the bijection $\Irr(\Pi(\bM)^{k,p}) \to \Irr(\Pi(\bU)^{k,0})$ from Lemma~\ref{lusztig2}(b).
Finally, let $Z'$ be the closure of $(Z^{k,p})'$ in $\Pi(\bU)$.
It follows that $Z' \in \Irr(\Pi(\bU))$, since $\Pi(\bU)^{k,0}$ is
non-empty and open in $\Pi(\bU)$.
It is straightforward that the component $Z'$ is again a crystal component.
By induction, $Z'$ is maximal, i.e. $\dim(Z') = \dim H(\bU)$.
Now Lemma~\ref{lusztig2} implies that $\dim(Z) = \dim H(\bM)$.
In other words, $Z$ is maximal.

(i) $\implies$ (ii):
Let $\bM = (M_1,\ldots,M_n)$ be locally free.
Assume that $Z \in \Irr(\Pi(\bM))$ is maximal, and
that $Z$ is not a crystal component.
Let $\br := \rk(\bM)$ be minimal such that such a $Z$ exists.

By minimality, it follows that $\fac_k(Z)$ or $\sub_k(Z)$ 
is not free for some $k$.
Without loss of generality we assume that $\fac_1(Z)$ is not free.
Again by minimality, we know that $\vph_1^*(Z) = 0$, i.e.
$\fac_1(Z)$ does not have a direct summand isomorphic to $E_1$.

There exists some $s \in I$ such that $\vph_s(Z) > 0$, i.e.
$\sub_s(Z)$ contains 
a direct summand isomorphic to $E_s$.
Now choose $\bU = (U_1,\ldots,U_n)$ such that
$U_i = M_i$ for all $i \not= s$, and $U_s = M_s/U$ is a free $H_s$-factor module of $M_s$ with $U \cong E_s$.

There is a partition $\bp = (c_s,q_1,\ldots,q_t)$ such that
$Z_{s,\bp} := Z \cap \Pi(\bM)_{s,\bp}$ is open and dense in $Z$.
We have $Z_{s,\bp} \in \Irr(\Pi(\bM)_{s,\bp})$.
Set $\bq := (q_1,\ldots,q_t)$.

Under the bijection  $\Irr(\Pi(\bM)_{s,\bp}) \to \Irr(\Pi(\bU)_{s,\bq})$ from
the dual of Lemma~\ref{lusztig2}(b), let
$Z_{s,\bp}' \in \Irr(\Pi(\bU)_{s,\bq})$ be the irreducible component
corresponding to $Z_{s,\bp}$.
Let $Z'$ be the closure of $Z_{s,\bp}'$ in $\Pi(\bU)$.

The dual of Lemma~\ref{lusztig2} yields $\bU$ and an irreducible
component $Z_{s,\bq}'$ of $\Pi(\bU)_{s,\bq}$ corresponding to $Z$.
Let $Z'$ be the closure of $Z_{s,\bq}$ in $\Pi(\bU)$.
By the dual of Lemma~\ref{lusztig2}(c) we know that $Z'$ is a maximal
irreducible component of $\Pi(\bU)$.
Furthermore, by induction $Z'$ is a crystal component.
In particular, this implies that $\fac_1(Z')$ is free.

Let $M$ be generic in $Z$.
There is a short exact sequence
$$
0 \to E_s \xrightarrow{f} M \to M' \to 0
$$
with $M'$ generic in $Z'$.
This implies that $s=1$. 
(Otherwise $\fac_1(M) \cong \fac_1(M')$ and
therefore $\fac_1(Z) = \fac_1(Z')$, a contradiction.)
The short exact sequence above is non-split.
(Otherwise $\fac_1(M) = \fac_1(Z)$ would contain a direct summand
isomorphic to $E_1$, a contradiction.)
In other words, we have $\Ext_\Pi^1(M',E_1) \not= 0$.

Without loss of generality we assume that $f\df E_1 \to M$ is just an
inclusion map and that
$$
M_1 = U_1 \oplus E_1.
$$

By Lemma~\ref{lem:strictlya} we have
$$
\fac_1(M') \cong \fac_1(M/E_1) \cong \fac_1(M)/(E_1+K_1(M)).
$$
We have $E_1 + K_1(M) = p(E_1)$, where
$$
p\df M \to \fac_1(M)
$$
is the obvious canonical epimorphism.
Since $\fac_1(M')$ is free, and $\fac_1(M)$ is not, this implies 
$p(E_1) \not= 0$.
Since $\fac_1(M)$ does not contain a free direct summand,
and $\fac_1(M')$ is free, we even get
$p(E_1) = \fac_1(M)$ and therefore $\fac_1(M') = 0$.
In particular, $\fac_1(M)$ is isomorphic to a proper factor module
of $E_1$.

We have $M = (M_{ij}) \in \Pi(\bM)$ and
$M' = (M_{ij}') \in \Pi(\bU)$ with
$M_{ij}' = M_{ij}$ for all $(i,j)$ with $i \not=1$ and $j \not= 1$.
Furthermore, we have $M_{1,\out}|_{U_1} = M_{1,\out}'$ and $M_{1,\out}|_{E_1} = 0$.
(For the last equality we used that $E_1$ is a submodule of $M$.)
In particular, we have $\Ima(M_{1,\out}) = \Ima(M_{1,\out}')$.

By induction we know that $M'$ is a crystal module.
This implies that $\Ima(M_{1,\out}')$, $\Ker(M_{1,\inn}')$
and therefore also $\Ker(M_{1,\out}')/\Ima(M_{1,\out}')$ are
free $H_1$-modules.

We now describe the $H_1$-linear maps 
$$
M_{1,\inn}\df \wti{M}_1 \to M_1
\text{\;\;\; and \;\;\;}
M_{1,\inn}'\df \wti{M}_1 \to U_1
$$
where 
$$
\wti{M}_1 = \bigoplus_{j \in \ov{\Omega}(1)} {_1}H_j \otimes_j M_j.
$$
We have a decomposition 
$$
\wti{M}_1 = \Ima(M_{1,\out}') \oplus V \oplus W
$$
into a direct sum of $H_1$-modules, where $\Ima(M_{1,\out}') \oplus V = \Ker(M_{1,\inn}')$.
(Here we used that $\Ima(M_{1,\out}')$, $\Ker(M_{1,\inn}')$ and 
$\Ima(M_{1,\inn}') \cong W$ are free $H_1$-modules. 
It follows also that $V$ is free.)
We have 
$$
V \cong \Ker(M_{1,\inn}')/\Ima(M_{1,\out}') \cong
\Ext_\Pi^1(M',E_i) \not= 0.
$$
(For the last isomorphism we used Lemma~\ref{lem:3.7}(iii).)
Using both decompositions 
$\wti{M}_1 = \Ima(M_{1,\out}') \oplus V \oplus W$ and
$M_1 = U_1 \oplus E_1$ we can
write $M_{1,\inn}\df \wti{M}_1 \to M_1$ as a matrix
$$
M_{1,\inn} = \left(\bbm 0&0&f_{13}\\0&f_{22}&f_{23}\ebm\right)\df 
\Ima(M_{1,\out}') \oplus V \oplus W \to U_1 \oplus E_1
$$
where the $f_{ij}$ are $H_1$-module homomorphisms, and
$M_{1,\inn}'\df \wti{M}_1 \to U_1$ is given by the matrix
$$
M_{1,\inn}' = \left(\bbm 0&0&f_{23}\ebm\right)\df 
\Ima(M_{1,\out}') \oplus V \oplus W \to U_1.
$$
Since $\fac_1(M') = 0$, 
we get that $f_{13}\df W \to U_1$ is an isomorphism. 

We now define a new $\Pi$-module $\ov{M}$
by replacing $f_{22}\df V \to E_1$ by an $H_1$-linear map $\ov{f}_{22}\df V \to E_1$
of maximal rank. 
Thus $\ov{f}_{22}$ is an epimorphism, since $V$ is non-zero and free.
It is clear that $\ov{M}$ is indeed a $\Pi$-module.
(Using that $M_{1,\out}|_{E_1} = 0$ and that $M_{i,\inn} \circ M_{i,\out} = 0$ for all $i$,
we get that $\ov{M}_{i,\inn} \circ \ov{M}_{i,\out} = 0$ for all $i$.) 
Since $f_{13}$ and $\ov{f}_{22}$ are both epimorphisms, we get that
$\ov{M}_{1,\inn}$ is an epimorphism.
This means that $\fac_1(\ov{M}) = 0$.

We have $\ov{M}/E_1 = M'$.
Since $M'$ is generic in $Z'$, we get that
$\ov{M}$ is also contained in $Z$.
(Here we used again Lemma~\ref{lusztig2}.)
This is a contradiction to $M$ being generic in $Z$, since
$\fac_1(M) \not= 0$ and $\fac_1(\ov{M}) = 0$.
Thus we got a contradiction to our assumption that $\fac_1(M)$ is not free.

So we proved that $\fac_i(M)$ is free for all $i$.
Dually one shows that $\sub_i(M)$ is free for all $i$.
Thus by induction, $Z$ is a crystal component.
\end{proof}

\begin{Cor}\label{cor:3.12}
$\Pi(\bM)^\cry$ is equidimensional of dimension 
$\dim H(\bM)$.
\end{Cor}

\begin{Cor}\label{cor:strata}
For a partition $\bp \in \cP_{c_k}$ which is not of the form $\bp = (c_k^p)$
for some $p$, we have
$\dim \Pi(\bM)^{k,\bp} < \dim H(\bM)$ and
$\dim \Pi(\bM)_{k,\bp} < \dim H(\bM)$.
\end{Cor}

Examples of non-maximal irreducible components can be found in Section~\ref{sec:examples}.


\section{Geometric construction of crystal graphs}\label{sec:crystal}


This section follows very closely \cite{NT}, which on the other hand is
based on \cite{KS}.

\subsection{Kac-Moody algebras}
\label{subsec:recallkacmoody}
Let $C = (c_{ij}) \in M_n(\Z)$ be a symmetrizable generalized Cartan
matrix.
Recall that $I = \{ 1,\ldots,n \}$.

Let $\h$ be a $\C$-vector space of dimension $2n-\operatorname{rank}(C)$,
and let $\{ \alpha_1,\ldots,\alpha_n \} \subset \h^*$ and 
$\{ \alpha_1^\vee,\ldots,\alpha_n^\vee \} \subset \h$
be linearly independent subsets of the vector spaces $\h^*$ and $\h$,
respectively, such that
$$
\alpha_i(\alpha_j^\vee) = c_{ji}
$$
for all $i,j$.
(Here $\h^* = \Hom_\C(\h,\C)$ is the dual space of $\h$.)

Now $\g = (\g,[-,-])$ is the Lie algebra over $\C$
generated by $\h$ and the symbols $e_i$
and $f_i\ (i \in I)$ satisfying the following 
defining relations:
\begin{itemize}

\item[(i)]
$[h,h'] = 0$ for all $h,h' \in \h$;

\item[(ii)]
$[h,e_i] = \alpha_i(h)e_i$ and
$[h,f_i] = -\alpha_i(h)f_i$ for all $i$ and all $h \in \h$;

\item[(iii)]
$[e_i,f_j] = \delta_{ij}\alpha_i^\vee$ for all $i,j$;

\item[(iv)]
$(\ad(e_i)^{1-c_{ij}})(e_j) = 0$ for all $i \not= j$;

\item[(v)]
$(\ad (f_i)^{1-c_{ij}})(f_j) = 0$ for all $i \not= j$.
\end{itemize}
(For $x,y \in \g$ and $m \ge 1$
we set $\ad(x)(y) := \ad(x)^1(y) := [x,y]$
and
$\ad(x)^{m+1}(y) := \ad(x)^m([x,y])$.)
The Lie algebra $\g$ is the \emph{Kac-Moody algebra} associated with $C$.
As a general reference on Kac-Mody algebras, we refer to Kac's book \cite{Ka}.

Let $\n = \n(C)$ be the Lie subalgebra of $\g$ 
generated by $e_i\ (i \in I)$.
Then $U(\n)$ is the associative $\C$-algebra with generators $e_i\ (i \in I)$ subject to the 
relations
\[
 (\ad e_i)^{1-c_{ij}}(e_j) = 0
\]
for all $i \not= j$.
(Here we interpret $[x,y]$ as a commutator $xy-yx$.)

Let $\h^* = \h_1^* \oplus \h_2^*$ be a vector space
decomposition, where $\h_1^*$ is just the subspace with 
basis $\{ \alpha_1,\ldots,\alpha_n \}$, and $\h_2^*$ is any
direct complement of $\h_1^*$ in $\h^*$.
Let 
$$
\bil{-,-}\df \h^* \times \h^* \to \C
$$
be the standard bilinear form, defined by
$\bil{\alpha_i,\alpha_j} := \alpha_i(\alpha_j^\vee) = c_{ji}$,
$\bil{\alpha_i,x} := \bil{x,\alpha_i} := x(\alpha_i^\vee)$,
and $\bil{x,y} := 0$ for all $x,y \in \h_2^*$ and $i,j \in I$.
(Identifying the $\alpha_i$ with the standard basis of $\Z^n$, this
definition of $\bil{-,-}$ is compatible with the bilinear form defined in 
Section~\ref{subsec:DefPi}.)

Finally, let us fix a basis 
$\{\vpi_j\mid 1\le j \le 2n - \operatorname{rank}(C) \}$ of
$\mathfrak{h}^*$ such that
\[
\vpi_j(\alpha^\vee_i) = \delta_{ij},\qquad
(i \in I,\ 1\le j\le 2n-\operatorname{rank}(C)).  
\]
The $\vpi_j$ are the {\it fundamental weights}. 
Note that for $i \in I$ we have
$$
\alpha_i = \sum_{j \in I} c_{ji}\vpi_j. 
$$
We denote by 
$$
P := \{ \nu \in \h^* \mid \bil{\nu,\alpha_i} \in \Z \text{ for all }
i \in I \}
$$
the {\it integral weight lattice}, and we set
$$
P^+ := \{ \nu\in P \mid \bil{\nu,\alpha_i} \ge 0 \text{ for all } 
i \in I \}.
$$
The elements in $P^+$ are called {\it dominant integral weights}.
We have 
$$
P = \bigoplus_{j \in I} \Z \vpi_j \oplus
\bigoplus_{j=n+1}^{2n-\operatorname{rank}(C)} \C\vpi_j
\text{\;\;\; and \;\;\;}
P^+ = \bigoplus_{j \in I} \N \vpi_j \oplus
\bigoplus_{j=n+1}^{2n-\operatorname{rank}(C)} \C\vpi_j.
$$
For 
$$
\lambda = \sum_{j \in I} a_j \vpi_j +
\sum_{j=n+1}^{2n-\operatorname{rank}(C)} a_j \vpi_j
$$
in $P$,
we have 
$$
a_j = \bil{\lambda,\alpha_j}
$$
for $1 \le j \le n$.
Let
$$
R := \bigoplus_{i \in I} \Z\alpha_i
$$
be the \emph{root lattice},
and set
$$
R^+ := \bigoplus_{i \in I} \N\alpha_i.
$$

\subsection{Crystals}\label{subsec:2.2}
As before,
let $C$ be a symmetrizable generalized Cartan matrix with symmetrizer $D$,
and let $P$ be the associated integral weight lattice.

Following \cite[Section~7.2]{K1},
a \emph{crystal} is a tuple $(B,\wt,\tilde{e}_i,\tilde{f}_i,\vep_i,\vph_i)$ 
where $B$ is a set and 
$$
\wt\df B \to P,
\quad\quad\quad
\tilde{e}_i,\tilde{f}_i\df B \to B \cup \{ \varnothing \},
\quad\quad\quad
\varepsilon_i,\varphi_i\df B \to \Z
$$
with $i \in I$ 
are maps
such that for all $i \in I$ and all $b \in B$ the following hold:
\begin{itemize}

\item[(cr1)]
$\varphi_i(b) = \varepsilon_i(b) + \bil{\wt(b),\alpha_i}$;
￼

\item[(cr2)] 
$\vph_i(\tilde{e}_i(b)) = \vph_i(b)+1$,\qquad
$\vep_i(\tilde{e}_i(b)) = \vep_i(b)-1$,\qquad
$\wt(\tilde{e}_i(b)) = \wt(b)+\alpha_i$;

\item[(cr3)] 
For all $b,b' \in B$ the following are equivalent:
\begin{itemize}
\item[(a)]
$\tilde{f}_i(b) = b'$;
\item[(b)]
$\tilde{e}_i(b') =b$.
\end{itemize}
\end{itemize}

Kashiwara \cite{K1} also allows the values of $\varepsilon_i$ and $\varphi_i$ to be $-\infty$.
This assumption is not needed here.

A \emph{lowest weight crystal} is a crystal with a distinguished element $b_- \in B$ (the \emph{lowest weight element}) such that the following hold:
\begin{itemize}

\item[(cr4)]
For each $b \in B$ there exists a sequence $(i_1,\ldots,i_t)$ with $i_k \in I$
for all $1 \le k \le t$ such that
$$
b_- = \tilde{f}_{i_1} \cdots \tilde{f}_{i_t}(b).
$$

\item[(cr5)] 
For each $b \in B$ and $i \in I$ we have
$$ 
\varphi_i(b) = \max\{ m \mid \tilde{f}_i^m(b)  \not= \varnothing \}.
$$

\end{itemize}
For lowest weight crystals, the functions
$\wt$, $\tilde{f}_i$, $\vep_i$ and $\vph_i$
are determined by the $\tilde{e}_i$ and the weight 
of $b_-$.
Here we are mainly interested in the \emph{infinity crystal} 
$B(-\infty)$ of $U_q(\n)$. 
Kashiwara and Saito \cite[Proposition~3.2.3]{KS} gave a criterion
when a lowest weight crystal is isomorphic to the crystal
$B(-\infty)$.
The following is a reformulation of this criterion due to Tingley and Webster
\cite[Proposition~1.4]{TW}.
We use the criterion as a definition of $B(-\infty)$.

\begin{Prop}\label{prop:2.4}
Fix a set $B$ with operators 
$$
\tilde{e}_i,\tilde{f}_i,\tilde{e}_i^*,\tilde{f}_i^*\df B \to B \cup \{ \varnothing \}.
$$ 
Assume $(B,\tilde{e}_i,\tilde{f}_i)$ and $(B,\tilde{e}_i^*,\tilde{f}_i^*)$ are both lowest weight crystals with the same lowest weight element $b_-$, 
where the other data is determined by setting $\wt(b_-) = 0$. 
Assume further that for all $i,j \in I$ and all $b \in B$
the following hold:
\begin{itemize}

\item[(i)]
$\tilde{e}_i(b)$, $\tilde{e}_i^*(b) \not= \varnothing$. 

\item[(ii)]
If $i \not= j$, then
$\tilde{e}_i^*\tilde{e}_j(b) = \tilde{e}_j\tilde{e}_i^*(b)$.

\item[(iii)]
For all $b \in B$ we have $\varphi_i(b) + \varphi_i^*(b) - 
\bil{\wt(b),\alpha_i} \ge 0$.

\item[(iv)]
If $\varphi_i(b) + \varphi_i^*(b) - \bil{\wt(b),\alpha_i} = 0$, 
then $\tilde{e}_i(b) = \tilde{e}_i^*(b)$.

\item[(v)]
If $\varphi_i(b) + \varphi_i^*(b) - \bil{\wt(b),\alpha_i} \ge 1$, 
then $\varphi_i(\tilde{e}_i^*(b)) = \varphi_i(b)$ and 
$\varphi_i^*(\tilde{e}_i(b)) = \varphi_i^*(b)$.

\item[(vi)]
If $\varphi_i(b) + \varphi_i^*(b) - \bil{\wt(b),\alpha_i} \ge 2$, 
then $\tilde{e}_i\tilde{e}_i^*(b) = \tilde{e}_i^*\tilde{e}_i(b)$.

\end{itemize}
Then $(B,\tilde{e}_i,\tilde{f}_i) \cong (B,\tilde{e}_i^*,\tilde{f}_i^*) \cong B(-\infty)$.
\end{Prop}

\subsection{Geometric crystal operators}\label{subsec3.5a}
As before, let
$$
\cB = \bigsqcup_{\br \in \N^n} \Irr(\Pi(\br))^\mx.
$$
We set 
$$
\cB_\br := \Irr(\Pi(\br))^\mx.
$$
We know that $Z \cap \Pi(\br)^\cry$ is dense in $Z$
for each $Z \in \cB_\br$.
The operators $e_{i,r},f_{i,r},e_{i,r}^*,f_{i,r}^*$ defined in Section~\ref{subsec:3.3}
yield bijections 
$$
f_{i,r}\df \bigsqcup_{\br \in \N^n} \Irr(\Pi(\br)_{i,p})^\mx 
\to \bigsqcup_{\br \in \N^n} \Irr(\Pi(\br)_{i,q})^\mx
$$
and
$$
f_{i,r}^*\df \bigsqcup_{\br \in \N^n} \Irr(\Pi(\br)^{i,p})^\mx 
\to \bigsqcup_{\br \in \N^n} \Irr(\Pi(\br)^{i,q})^\mx. 
$$
where $r := p-q \ge 1$.
For $Z \in \cB$ we set
$$
\tilde{f}_i(Z) := 
\begin{cases}
\ov{f_{i,1}(Z)} & \text{if $\vph_i(Z) \ge 1$},
\\
\varnothing & \text{otherwise},
\end{cases}
\text{\;\;\;\; and \;\;\;\;}
\tilde{f}_i^*(Z) := 
\begin{cases}
\ov{f_{i,1}^*(Z)} & \text{if $\vph_i^*(Z) \ge 1$},
\\
\varnothing & \text{otherwise}.
\end{cases}
$$
Similarly, we have
bijections
$$
e_{i,r}\df \bigsqcup_{\br \in \N^n} \Irr(\Pi(\br)_{i,q})^\mx 
\to \bigsqcup_{\br \in \N^n} \Irr(\Pi(\br)_{i,p})^\mx
$$
and
$$
e_{i,r}^*\df \bigsqcup_{\br \in \N^n} \Irr(\Pi(\br)^{i,q})^\mx 
\to \bigsqcup_{\br \in \N^n} \Irr(\Pi(\br)^{i,p})^\mx 
$$
where $r := p-q \ge 1$.
For $Z \in \cB$ we set
$$
\tilde{e}_i(Z) := \ov{e_{i,1}(Z)} 
\text{\;\;\; and \;\;\;}
\tilde{e}_i^*(Z) := \ov{e_{i,1}^*(Z)}.
$$
Thus, we defined maps
$$
\tilde{f}_i,\tilde{f}_i^*\df \cB \to \cB \cup \{ \varnothing \}
\text{\;\;\; and \;\;\;}
\tilde{e}_i,\tilde{e}_i^*\df \cB \to \cB.
$$

Note that our definition of the crystal operators is slightly different from
the one used in \cite{KS}, see also \cite{NT}.
The reason is that we are working with a refined version of Lusztig's
bundle construction, see our discussion in Section~\ref{subsec:comparison}.

For $Z \in \Irr(\Pi(\br))^\mx$ define 
\begin{align*}
\wt(Z) &:= \br,
\\
\vph_i(Z) &:= \min \{ \vph_i(M) \mid M \in Z \},
\\
\vep_i(Z) &:= \vph_i(Z) - \bil{\wt(Z),\alpha_i},
\\
\vph_i^*(Z) &:= \min \{ \vph_i^*(M) \mid M \in Z \}, 
\\
\vep_i^*(Z) &:= \vph_i^*(Z) - \bil{\wt(Z),\alpha_i}.
\end{align*}
(In the definition of $\wt(Z)$, we identify the rank vector 
$\br = (r_1,\ldots,r_n)$ with $r_1\alpha_1 + \cdots + r_n\alpha_n 
\in R^+ \subset P$.)

\subsection{The $*$-operator}\label{subsec3.5b}
For a matrix $A$ let ${^t}A$ denote its transpose.
Let $\Pi$ and $\cB$ be defined as before.
For a representation $M = (M(\vep_i),M(\alpha_{ij}^{(g)})) \in \nil_E(\Pi,d)$ let
$$
S(M) := (S(M(\vep_i)), S(M(\alpha_{ij}^{g}))) \in \nil_E(\Pi,d)
$$ 
where
$$
S(M(\vep_i)) := {^t}M(\vep_i)
\text{\;\;\; and \;\;\;}
S(M(\alpha_{ij}^{(g)})) := {^t}M(\alpha_{ji}^{(g)}).
$$
For each dimension vector $d$, we get an automorphism $S_d$ of the variety $\nil_E(\Pi,d)$ defined by $S_d(M) := S(M)$.
This construction yields an automorphism $S_\br$ of $\Pi(\br)$ for each rank vector $\br$.
The automorphism $S_\br$ induces a permutation 
$$
*_\br\df \cB_\br \to \cB_\br.
$$
This yields a permutation 
$$
* \df \cB \to \cB.
$$
For all $i \in I$
we get
$$
*\tilde{e}_i* = \tilde{e}_i^*,\qquad *\tilde{e}_i^** = \tilde{e}_i,\qquad *\tilde{f}_i* = \tilde{f}_i^*, \qquad *\tilde{f}_i^** = \tilde{f}_i.
$$

\subsection{Examples}
Let $\Pi = \Pi(C,D)$ with
$$
C = \left(\bbm 2&-1\\-2&2\ebm\right)
\text{\;\;\; and \;\;\;}
D = \left(\bbm 2&0\\0&1\ebm\right).
$$
Thus $C$ is of Dynkin type $B_2$ and $D$ is minimal.
Let $Z \in \Irr(\Pi((2,1))^\mx$ 
be the maximal irreducible component with  generic $\Pi$-module
$$
M= 
{\bsm 1\\1\esm \oplus \bsm 1\\1\\2\esm}.
$$
(Each number stands for a basis vector of $M$, with $i$
belonging to $e_iM$. At the same time, $i$ represents a composition
factor isomorphic to $S_i$.
The module $M$ is a direct sum of two serial modules, whose composition
series look as indicated.)
The following picture illustrates how the various operators
$\tilde{e}_k,\tilde{f}_k,\tilde{e}_k^*,\tilde{f}_k^*$ 
act on $Z$.
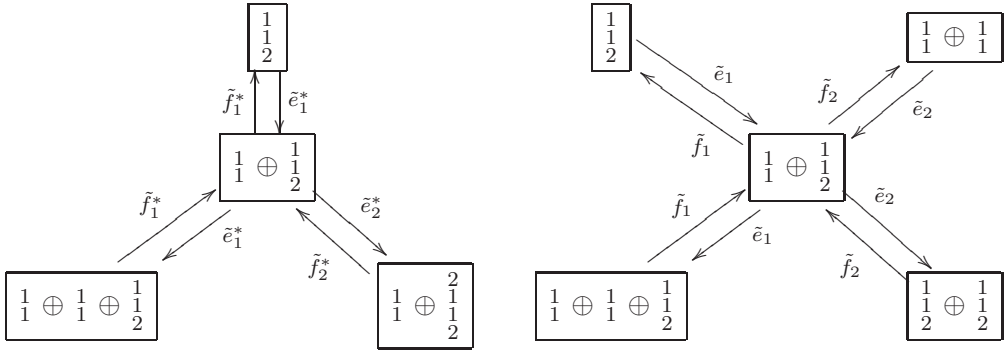
$$
\xymatrix{
& *+[F]{\bsm 1\\1\\2\esm} \ar@<1ex>[d]^{\tilde{e}_1^*} &&
*+[F]{\bsm 1\\1\\2\esm} \ar@<1ex>[dr]^{\tilde{e}_1}
&&
*+[F]{\bsm 1\\1\esm \oplus \bsm 1\\1\esm} \ar@<1ex>[dl(0.7)]^{\tilde{e}_2}
\\
& *+[F]{\bsm 1\\1\esm \oplus \bsm 1\\1\\2\esm} \ar@<1ex>[dl(0.6)]^{\tilde{e}_1^*}
\ar@<1ex>[dr]^{\tilde{e}_2^*}\ar@<1ex>[u]^{\tilde{f}_1^*} &&
& *+[F]{\bsm 1\\1\esm \oplus \bsm 1\\1\\2\esm} \ar@<1ex>[dl(0.6)]^{\tilde{e}_1}
\ar@<1ex>[dr(0.8)]^{\tilde{e}_2}\ar@<1ex>[ul(0.8)]^{\tilde{f}_1}\ar@<1ex>[ur(0.7)]^{\tilde{f}_2}
\\
*+[F]{\bsm 1\\1\esm \oplus \bsm 1\\1\esm \oplus \bsm 1\\1\\2\esm} \ar@<1ex>[ur]^{\tilde{f}_1^*} &&
*+[F]{\bsm 1\\1\esm \oplus \bsm 2\\1\\1\\2\esm} \ar@<1ex>[ul]^{\tilde{f}_2^*} &
*+[F]{\bsm 1\\1\esm \oplus \bsm 1\\1\esm \oplus \bsm 1\\1\\2\esm} \ar@<1ex>[ur]^{\tilde{f}_1} &&
*+[F]{\bsm 1\\1\\2\esm \oplus \bsm 1\\1\\2\esm} \ar@<1ex>[ul]^{\tilde{f}_2}
}
$$
We also have $\tilde{f}_2^*(Z) = \varnothing$.
We have
\begin{align*}
\vph_1^*(Z) &= 2, & \vph_2^*(Z) &= 0,
&
 \vph_1(Z) &= 1, & \vph_2(Z) &= 1,
\\
\wt(Z) &= (2,1), &
\bil{\wt(Z),\alpha_1} &= 3, &  \bil{\wt(Z),\alpha_2} &= -2.
\end{align*}
Thus we get
\begin{align*}
\dim \Ext_\Pi^1(M,E_1) &= c_1(\vph_1(Z) + \vph_1^*(Z) - 
\bil{\wt(Z),\alpha_1}) = 2(1+2-3) = 0,
\\
\dim \Ext_\Pi^1(M,E_2) &= c_2(\vph_2(Z) + \vph_2^*(Z) - 
\bil{\wt(Z),\alpha_2}) = 1(1+0-(-2)) = 3.
\end{align*}
Let $Z' \in \Irr(\Pi((2,2))^\mx$ be the maximal 
irreducible component with generic $\Pi$-module
$$
M' = 
{\bsm 1\\1\esm \oplus \bsm 2\\1\\1\\2\esm}.
$$
We get
\begin{align*}
\dim \Ext_\Pi^1(M',E_1) &= c_1(\vph_1(Z') + \vph_1^*(Z') - 
\bil{\wt(Z'),\alpha_1}) = 2(1+1-2) = 0,
\\
\dim \Ext_\Pi^1(M',E_2) &= c_2(\vph_2(Z') + \vph_2^*(Z') - 
\bil{\wt(Z'),\alpha_2}) = 1(1+1-0) = 2.
\end{align*}
%

\subsection{Realization of $B(-\infty)$}\label{subsec3.5c}
The formula in the following lemma is an analogue of the formula
in \cite[Lemma~3.16]{NT}.

\begin{Lem}\label{lem:3.14}
Let $Z \in \cB$, and let $M$ be generic in $Z$.
Then we have
$$
\dim \Ext_\Pi^1(M,E_i) = c_i(\vph_i(Z) + \vph_i^*(Z) - \bil{\wt(Z),\alpha_i}).
$$
\end{Lem}

\begin{proof}
This follows from Corollary~\ref{cor:3.8}, Lemma~\ref{lem:3.7} and the definitions of 
$\vph_i(Z)$ and $\vph_i^*(Z)$.
\end{proof}

The next lemma is an analogue of \cite[Proposition~3.17]{NT}.

\begin{Lem}\label{lem:3.15}
For $Z \in \cB$ and $i,j \in I$ the following hold:
\begin{itemize}

\item[(i)]
$\tilde{e}_i(Z)$, $\tilde{e}_i^*(Z) \not= \varnothing$. 

\item[(ii)]
If $i \not= j$, then
$\tilde{e}_i^*\tilde{e}_j(Z) = \tilde{e}_j\tilde{e}_i^*(Z)$.

\item[(iii)]
For all $Z \in \cB$ we have $\varphi_i(Z) + \varphi_i^*(Z) - 
\bil{\wt(Z),\alpha_i} \ge 0$.

\item[(iv)]
If $\varphi_i(Z) + \varphi_i^*(Z) - \bil{\wt(Z),\alpha_i} = 0$, 
then $\tilde{e}_i(Z) = \tilde{e}_i^*(Z)$.

\item[(v)]
If $\varphi_i(Z) + \varphi_i^*(Z) - \bil{\wt(Z),\alpha_i} \ge 1$, 
then $\varphi_i(\tilde{e}_i^*(Z)) = \varphi_i(Z)$ and 
$\varphi_i^*(\tilde{e}_i(Z)) = \varphi_i^*(Z)$.

\item[(vi)]
If $\varphi_i(Z) + \varphi_i^*(Z) - \bil{\wt(Z),\alpha_i} \ge 2$, 
then $\tilde{e}_i\tilde{e}_i^*(Z) = \tilde{e}_i^*\tilde{e}_i(Z)$.

\end{itemize}
\end{Lem}

\begin{proof}
Throughout, let $Z \in \cB$, and let $M \in Z$ be generic.
In particular, we assume that the maps $\vph_i$ and $\vph_i^*$
take minimal values on $M$.

(i):
This follows from the definition of $\tilde{e}_i$ and $\tilde{e}_i^*$ combined with
Lemma~\ref{lusztig2}.

(ii):
Let
$Z_1 := \tilde{e}_i^*\tilde{e}_j(Z)$ and $Z_2 := \tilde{e}_j\tilde{e}_i^*(Z)$.
Since $i \not= j$, the canonical homomorphisms
$\sub_j(Z_k) \to \fac_i(Z_k)$ with $k=1,2$
are both zero.
This implies 
$$
\tilde{f}_i^*\tilde{f}_j(Z_k) = \tilde{f}_j\tilde{f}_i^*(Z_k) = Z
$$
for $k=1,2$.
Here we used Lemma~\ref{lem:strictly0}(ii).
Since $\tilde{f}_p\tilde{e}_p = 1_{\cB}$ and
$\tilde{f}_p^*\tilde{e}_p^* = 1_{\cB}$ for all $p \in I$ we get that
$Z_1 = Z_2$.

(iii):
This follows directly from Lemma~\ref{lem:3.14}.

(iv):
Assume that $\varphi_i(Z) + \varphi_i^*(Z) - \bil{\wt(Z),\alpha_i} = 0$.
Then
Lemma~\ref{lem:3.14} yields that 
$$
\Ext_\Pi^1(M,E_i) = \Ext_\Pi^1(E_i,M) = 0.
$$
This implies that 
$$
\tilde{e}_i(Z) = \tilde{e}_i^*(Z) = \ov{Z \oplus \cO(E_i)}.
$$
(Here we used the notion of direct sums of irreducible components from
\cite{CBS}.)

(v):
Assume that $\varphi_i(Z) + \varphi_i^*(Z) - \bil{\wt(Z),\alpha_i} \ge 1$.
Then
Lemma~\ref{lem:3.14} implies that 
$\dim \Ext_\Pi^1(M,E_i) > 0$.
Let $Z' := \tilde{e}_i(Z)$.
There is a short exact sequence
$$
0 \to E_i \to M' \to M \to 0
$$
with $M'$ generic in $Z'$.
This sequence is non-split, since $\Ext_\Pi^1(M,E_i) \not= 0$.
Applying $\Hom_\Pi(-,E_i)$ we get
$$
\dim \fac_i(M') - \dim \fac_i(M) < \dim(E_i).
$$
Since both $\fac_i(M')$ and $\fac_i(M)$ are free (using that $M$ and $M'$ are crystal modules), this inequality implies that
$\fac_i(M') \cong \fac_i(M)$ and therefore $\fac_i(Z') \cong \fac_i(Z)$.
This implies 
$$
\vph_i^*(\tilde{e}_i(Z)) = \vph_i^*(Z). 
$$
The other equality in (ii) is proved dually, working with
$Z' = \tilde{e}_i^*(Z)$ instead of $Z' = \tilde{e}_i(Z)$.

(vi):
Assume that $\varphi_i(Z) + \varphi_i^*(Z) - \bil{\wt(Z),\alpha_i} \ge 2$.
Consider a generic $M'$ in $\tilde{e}_i(Z)$ and a generic $M''$ in 
$\tilde{e}_i^*\tilde{e}_i(Z)$. 
We claim that the canonical homomorphism from 
$f''\df \sub_i(M'') \to \fac_i(M'')$ is trivial. 
By Lemma~\ref{lem:strictly0}(i)
it is enough to show that
$E_i$ is not a direct summand of $M''$. 
First, note that $E_i$ cannot be a summand of $M$.
Namely, if $M = E_i \oplus N$, then, since $M$ is generic, this would imply $\Ext_\Pi^1(M,E_i) = 0$, 
which is false by Lemma~\ref{lem:3.14}. 
Consequently, since $\Ext_\Pi^1(E_i,M) > 0$, a generic $M' \in \tilde{e}_i(Z)$ also doesn't contain $E_i$ as a direct summand. 
Thus we get a non-split short exact sequence
$$
0 \to E_i \to M' \to M \to 0.
$$
Applying $\Hom_\Pi(-,E_i)$ and keeping in mind that
$\Ext_\Pi^1(E_i,E_i) = 0$ we get
$$
\dim \Ext_\Pi^1(M',E_i) \ge \dim \Ext_\Pi^1(E_i,M) - c_i > 0.
$$
For the second inequality we used that $\vph_i(Z) + \vph_i^*(Z) -
\bil{\wt(Z),\alpha_i} \ge 2$.
Now the same argument as before shows that $M''$ does not contain $E_i$ as a direct summand.
Thus we proved that $f'' = 0$.
Now we can proceed as in the proof of part (ii).
This finishes the proof. 
\end{proof}

Finally, the following theorem is an analogue of \cite[Theorem~3.18]{NT}.

\begin{Thm}\label{thm:3.16}
We have
$$
(\cB,\wt,\tilde{e}_i,\tilde{f}_i,\vep_i,\vph_i) \cong (\cB,\wt,\tilde{e}_i^*,\tilde{f}_i^*,\vep_i^*,\vph_i^*)
\cong B(-\infty).
$$
\end{Thm}

\begin{proof}
The set $\cB$ of maximal irreducible components together with either set of operators 
$(\wt,\tilde{e}_i,\tilde{f}_i,\vep_i,\vph_i)$ or $(\wt,\tilde{e}_i^*,\tilde{f}_i^*,\vep_i^*,\vph_i^*)$
defined in Section~\ref{subsec3.5a} is a crystal.
(In (cr1) we just define 
$\vep_i(Z) := \vph_i(Z) - \bil{\wt(Z),\alpha_i}$.
The first and third equalities in (cr2) are clearly satisfied for
$\cB$.
These together with (cr1) imply the second equality of (cr2).
To check (cr3) is straightforward with the help of Lemma~\ref{lusztig2}.)

For any $0 \not=Z \in \cB$, there exist $i$ and $j$ such that
$\tilde{f}_i(Z) \not= 0$ and $\tilde{f}_j^*(Z) \not= 0$.
We also know that in these cases we have $\wt(\tilde{f}_i(Z)) = \wt(Z) - \alpha_i$
and $\wt(\tilde{f}_j^*(Z)) = \wt(Z) - \alpha_j$.
For $b_-$ we take the (unique) irreducible component
$Z_-$ of $\Pi(0)$. 
(The variety $\Pi(0)$ is just a point.)
Together with the definitions of $\vph_i$ and $\vph_i^*$, this implies 
that the crystals $(\cB,\wt,\tilde{e}_i,\tilde{f}_i,\vep_i,\vph_i)$ and 
$(\cB,\wt,\tilde{e}_i^*,\tilde{f}_i^*,\vep_i^*,\vph_i^*)$ are both
lowest weight crystals. 

The conditions of Proposition~\ref{prop:2.4} are all satisfied by
Lemma~\ref{lem:3.15}. 
This yields isomorphisms of crystals
$
B(-\infty) \cong (\cB,\wt,\tilde{e}_i,\tilde{f}_i,\vep_i,\vph_i) \cong (\cB,\wt,\tilde{e}_i^*,\tilde{f}_i^*,\vep_i^*,\vph_i^*)$.
\end{proof}
􏰭

\subsection{Littlewood-Richardson coefficients}
Let $\Pi = \Pi(C,D)$, $\g = \g(C)$ and $\cB$ be defined as before.

For $\la \in P^+$ a dominant integral weight, let 
$V(\lambda)$ be the associated irreducible integrable highest weight $\g$-module with highest weight $\la$. 

One of the main applications of crystal graphs is the calculation of tensor product multiplicities.
More precisely, it is well known that the tensor product multiplicities
\[
c_{\la,\mu}^\nu := [V(\la)\otimes V(\mu) : V(\nu)]
\]
can be expressed in terms of crystal graphs. 
The numbers $c_{\la,\mu}^\nu$ are called
\emph{Littlewood-Richardson coefficients}.

For $\la \in P^+$ define
\begin{align*}
\cB_\la &:= \{Z \in \cB \mid \varphi_i(Z) \le a_i \mbox{ for every } i\in I\},\\[2mm]
\cB_\la^* &:= \{Z \in \cB \mid \varphi_i^*(Z) \le a_i  \mbox{ for every } i\in I\},
\end{align*}
where $a_i := \bil{\la,\alpha_i}$.

The permutation $*\df \cB \to \cB$ yields equalities
$$
*(\cB_\la) = \cB_\la^* 
\text{\;\;\; and \;\;\;}
*(\cB_\la^*) = \cB_\la.
$$

Define
$$
\cB_{\la,\mu}^\nu := 
\{Z\in \cB_\la^* \cap \cB_\mu \mid \wt(Z) = \la + \mu - \nu\}.
$$
An example can be found in Section~\ref{subsec:LRcoeff}.

Using our description of $B(-\infty)$, this gives the following result.

\begin{Prop}
\[
c_{\la,\mu}^{\nu} = |\cB_{\la,\mu}^\nu|.
\]
\end{Prop}

\begin{proof}
Let $B(\la)$ denote the crystal graph of $V(\lambda)$ with highest weight vertex $b_\la$ of weight $\la$. 
It is known \cite[Proposition 4.2]{K1} that
\[
c_{\la,\mu}^\nu = |\{ b \in B(\la) \mid \wt(b) = \nu - \mu \mbox{ and } 
\varepsilon_i(b) \le \bil{\mu,\alpha_i} \mbox{ for every } i\in I\}|.
\]
It is also known that $B(\la)$ can be realized as a subgraph of $\cB\equiv B(-\infty)$. More precisely, it follows
from \cite[Proposition 8.2]{K1} that there is a unique injective map 
$$
\iota_\la\df B(\la) \to B(-\infty)
$$
sending $b_\la$ to the lowest weight element of $B(-\infty)$ and satisfying
\[
\iota_\la \tilde{e}_i = \tilde{f}_i \iota_\la,\qquad
\varepsilon_i(b) = \varphi_i(\iota_\la(b)),\qquad
\wt(\iota_\la(b)) = \la - \wt(b),\qquad (b\in B(\la)).
\]
Moreover, we have
$$
\iota(B(\la)) = \cB_\la^*.
$$
This shows that the sets
$$
\iota_\la\left(\{ b \in B(\la) \mid \wt(b) = \nu - \mu \mbox{ and } \varepsilon_i(b) \le \bil{\mu,\alpha_i} \text{ for all } i \}\right)
$$
and
$$
\cB_{\la,\mu}^\nu = 
\{Z\in \cB_\la^* \cap \cB_\mu \mid \wt(Z) = \la + \mu - \nu\}
$$
are equal.
\end{proof}


\section{Convolution algebras}
\label{sec:convolution}


In this section, assume that $K = \C$.

\subsection{The convolution algebra $\cM(\Pi)$}
Let $\Pi = \Pi(C,D)$ and define the convolution algebra $\tF(\Pi)$
as in Section~\ref{subsec:convolution}.
For $c_{ij} \le 0$ we define
$$
\wti{\theta}_{ij} := \ad(\wti{\theta}_i)^{1-c_{ij}}(\wti{\theta}_j)
\in \tM(\Pi).
$$
Let $\Serre$ be the ideal in $\tM(\Pi)$ generated by the
functions $\ttheta_{ij}$ with $c_{ij} \le 0$.
Define
$$
\cM(\Pi) := \tM(\Pi)/\Serre.
$$
For $\br \in \N^n$ set 
$$
\Serre_\br := \Serre \cap \tM(\Pi)_\br
\text{\;\;\; and \;\;\;}
\cM(\Pi)_\br := \cM(\Pi) \cap \tM(\Pi)_\br.
$$
We get 
$$
\Serre = \bigoplus_{\br \in \N^n} \Serre_\br
\text{\;\;\; and \;\;\;}
\cM(\Pi) = \bigoplus_{\br \in \N^n} \cM(\Pi)_\br.
$$

Let $\theta_i := \ttheta_i + \Serre$ be the residue class of
$\ttheta_i$ in $\cM(\Pi)$.
It follows immediately, that we have a surjective algebra
homomorphism
$$
U(\n) \to \cM(\Pi)
$$
defined by $e_i \mapsto \theta_i$.

\subsection{Serre relations}
In contrast to \cite[Proposition~3.10]{GLS5}, the functions $\ttheta_i$ 
do not in general satisfy the Serre relations.

\begin{Lem}\label{serre}
For $\Pi = \Pi(C,D)$ assume that $c_{ij} < 0$ and 
$c_i \ge 2$ for some $i,j \in I$.
Then there exists an indecomposable locally free $\Pi$-module
$X(i,j)$
with rank vector $(1-c_{ij})\alpha_i +\alpha_j$.
\end{Lem}

\begin{proof}
Recall that $g_{ij} = |\gcd(c_{ij},c_{ji})|$, $f_{ij} = |c_{ij}|/g_{ij}$ and
$c_ic_{ij} = c_jc_{ji}$.
It follows that $f_{ij} \le c_j$.
Without loss of generality assume $c_{12}<0$ and $c_1 \ge 2$.
For each $1 \le f \le f_{12}$ and $1 \le g \le g_{12}$ let
$E_{1f}^{(g)}$ be a copy of $E_1$ with
basis $\{b_{1f}^{(g)},\ldots,b_{c_1f}^{(g)} \}$ such that
$$
\varepsilon_1b_{if}^{(g)} = 
\begin{cases}
b_{i-1f}^{(g)} & \text{if $i \ge 2$},\\
0 & \text{otherwise}.
\end{cases}
$$
Furthermore, let $
\{b_1,\ldots,b_{c_1} \}$ be a basis of another copy of $E_1$ 
such that
$$
\varepsilon_1b_i = 
\begin{cases}
b_{i-1} & \text{if $i \ge 2$},\\
0 & \text{otherwise}.
\end{cases}
$$
Let $a_1,\ldots,a_{c_2}$ be a basis of $E_2$ such that
$$
\varepsilon_2a_i = 
\begin{cases}
a_{i-1} & \text{if $i \ge 2$},\\
0 & \text{otherwise}.
\end{cases}
$$
For $1 \le f \le f_{12}$ and $1 \le g \le g_{12}$ define
$$
\alpha_{12}^{(g)}a_{c_2-f+1} := b_{1f}^{(g)}
$$
and 
$$
\alpha_{21}^{(g)}b_{c_1} := a_1.
$$
It is easy to check that thus defines a locally free $\Pi$-module $X(1,2)$
with
$\rk(X(1,2)) = (1-c_{12})\alpha_1 + \alpha_2$.
Note that $X(1,2)$ is a tree module in the sense of Crawley-Boevey \cite{CB1}.
In particular, $X(1,2)$ is indecomposable.
This finishes the proof.
\end{proof}

\begin{Prop}
For $\Pi = \Pi(C,D)$ the following are equivalent:
\begin{itemize}

\item[(i)]
The functions $\ttheta_1,\ldots,\ttheta_n$ satisfy the Serre relations.

\item[(ii)]
$\Serre = 0$.

\item[(iii)]
If $c_{ij} < 0$ for some $i,j \in I$, then $c_i = 1$.

\end{itemize}
\end{Prop}

\begin{proof}
It is obvious that (i) and (ii) are equivalent.

(i) $\implies$ (iii):
Assume $c_{ij} < 0$ and $c_i \ge 2$ for some $i,j \in I$.
For $X(i,j)$ as defined in the proof of Lemma~\ref{serre} it is straightforward
to check that
$$
\wti{\theta}_{ij}(X(i,j)) \not= 0.
$$
Thus $\ttheta_1,\ldots,\ttheta_n$ do not satisfy the Serre relations.

(iii) $\implies$ (i),(ii): 
Suppose (iii) holds.
We can assume that $\ov{Q}(C)$ is connected.
If $n \ge 2$, then $C$ is symmetric, and $D$ is the identity matrix. 
Thus $\Pi(C,D)$ is a classical preprojective algebra
associated with a quiver.
If $n = 1$, then
$\Pi(C,D) = K[\vep_1]/(\vep_1^{c_1})$.
In the first case, Lusztig \cite{L1} proved that $\ttheta_1,\ldots,\ttheta_n$ satisfy the Serre relations.
In the second case, $\Serre = 0$, since
there are no Serre relations.
\end{proof}

\subsection{Example}
Let $\Pi = \Pi(C,D)$ where
$$
C = \left(\bbm 2&-6\\-2&2\ebm\right)
\text{\;\;\; and \;\;\;}
D = \left(\bbm 2 &0\\0&6\ebm\right).
$$
We have
$c_1 = 2$, $c_2 = 6$,
$f_{12} = 3$ and $g_{12} = 2$.
The $\Pi$-module $X(1,2)$ 
constructed in the proof of Lemma~\ref{serre} looks as follows:
$$
\xymatrix@-1ex{
&&&&&&&&1 \ar[d]\ar[r] & 1
\\
& 2 \ar[d]_{\alpha_{12}^{(1)}} \ar@/^2ex/[dd]^>>>>>{\alpha_{12}^{(2)}} \ar[rr] && 
2 \ar[d]_{\alpha_{12}^{(1)}} \ar@/^2ex/[dd]^>>>>>{\alpha_{12}^{(2)}} \ar[rr] && 
2 \ar[d]_{\alpha_{12}^{(1)}} \ar@/^2ex/[dd]^>>>>>{\alpha_{12}^{(2)}} \ar[r] & 
2 \ar[r] & 2 \ar[r] & 2
\\
1 \ar[r] & 1 & 1 \ar[r] & 1 & 1 \ar[r] & 1
\\
1 \ar[r] & 1 & 1 \ar[r] & 1 & 1 \ar[r] & 1
}
$$
(The numbers in the picture correspond to basis vectors of $X(1,2)$
with $i$ being in $e_iX(1,2)$.
The arrows indicate how the arrows of the quiver of $\Pi$ act on the
basis vectors.)
We get
$$
\ttheta_{12}(X(1,2)) = \ad(\ttheta_1)^7(\ttheta_2)(X(1,2)) =
- 7 \ttheta_1^6\ttheta_2\ttheta_1(X(1,2)) =  -7 \cdot (6!) =
-(7!). 
$$
Thus we see that $\Serre \not= 0$.

As a smaller example, one could also take the preprojective algebra $\Pi$ of type $B_2$ with minimal symmetrizer together 
with the module $X$ displayed in Section~\ref{subsub:nonmaxsupp}.

\subsection{The support of the Serre relations}

\begin{Lem}\label{lem:strictly2}
Suppose $c_{ij} \le 0$.
Then there is no indecomposable crystal module
$M \in \nil_E(\Pi)$ with $\rk(M) = (1-c_{ij},1)$.
\end{Lem}

\begin{proof}
Without loss of generality assume $c_{12} \le 0$.
Let $\br = (1-c_{12},1)$, and let $M \in \nil_E(\Pi)$ be a crystal module with $\rk(M) = \br$.
 
We consider the maps
$$
M_1 \xrightarrow{M_{1,\out}} \wti{M}_1 \xrightarrow{M_{1,\inn}} M_1
$$
as defined in Section~\ref{subsec:species}.
The maps $M_{1,\out}$ and $M_{1,\inn}$ are $H_1$-module
homomorphisms with $\Ima(M_{1,\out}) \subseteq \Ker(M_{1,\inn})$.
Since $M$ is a crystal module, we know that $M_{1,\out}$ and $M_{1,\inn}$
are split, i.e. their images, kernels and cokernels are free $H_1$-modules
and therefore direct summands.
As $H_1$-modules, we have 
$M_1 \cong H_1^{1-c_{12}}$ and $\wti{M}_1 = {_1}H_2 \otimes_2 H_2
\cong  H_1^{-c_{12}}$.

Let $r_1 := \rk(\Ima(M_{1,\out}))$ and $r_2 := \rk(\Ima(M_{1,\inn}))$.
Since $\Ima(M_{1,\out}) \subseteq \Ker(M_{1,\inn})$, we get $r_1 + r_2 \le -c_{12}$.
Let $C$ be a submodule of $M_1$ such that
$\Ima(M_{1,\inn}) \oplus C = M_1$.
Thus $C \cong \Coker(M_{1,\inn})$.
We have $\rk(\Ker(M_{1,\out})) = (1-c_{12})-r_1$ and 
$\rk(C) = (1-c_{12})-r_2$.
Thus $\Ker(M_{1,\out}) \cap C$ contains a
submodule $U$ isomorphic to $H_1$.
(Here we use the following fact: If $V_1$ and $V_2$ are free submodules
of $H_1^m$ with $\rk(V_1) + \rk(V_2) \ge m+1$,
then $V_1 \cap V_2$ contains a free submodule $V$ with $\rk(V) = 1$.
Namely, there is a short exact sequence
$$
0 \to V_1 \cap V_2 \to V_1 \oplus V_2 \xrightarrow{f} V_1+V_2 \to 0
$$
with $f(v_1,v_2) := v_1-v_2$.
We have $\dim \tp(V_1+V_2) \le m$, since $V_1 + V_2 \subseteq H_1^m$.
The module $V_1 \oplus V_2$ is a projective $H_1$-module with 
$\dim \tp(V_1 \oplus V_2) \ge m+1$.
Thus $V_1 \cap V_2$ contains a direct summand isomorphic to $H_1$.)
It follows that $U$ is a direct summand of $M$.
Thus the $\Pi$-module $M$ is decomposable.
This finishes the proof.
\end{proof}

\begin{Cor}\label{cor:strictly3}
Suppose $c_{ij} \le 0$.
For each crystal module $M \in \Pi(\br)$ we have
$$
\ttheta_{ij}(M) = 0.
$$
\end{Cor}

\begin{proof}
Since $\ttheta_{ij}$ is defined as an iterated Lie bracket of the
generators $\ttheta_i$ and $\ttheta_j$, it is a primitive element in the Hopf
algebra $\tM(\Pi)$.
Thus the support of $\ttheta_{ij}$ consists of indecomposable
$\Pi$-modules.
Let $\br = (1-c_{ij},1)$, and let $M \in \Pi(\br)^\cry$.
By Lemma~\ref{lem:strictly2}, we know that $M$ is decomposable.
Thus we get $\ttheta_{ij}(M) = 0$.
\end{proof}

\begin{Cor}\label{cor:strictly4}
For $c_{ij} \le 0$ and $\br = (1-c_{ij})\alpha_i + \alpha_j$, we have
$$
\dim \supp(\ttheta_{ij}) < \dim H(\br).
$$
\end{Cor}

One can see Corollary~\ref{cor:strictly4} as a first step towards a proof
of Conjecture~\ref{mainconj}.


\section{Semicanonical bases}\label{sec:semican}


In this section, assume that $K = \C$.

\subsection{Semicanonical functions}
This section follows very closely Lusztig \cite{L2}.
Most of Lusztig's proofs translate almost literally to our more
general setup.

Let $Z \in \Irr(\Pi(\br))$.
Then for each $f \in \tM(\Pi)$ there exists a
unique $c \in \Z$ such that $f^{-1}(c) \cap Z$ contains a
dense open subset of $Z$. 
The map $f \mapsto c$ yields a linear map
$$
\rho_Z\df \tM(\Pi) \to \Z.
$$

\begin{Lem}\label{lem:Lu2.4}
Let $ Z \in \Irr(\Pi(\br))^\mx$. 
There exists some $\wti{f}_Z \in \tM(\Pi)_\br$ such that
for each $Z' \in \Irr(\Pi(\br))^\mx$ we have
$$
\rho_{Z'}(\wti{f}_Z) =
\begin{cases}
1 & \text{if $Z = Z'$},\\
0 & \text{otherwise}.
\end{cases}
$$
\end{Lem}

\begin{proof}
We argue by induction on $\br = (r_1,\ldots,r_n)$. 
When $\br = 0$, the result is trivial. 
Hence we may assume that $\br \not= 0$ and that the result is known 
for all smaller rank vectors. 
(This is the \emph{first induction hypothesis}.)
For our $\br$ we fix $i \in I$ and we shall prove the following:
\begin{itemize}

\item[(a)] 
The lemma holds for any $Z \in \Irr(\Pi(\br))^\mx$ such that 
$\vph_i^*(Z) > 0$.

\end{itemize}
We argue by descending induction on $\vph_i^*(Z)$. 
Since $\vph_i^*(Z) \le r_i$,
we may assume that 
$\vph_i^*(Z) = p > 0$ and that (a) holds for any 
$\wti{Z} \in \Irr(\Pi(\br))^\mx$ 􏱰such that 
$\vph_i^*(\wti{Z}􏱰)> p$.
(Thus is the \emph{second induction hypothesis}.)
 
Note that 
$$
Z^{i,p} := Z \cap \Pi(\br)^{i,p}
$$ 
is open and dense in $Z$. 
Using the results in Section~\ref{sec:bundles}, we 
get that $Z^{i,p} \in \Irr(\Pi(\br)^{i,p})^\mx$.

By Lemma~\ref{lusztig2}, $Z^{i,p}$ corresponds to some
$Z_1 \in \Irr(\Pi(\bs)^{i,0})$
with $\bs := \br - p\alpha_i$.
$$
\xymatrix{
& Y \ar[dl]_{p'} \ar[dr]^{p''}\\
\Pi(\bs)^{i,0} \times J_0 && \Pi(\br)^{i,p}
}
$$
Let $\ov{Z}_1$􏱮 be the closure of $Z_1$ in $\Pi(\bs)$. 
Theorem~\ref{thm:dimbound} implies that  $\ov{Z}_1 \in \Irr(\Pi(\bs))^\mx$.

By the first induction hypothesis, there exists 
$\wti{g} \in \tM(\Pi)_\bs$ such that 
$\rho_{\ov{Z}_1}(\wti{g}) = 1$ and $\rho_{Z_2}(\wti{g}) = 0$ 
for any $Z_2 \in  \Irr(\Pi(\bs))^\mx \setminus \{ \ov{Z}_1 \}$.
In other words, we have
$$
\wti{g} = \wti{f}_{\ov{Z}_1}.
$$

For each $M \in \Pi(\br)^{i,p}$ there is a uniquely determined submodule
$U$ of $M$ such that $M/U \cong E_i^p$.
We obviously have $U \in \Pi(\bU)^{i,0}$ for some locally free
$\bU = (U_1,\ldots,U_n)$ with $\rk(\bU) = \bs$.
We identify $\Pi(\bU)$ and $\Pi(\bs)$ and consider $U$ as an element
in $\Pi(\bs)$.

Let 
$$
\wti{g}^{i,p}\df \Pi(\br)^{i,p} \to \Z
$$ 
be defined by 
$\wti{g}^{i,p}(M) := \wti{g}(U)$.

Let 
$$
\wti{f} :􏱰= \wti{g} * 1_{E_i^p} \in \tM(\Pi)_\br.
$$ 

From the definitions we see that
\begin{itemize}

\item[(b)] 
$\wti{f}|_{\Pi(\br)^{i,p}}􏱰 = \wti{g}^{i,p}$;

\item[(c)]
If $\wti{f}􏱰(M) \not= 0$ for some $M \in \Pi(\br)$, then 
$M \in \Pi(\br)^{i,\bp'}$ for some partition $\bp'$ with
$\bp'(c_i) \ge p$. 

\end{itemize}

Using (b) and the definitions we see that 
$\rho_Z(\wti{f}􏱰) = 1$ and $\rho_{Z'}(\wti{f}) = 0$ for all
$Z' \in \Irr(\Pi(\br))^\mx \setminus \{ Z \}$ such that 
$\vph_i^*(Z') = p$. 

Using (c), we see that $\rho_{Z'}(\wti{f}) = 0$ for all 
$Z' \in \Irr(\Pi(\br))^\mx$ such that $\vph_i^*(Z') < p$. 
By the second induction hypothesis, for all $Z' \in \Irr(\Pi(\br))^\mx$
such that $\vph_i^*(Z') > p$ we can find a function 
$\wti{f}_{Z'} \in \tM(\Pi)$ such that $\rho_{Z'}(\wti{f}_{Z'}) = 1$ 
and $\rho_{\wti{Z}}(\wti{f}_{Z'}) = 0$ for any 
$\wti{Z}􏱰 \in \Irr(\Pi(\br))^\mx \setminus \{ Z' \}$. 

Let
$$
\wti{f}_Z := \wti{f} - \sum_{Z'}􏱰 \rho_{Z'}(\wti{f})\wti{f}_{Z'}
$$
where $Z'$ runs over all irreducible components in 
$\Irr(\Pi(\br))^\mx$ with $\vph_i^*(Z') > p$.

We have $\wti{f}_Z \in \tM(\Pi)$. It is clear that $\wti{f}_Z$ satisfies the 
requirements of the lemma. 
Thus (a) is proved (assuming the first induction hypothesis). 
Now, by Lemma~\ref{lusztig0}(c) we know that any $Z  \in \Irr(\Pi(\br))^\mx$ satisfies 
$\vph_i^*(Z) > 0$ for some $i$. 
Hence the lemma holds for $Z$ (assuming the first induction hypothesis). 
This provides the induction step. 
The lemma is proved.
\end{proof}

Let us stress that the inductive construction of the maps $\wti{f}_Z$ in the proof
of Lemma~\ref{lem:Lu2.4} involves the choice of some $i$ with $\vph_i^*(Z) > 0$.

\begin{Thm}\label{thm:Lu2.6}
For each $\br \in \N^n$ we have 
$$
\dim(U(\n)_\br) = |\Irr(\Pi(\br))^\mx|.
$$
\end{Thm}

\begin{proof}
This follows from our geometric realization of the crystal graph $B(-\infty)$ (see Theorem~\ref{thm:3.16})
combined with the ground breaking results in \cite{K2}.
\end{proof}

Recall that 
$$
\cB = \bigsqcup_{\br \in \N^n} \Irr(\Pi(\br))^\mx.
$$
Slightly rephrasing Lemma~\ref{lem:Lu2.4}, we proved the following theorem.

\begin{Thm}\label{thm:semican1}
The convolution algebra $\tM(\Pi)$ contains 
a set 
$$
\wti{\cS} := \{ \wti{f}_Z \mid Z \in \cB \}
$$ 
of constructible functions 
such that for each $Z' \in \cB$ we have
$$
\rho_{Z'}(\wti{f}_Z) =
\begin{cases}
1 & \text{if $Z = Z'$},\\
0 & \text{otherwise}.
\end{cases}
$$
\end{Thm}

Recall that $\Serre$ is the ideal in $\tM(\Pi)$ generated by the elements
$\ttheta_{ij}$ with $c_{ij} \le 0$, and that 
$$
\cM(\Pi) := \tM(\Pi)/\Serre.
$$
As mentioned in Section~\ref{subsec:convolution}, 
the convolution algebra
$\tM(\Pi)$ is a Hopf algebra with comultiplication 
$\tM(\Pi) \to \tM(\Pi) \otimes \tM(\Pi)$ defined by $\ttheta_i \mapsto \ttheta_i \otimes 1 + 1 \otimes \ttheta_i$.
Furthermore, $\tM(\Pi)$ is isomorphic to the universal enveloping
algebra $U(\cP(\tM(\Pi)))$ of the Lie algebra $\cP(\tM(\Pi))$ 
of primitive elements in $\tM(\Pi)$.

The surjective algebra homomorphism $\tM(\Pi) \to \cM(\Pi)$ defined by
$\ttheta_i \mapsto \theta_i$ yields a Hopf algebra structure on $\cM(\Pi)$
with comultiplication defined by $\theta_i \mapsto \theta_i \otimes 1 + 1 \otimes \theta_i$.

\subsection{Proof of Theorem~\ref{introthm:main}}
Let 
\begin{align*}
\cB_r &:= \Irr(\Pi(\br))^\mx,
\\
\wti{\cS}_\br &:= \wti{\cS}(C,D)_\br := \{ \wti{f}_Z \mid Z \in \cB_\br \} \subset
\tM(\Pi)_\br,
\\
\cS_\br &:= \cS(C,D)_\br := \{ f_Z \mid Z \in \cB_\br \} \subset \cM(\Pi)_\br
\text{ where } f_Z := \wti{f}_Z + \Serre.
\end{align*}
We have disjoint unions
$$
\cB = \bigcup_{\br \in \N^n} \cB_\br,
\qquad
\wti{\cS} := \wti{\cS}(C,D) = \bigcup_{\br \in \N^n} \wti{\cS}_\br,
\qquad
\cS := \cS(C,D) = \bigcup_{\br \in \N^n} \cS_\br.
$$

\begin{Thm}\label{thm:mainproof}
Assume that Conjecture~\ref{mainconj} is true.
For $\Pi = \Pi(C,D)$, $\n = \n(C)$ and $\cS = \cS(C,D)$ the following hold:
\begin{itemize}

\item[(i)]
There is a Hopf algebra isomorphism
$$
\eta_\Pi\df U(\n) \to \cM(\Pi)
$$
defined by $e_i \mapsto \theta_i$.

\item[(ii)]
Via the isomorphism $\eta_\Pi$, 
$\cS_\br$
is a $\C$-basis of $U(\n)_\br$, and
$\cS$ is a $\C$-basis of $U(\n)$.

\item[(iii)]
For $0 \not= f \in \tM(\Pi)$ the following are equivalent:
\begin{itemize}

\item[(a)]
$f \in \Serre$.

\item[(b)]
$f$ has non-maximal support.

\end{itemize}

\end{itemize}
\end{Thm}

\begin{proof}
There is a surjective algebra homomorphism 
$$
\eta_\Pi\df U(\n) \to \cM(\Pi)
$$
defined by $e_i \mapsto \theta_i$.
(Dividing $\tM(\Pi)$ by the ideal $\Serre$ forces the algebra generators
$\theta_i$ of $\cM(\Pi)$ to satisfy the Serre relations.)
It is also clear that $\eta_\Pi$ induces a surjective $K$-linear map
$$
\eta_{\Pi,\br}\df U(\n)_\br \to \cM(\Pi)_\br.
$$
As an immediate consequence of Theorem~\ref{thm:semican1}, the set
$\wti{\cS}_\br$ is linearly independent in $\tM(\Pi)_\br$.
Theorem~\ref{thm:3.16} implies that
$$
|\wti{\cS}_\br| = \dim U(\n)_\br.
$$

Assume that 
$$
f := \sum_{Z \in \cB_\br} \lambda_Z f_Z = 0
$$
for some $\lambda_Z \in K$.
It follows that
$$
\sum_{Z \in \cB_\br} \lambda_Z \wti{f}_Z \in \Serre.
$$
By our assumption that Conjecture~\ref{mainconj} holds, 
it follows that $\lambda_Z = 0$ for all $Z$.

It follows that the set $\cS_\br$ is linearly independent in $\cM(\Pi)_\br$.
So for dimension reasons, 
$$
\eta_{\Pi,\br}\df U(\n)_\br \to \cM(\Pi)_\br
$$
is an isomorphism of $\C$-vector spaces, and therefore
$\eta_\Pi$ is an algebra isomorphism.

It also follows that $\cS_\br$ is a 
$\C$-basis of $U(\n)_\br$, and $\cS$ is a $\C$-basis of $U(\n)$.
Thus we proved (ii).

As a $K$-vector space we get a direct sum decomposition
$$
\tM(\Pi)_\br = \cU_\br \oplus \Serre_\br
$$
where $\cU_\br$ is the subspace generated by $\wti{\cS}_\br$.
Each function in $\cU_\br$ has maximal support,
and by our assumption that Conjecture~\ref{mainconj} holds, each function in $\Serre_\br$ has non-maximal support.

Clearly, for  each sum $h := f + g$ with $f \in \cU_\br$
and $g \in \Serre_\br$, we have that $h$ has non-maximal support 
if and only if $f = 0$.
This finishes the proof of (iii).

The enveloping algebra $U(\n)$ is a Hopf algebra with comultiplication $U(\n) \to U(\n) \otimes U(\n)$ defined by $e_i \mapsto e_i \otimes 1 + 1 \otimes e_i$.
The algebra isomorphism 
$\eta_\Pi\df U(\n) \to \cM(\Pi)$ 
is obviously a Hopf algebra isomorphism.
This finishes the proof of (i).
\end{proof}

For $\Pi = \Pi(C,D)$ and $\n = \n(C)$
we call $\cS = \cS(C,D)$ the \emph{semicanonical basis} of $U(\n)$.
For $C$ symmetric and $D$ the identity matrix, $\cS$ coincides
with Lusztig's semicanonical basis of $U(\n)$.

\begin{Prop}\label{prop:unique}
Assume that Conjecture~\ref{mainconj} is true.
Let $\wti{\cS} = \{ \wti{f}_Z \mid Z \in \cB \}$ and
$\wti{\cG} = \{ \wti{g}_Z \mid Z \in \cB \}$
be subsets of $\tM(\Pi)$ satisfying
$$
\rho_{Z'}(\wti{f}_Z) = \rho_{Z'}(\wti{g}_Z) =
\begin{cases}
1 & \text{if $Z = Z'$},\\
0 & \text{otherwise}
\end{cases}
$$
for all $Z,Z' \in \cB$.
Then $\wti{f}_Z - \wti{g}_Z \in \Serre$.
\end{Prop}

\begin{proof}
By definition we have
$$
\rho_{Z'}(\wti{f}_Z-\wti{g}_Z) = 0
$$
for all $Z' \in \cB$.
This implies
$\dim \supp(\wti{f}_Z-\wti{g}_Z) < \dim H(\br)$
for all $Z \in \cB_\br$.
By Theorem~\ref{thm:mainproof}(iii) we get 
$\wti{f}_Z-\wti{g}_Z \in \Serre$.
\end{proof}

\subsection{Semicanonical bases for irreducible integrable highest
weight modules}
Let $\Pi = \Pi(C,D)$, $\g = \g(C)$, $\n = \n(C)$ and $\cB$ be defined as 
before.
Assume that Conjecture~\ref{mainconj} is true.

Recall that
for $\la \in P^+$ a dominant integral weight, 
$V(\lambda)$ denotes the irreducible integrable highest 
weight $\g$-module with highest weight $\la$. 

In view of Theorem~\ref{thm:mainproof}, we can then identify $\cM(\Pi)$ with $U(\n)$, and we consider the semicanonical basis $\cS = \cS(C,D)$ of
$\cM(\Pi)$ as a basis of $U(\n)$.

Let $\la \in P^+$ be a dominant integral weight. 
Fix a highest weight vector
$v_\la \in V(\la)$. Furthermore, let $x \mapsto x^-$ denote the algebra automorphism of $U(\g)$  defined by
\[
e_i^- := f_i,\quad f_i^- := e_i,\quad h^- := -h,\qquad (i\in I,\; h \in \h).
\]
We then have a surjective homomorphism of $U(\n)$-modules 
$$
\pi_\la\df U(\n) \to V(\la)
$$ 
defined by $x \mapsto x^-v_\la$.

\begin{Prop}
Assume that Conjecture~\ref{mainconj} is true.
For each $\lambda \in P^+$ the following hold:
\begin{itemize}

\item[(i)]  
$\pi_\la(f_Z) = 0$ if and only if $Z \not \in \cB_\la^*$.

\item[(ii)]  
$\cS_\la := \{ \pi_\la(f_Z) \mid Z \in \cB_\la^* \}$ is a basis of $V(\la)$.

\end{itemize}
\end{Prop}

\begin{proof}
This is similar to \cite[Section 3]{L2}, so we will only sketch the argument.
It follows from the proof of Lemma~\ref{lem:Lu2.4} that 
for every $Z\in\cB$,
\[
f_Z e_i^p/p! = f_{Z'} + \sum_{Z''} \mu_{Z''} f_{Z''},
\]
where $Z' = (\tilde{e}_i^*)^p(Z)$, and the sum is over $Z''$ with $\varphi_i^*(Z'') > \varphi_i^*(Z')$.
(The function $1_{E_i^p}$ in $\cM(\Pi)$ corresponds to $e_i^p/p!$ in
$U(\n)$.)
This implies that the left ideal $U(\n)e_i^p$ is contained in the subspace spanned by
$\{f_Z \mid \varphi_i^*(Z) \ge p\}$. 
More generally, if $d = (d_i) \in \N^I$, we have
\[
\sum_{i\in I} U(\n)e_i^{d_i} \subseteq \cW_d:= \Span\{ f_Z \mid \varphi_i^*(Z) \ge d_i \mbox{ for some } i\in I \}.
\]

Conversely, consider $f_Z \in \cS$ such that $\varphi_i^*(Z) = p$. 
Using again the proof of Lemma~\ref{lem:Lu2.4}, we get that
\[
f_Z = f_{Z'}e_i^p/p! + \sum_{Z''} \nu_{Z''} f_{Z''},
\]
where $Z' = (\tilde{f}_i^*)^p(Z)$, and the sum is over $Z''$ with $\varphi_i^*(Z'') > \varphi_i^*(Z)$.
Using descending induction on $p$, it follows that  
$$
\sum_{i\in I} U(\n)e_i^{d_i} \supseteq \cW_d.
$$
Hence the left ideal $\sum_{i\in I} U(\n)e_i^{d_i}$ coincides with the subspace
$\cW_d$ spanned by a subset of $\cS$.

Now it is known that 
$$
\Ker(\pi_\la) = \sum_{i\in I} U(\n)e_i^{a_i+1}.
$$ 
Therefore $\Ker(\pi_\la) =\cW_{d}$ with $d=(a_i+1)$, that is,
$$
\Ker(\pi_\la) = \Span\{f_Z \mid Z \not \in B_\la^*\}
$$ 
and the proposition follows.
\end{proof}


\section{Examples}\label{sec:examples}


\subsection{Maximal irreducible components for the Dynkin 
cases}\label{subsec:maxcompDynkin}
Let $\Pi = \Pi(C,D) = \Pi(C,D,\Omega)$.
Let $Q = Q(C,\Omega) = (I,Q_1,s,t)$ be the full subquiver of 
$\ov{Q}(C) = (I,\ov{Q}_1,s,t)$
with arrow set
$$
Q_1 = \{ \alpha_{ij}^{(g)} \in \ov{Q}_1 \mid (i,j) \in \Omega, 1 \le g \le g_{ij} \}
\cup \{ \vep_i \mid i \in I\}. 
$$
Let $H = H(C,D,\Omega)$ be the subalgebra of $\Pi$ given by
$Q(C,\Omega)$.
Thus we have
$$
H = KQ/J
$$ 
where $KQ$ is the path algebra of $Q$ and
$J$ is the ideal defined by the following
relations:
\begin{itemize}

\item[(H1)]
For each $i$ we have the \emph{nilpotency relation}
\[
\vep_i^{c_i} = 0.
\]

\item[(H2)]
For each $(i,j) \in \Omega$ and each $1 \le g \le g_{ij}$ we have
the \emph{commutativity relation}
\[
\vep_i^{f_{ji}}\alpha_{ij}^{(g)} = \alpha_{ij}^{(g)}\vep_j^{f_{ij}}.
\]

\end{itemize}

There is an obvious embedding $\rep(H) \to \rep(\Pi)$.
Thus each $H$-module can be seen as a $\Pi$-module.
Let $TC^+\df \rep(H) \to \rep(H)$ denote the twisted Coxeter functor defined in \cite{GLS3}.

As before, 
for a dimension vector $d$ let
$\rep(H,d)$ and $\rep(\Pi,d)$ be the varieties of
representation of $H$-modules
$\Pi$-modules with dimension vector $d$, respectively.

Let $\rep_\vp(H,d) \subseteq \rep(H,d)$ and $\rep_\vp(\Pi,d) \subseteq \rep(\Pi,d)$ denote the subvarieties of 
locally free modules.

Let 
$$
\pi_H\df \rep(\Pi,d) \to \rep(H,d)
$$ 
be the obvious restriction map.

\begin{Prop}\label{fibredim1}
For each $M \in \rep_\vp(H,d)$ we have
$$
\pi_H^{-1}(M) \cong \Hom_H(M,TC^+(M)).
$$
\end{Prop}

\begin{proof}
Using  
\cite{GLS3} one can adapt the construction in \cite{R} to obtain the result.
\end{proof}

Recall from \cite{GLS3} that for all $M \in \rep_\vp(H)$ 
we have a functorial isomorphism 
$$
TC^+(M) \cong {\rm D}\Ext_H^1(M,{_H}H) \cong \tau_H(M).
$$

\begin{Prop}\label{fibredim2}
For $M \in \rep_\vp(H,d)$ we have
$$
\dim \pi_H^{-1}(\cO(M)) = \dim \rep_\vp(H,d).
$$
\end{Prop}

\begin{proof}
The proof is based in Proposition~\ref{fibredim1}.
For $M \in \rep_\vp(H,d)$ we have
\begin{align*}
\dim \cO(M) + \dim \Hom_H(M,TC^+(M)) 
&= \dim \cO(M) + \dim \Ext_H^1(M,M)\\
&= \sum_{i \in I} d_i^2 - \dim \End_H(M) + \dim \Ext_H^1(M,M)\\
&= \sum_{i \in I} d_i^2 - \sum_{i \in I} c_ia_i^2 + \sum_{(j,i) \in \Omega}
c_i|c_{ij}|a_ia_j\\
&= \dim \rep_\vp(H,d).
\end{align*}
Here $(a_1,\ldots,a_n)$ is the rank vector of $M$
The first equality follows since $TC^+(M) \cong \tau_H(M)$ for $M \in \rep_\vp(H)$ and
by the Auslander-Reiten formulas.
The second equality is just the general formula for orbit dimensions in representation
varieties of algebras, 
the third equality holds by \cite[Proposition~4.1]{GLS3}
and the last equality follows from \cite[Proposition~3.1]{GLS4}.
The result follows.
\end{proof}

Assume now that $C$ is of Dynkin type.
We assume also that the orientation $\Omega$ is \emph{acyclic},
i.e. that
for each sequence $((i_1,i_2),(i_2,i_3),\ldots,(i_t,i_{t+1}))$ with
$t \ge 1$ and $(i_s,i_{s+1}) \in \Omega$ for all $1 \le s \le t$ we have
$i_1 \not= i_{t+1}$.

For each positive root $\alpha \in \Delta^+(C)$ there
is a (unique) indecomposable preprojective
$H$-module $M_\alpha$ with $\rk(M_\alpha) = \alpha$, see \cite{GLS3}.
For a \emph{Kostant partition}
$\nu = (n_\alpha) \in \N^{\Delta_+(C)}$ let
$$
M_\nu := \bigoplus_{\alpha \in \Delta^+(C)}
M_\alpha^{n_\alpha}
$$
be the preprojective $H$-module associated with $\nu$, and let
$$
d(\nu) := \sum_{\alpha \in \Delta^+(C)} n_\alpha \dimv(M_\alpha).
$$ 
Furthermore, set
\[
Z_\nu := \overline{\pi_H^{-1}(\cO(M_\nu))} 
\subseteq \rep(\Pi,d(\nu)). 
\]

\begin{Lem}\label{lem:componentsdynkin1}
Let $\Pi = \Pi(C,D)$ and $H = H(C,D,\Omega)$.
For each Kostant partition $\nu \in \Delta^+(C)$ we have
$$
Z_\nu \in \Irr(\nil_E(\Pi,d(\nu)))^\mx.
$$
\end{Lem}

\begin{proof}
By definition we have $Z_\nu \subseteq \rep(\Pi,d(\nu))$.
We know that
$$
\dim(Z_\nu) = \dim \rep(H,d(\nu)).
$$
It remains to show that each $X \in Z_\nu$ is $E$-filtered.

For brevity let $F := TC^+$, where $T$ is the twist functor and
$C^+$ is the Coxeter functor, see \cite{GLS3}.
We know that the category $\rep(\Pi)$ can be identified with the
category of $H$-module homomorphisms $f\df M \to F(M)$.
For $M \in \rep(H)$ we have $M \cong (0\df M \to F(M))$.
Given such an $f$ let $(M,f)$ be the corresponding $\Pi$-module.

Now assume that $M = M_\nu$ is a preprojective $H$-module.
Thus we have 
$$
M \cong \bigoplus_{\alpha \in \Delta^+(C)} M_\alpha^{n_\alpha}
$$
for some $n_\alpha \ge 0$.
There exists some $\beta$ with $n_\beta \not= 0$ such that
$\Hom_H(M_\beta,\tau_H(M)) = 0$.
It follows that $0\df M_\beta^{n_\beta} \to F(M_\beta^{n_\beta})$
is a submodule of $(0\df M \to F(M)) \cong M$ with factor module of the form
$f\df M/M_\beta^{n_\beta} \to F(M/M_\beta^{n_\beta})$.
The $\Pi$-module $0\df M_\beta^{n_\beta} \to F(M_\beta^{n_\beta})$
is $E$-filtered, since $M_\beta$ is $E$-filtered.
Now the result follows by induction.
\end{proof}

\begin{Thm}\label{thm:componentsdynkin2}
Let $\Pi = \Pi(C,D)$ with $C$ of Dynkin type.
For $Z \in \Irr(\nil_E(\Pi,d))$ the following are equivalent:
\begin{itemize}

\item[(i)]
$Z$ is maximal.

\item[(ii)]
$Z = Z_\nu$ for some Kostant partition $\nu = (n_\alpha) \in \N^{\Delta^+(C)}$ with $d(\nu) = d$.

\end{itemize}
\end{Thm}

\begin{proof}
Let $M_\nu$ be a preprojective $H$-module in the sense of \cite{GLS3},
and let $\br = \rk(M_\nu)$.
By Lemma~\ref{lem:componentsdynkin1} we have
$$
Z_\nu \in \Irr(\nil_E(\Pi,d(\nu)))^\mx.
$$
For preprojective $H$-modules $M_\nu$ and $N_\mu$ we clearly have 
$Z_\nu = Z_\mu$ if and only if $M_\nu \cong M_\mu$.
By our geometric realization of $B(-\infty)$ we know that
$$
\dim U(\n)_\br = |\Irr(\nil_E(\Pi,d))^\mx|.
$$
Furthermore, 
the number of isomorphism classes of preprojective
$H$-modules $M$ with $\rk(M) = \br$ is exactly $\dim U(\n)_\br$.
This follows from \cite[Section~11.2]{GLS3}.
This finishes the proof.
\end{proof}

\subsection{Type $B_2$}
\label{sec:exampleB2}

\subsubsection{The preprojective algebra of type $B_2$}
For the whole Section~\ref{sec:exampleB2}, let
$\Pi = \Pi(C,D) = \Pi(C,D,\Omega)$ with
\[
C = \left(\bbm 2&-1\\-2&2 \ebm\right)
\text{\;\;\; and \;\;\;}
D = \left(\bbm 2&0\\0&1 \ebm\right)
\]
and $\Omega = \{ (1,2) \}$.
Set $H = H(C,D,\Omega)$ and $H^* = H(C,D,\Omega^*)$.
Thus $C$ is a Cartan matrix of Dynkin type $B_2$, and the symmetrizer $D$
is minimal.
We have $\Pi = K\ov{Q}/\ov{I}$ where
$\ov{Q} = \ov{Q}(C)$ is the quiver 
\[
\xymatrix{
1 \ar@(ul,dl)_{\vep_1} \ar@<0.5ex>[r]^{\alpha_{21}}& 2 \ar@<0.5ex>[l]^{\alpha_{12}}
}
\]
and $\ov{I}$ is generated by the set
$$
\{ \vep_1^2,\; \alpha_{12}\alpha_{21}\vep_1 + \vep_1\alpha_{12}\alpha_{21},\; -\alpha_{21}\alpha_{12}\}.
$$
Thus $\Pi$ is a finite-dimensional
special biserial algebra.
The modules and the AR-quiver of a special biserial algebra can be determined 
combinatorially, see for example \cite{BR}.
The indecomposable $\Pi$-modules are either projective-injective,
or string modules, or band modules.
The band modules are locally free, but they are not $E$-filtered.

The indecomposable projective $\Pi$-modules are shown in
Figure~\ref{Fig:projPB2}. 
(The arrows indicate when an arrow of the
algebra $\Pi$ acts with a non-zero scalar on a basis vector.)
\begin{figure}[!htb]
$$
\xymatrix@-0.7pc{ 
&1 \ar[dl]_{\alpha_{21}}\ar[dr]^{\vep_1}
&&&&
2 \ar[d]^{\alpha_{12}}
&&&&
\\
2 \ar[d]_{\alpha_{12}} && 1 \ar[d]^{\alpha_{21}} 
&&&
1 \ar[d]^{\vep_1} 
&&&&
\\
1 \ar[dr]_{\vep_1} && 2 \ar[dl]^{\alpha_{12}} 
&&&
1 \ar[d]^{\alpha_{21}}
&&&&
\\
&1 
&&&& 
2
&&&&
}
$$
\caption{
The indecomposable projective $\Pi(C,D)$-modules for type $B_2$.}
\label{Fig:projPB2}
\end{figure}
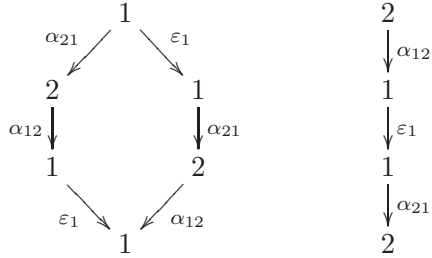

All results in Section~\ref{sec:exampleB2} can be proved by using the classification of
finite-dimensional indecomposable $\Pi$-modules.

\subsubsection{Irreducible components}\label{subsub:comp}
Up to isomorphism there are $8$ indecomposable rigid
$\Pi$-modules, namely
\begin{align*}
{P_1 = \bsm &1\\2&&1\\1&&2\\&1\esm}, &&
{P_2 = \bsm 2\\1\\1\\2\esm}, &&
{E_1 = \bsm 1\\1\esm}, && 
{E_2 = \bsm 2\esm}, 
\\
{T_1 = \bsm 1\\1\\2\esm}, && 
{T_2 = \bsm 2\\1\\1\esm} ,  &&
{T_3 = \bsm &1\\2&&1\\&&2\esm}, && 
{T_4 = \bsm 2\\1&&2\\&1\esm}.
\end{align*}
The following is a complete list of basic maximal rigid $\Pi$-modules:
\begin{align*}
P_1 \oplus P_2 \oplus E_1 \oplus T_1, && 
P_1 \oplus P_2 \oplus E_1 \oplus T_2, &&
P_1 \oplus P_2 \oplus E_2 \oplus T_3, 
\\
P_1 \oplus P_2 \oplus E_2 \oplus T_4, &&
P_1 \oplus P_2 \oplus T_1 \oplus T_3, &&
P_1 \oplus P_2 \oplus T_2 \oplus T_4.
\end{align*}
Let $R_1 \oplus R_2 \oplus R_3 \oplus R_4$ be one of these modules.
Then $R_1^{a_1} \oplus R_2^{a_2} \oplus R_3^{a_3} \oplus R_4^{a_4}$
is a rigid $\Pi$-module for all $a_1,a_2,a_3,a_4 \ge 0$, and
we obtain all rigid $\Pi$-modules in this way.

For $Z \in \Irr(\nil_E(\Pi,d))$ the following are equivalent:
\begin{itemize}

\item[(i)]
$Z$ is maximal.

\item[(ii)]
$Z = \ov{\cO(R)}$ with $R \in \rep_\vp(\Pi,d)$ rigid.

\end{itemize}

Recall that the dimension of an orbit $\cO(M)$ for 
$M \in \rep_\vp(\Pi,d)$ can be computed by the formula
$$
\dim \cO(M) = \dim G(d) - \dim \End_\Pi(M).
$$
For modules of small dimension, it is an easy exercise to compute
$\dim \End_\Pi(M)$.

We have $H = KQ/I$ where $Q = Q(C,\Omega)$ is the quiver
\[
\xymatrix{
1 \ar@(ul,dl)_{\vep_1} & 2 \ar[l]_{\alpha_{12}} 
}
\]
and $I$ is generated by $\{ \vep_1^2 \}$.
The indecomposable locally free $H$-modules are
\begin{align*}
E_1 &= {\bsm 1\\1\esm}, &
E_2 &= {\bsm 2 \esm}, &
T_2 &= {\bsm 2\\1\\1 \esm}, &
T_4 &=  {\bsm 2\\1&&2\\&1 \esm}, &
X_1 &=  {\bsm 1&&2\\&1\esm},
\end{align*}
and the  indecomposable locally free $H^*$-modules (apart from $E_1$ and $E_2$) are
\begin{align*}
T_1 &= {\bsm 1\\1\\2 \esm}, &
T_3 &= {\bsm &1\\2&&1\\&&2 \esm}, &
X_2 &= {\bsm &1\\2&&1 \esm}.
\end{align*}
Furthermore, we define 
certain indecomposable locally free $\Pi$-modules:
\begin{align*}
X\df &
{\xymatrix@-0.7pc{
& 1 \ar[r]\ar[d] & 1
\\
& 2 \ar[d]
\\
1 \ar[r] &1
}} &
M(\lambda)\df &
{\xymatrix@-1.0pc{
1 \ar[dd]\ar[dr]^{\lambda}
\\
& 2 \ar[dl]
\\
1 
}}
\end{align*}
where $\lambda \in K^*$.
The module $X$ is obviously $E$-filtered.
The modules $M(\lambda)$ are band modules sitting at the bottom 
of a $K^*$-family of $1$-tubes in the Auslander-Reiten quiver
of $\Pi$.
Note that none of the modules $X_1$, $X_2$, $X$, $M(\lambda)$ is a crystal module.

Using \cite{CBS} it is possible to determine all irreducible components
of $\nil_E(\Pi,d)$ for all $d$.
Here we just discuss one example.
Let $d = (4,1)$.
We have $\dim G(d) = 17$ and $\dim G(d) - q_{DC}(d/D) = 12$.
There are three locally free $H$-modules (up to isomorphism) 
with dimension vector $d$:
\begin{align*}
M_1 &= {\bsm 1\\1\esm \oplus \bsm 1\\1\esm \oplus \bsm2\esm}, &
M_2 &= {\bsm 1\\1\esm \oplus \bsm 2\\1\\1\esm}, &
M_3 &= {\bsm 1\\1\esm \oplus \bsm 1&&2\\&1\esm}.
\end{align*}
Denote by $Z_{M_1}$, $Z_{M_2}$, $Z_{M_3}$, respectively, the closures of the preimages of their orbits under 
$$
\pi_H\df \rep_\vp(\Pi,d) \to \rep_\vp(H,d).
$$
We have
\[
\rep_\vp(\Pi,d) = Z_{M_1} \cup Z_{M_2} \cup Z_{M_3}, 
\]
where
\begin{align*}
Z_{M_1} &= \cO(T_1\oplus E_1) \sqcup \cO(X_2\oplus E_1) \sqcup \cO(E_1^2 \oplus E_2) = \rep(H^*,d),
\\
Z_{M_2} &= \cO(T_2\oplus E_1) \sqcup \cO(X_1\oplus E_1) \sqcup 
\cO(E_1^2\oplus E_2) = \rep(H,d),
\\
Z_{M_3} &= \bigsqcup_{\la\in K*} \cO(M(\la) \oplus E_1) \sqcup \cO(X) \sqcup \cO(X_1\oplus E_1) \sqcup \cO(X_2\oplus E_1) \sqcup
\cO(E_1^2 \oplus E_2).
\end{align*}
The orbits $\cO(T_1\oplus E_1)$ and $\cO(T_2\oplus E_1)$ have dimension 12, 
and we have 
$$
Z_{M_1} = \ov{\cO(T_1 \oplus E_1)}
\text{\;\;\; and \;\;\;}
Z_{M_2} = \ov{\cO(T_2 \oplus E_1)}.
$$ 
Each orbit $\cO(M(\la)\oplus E_1)$ has dimension 11, so their union also has dimension 12, and we have
$$
Z_{M_3} = \ov{\bigsqcup_{\la\in K^*} \cO(M(\la) \oplus E_1)}.
$$
Hence
\[
\dim Z_{M_1}=\dim Z_{M_2}=\dim Z_{M_3} = \dim(\rep_\vp(H,d)) = 12. 
\]
The orbit $\cO(X)$ as dimension 11.
We have
\[
\nil_E(\Pi,d) = Z_{M_1}\cup Z_{M_2} \cup Z_{M_3}', 
\]
where
\[
Z_{M_3}' =\cO(X) \sqcup \cO(X_1\oplus E_1)\sqcup\cO(X_2\oplus E_1)\sqcup\cO(E_1^2 \oplus E_2)
\]
is an irreducible component of $\nil_E(\Pi,d)$ of non-maximal dimension 11, and we 
have
$$
Z_{M_3}' = \ov{\cO(X)}.
$$
(The fact that $X$ cannot be contained in any of the components $Z_{M_1}$
or $Z_{M_2}$ can be shown by a simple semicontinuity argument.)
 
We consider now the enveloping algebra $U(\n)$ for type $B_2$.
Then
the dimension of $U(\n)_{(2,1)}$
is two, which is perfectly in line with $\nil_E(\Pi,d)$
having exactly two maximal components.
(The rank vector $(2,1)$ corresponds to the dimension vector
$d = (4,1)$.)

\subsubsection{Semicanonical basis}
Now assume that $K = \C$.
We get 
\[
\frac{1}{2} \ttheta_2 *\ttheta_1 * \ttheta_1 = 1_{Z_{M_1}}, \qquad
\frac{1}{2} \ttheta_1 * \ttheta_1 * \ttheta_2 = 1_{Z_{M_2}}, 
\qquad
\ttheta_1* \ttheta_2 * \ttheta_1 = 
1_{Z_{M_1}} + 1_{Z_{M_2}} + 1_X,
\]
so the Serre relation is verified up to the function $1_X \in \Serre$.

The images of $\frac{1}{2} \ttheta_2 *\ttheta_1 * \ttheta_1$ and 
$\frac{1}{2} \ttheta_1 *\ttheta_1 * \ttheta_2$ in $\cM(\Pi)$ 
form the semicanonical basis $\cS_\br$ of $\cM(\Pi)_\br$,
where $\br = (2,1) = d/D$.
They evaluate to 1 at the generic point of one of the two maximal irreducible components of $\nil_E(\Pi,d)$,
and to 0 at the generic point of the other.

\subsubsection{Examples of constructible functions with non-maximal support}\label{subsub:nonmaxsupp}
Again we assume that $K = \C$.
We define indecomposable $\Pi$-modules $X$, $Y_1$ and $Y_2$ as follows:
$$
\xymatrix@-0.7pc{
X\df & 1 \ar[d]\ar[r] & 1 &&
Y_1\df& 1 \ar[r]\ar[d] & 1 &&
Y_2\df && 2 \ar[d]
\\
& 2 \ar[d] &&&
2 \ar[d]& 2 \ar[d] &&&
& 1 \ar[r]\ar[d] & 1
\\
1 \ar[r] & 1 &&&
1 \ar[r] & 1 &&&
& 2 \ar[d]
\\
&&&&&&&&1\ar[r] & 1
}
$$
An easy calculation shows that
$$
\ttheta_{12} = -2 \cdot 1_X.
$$
We have $\supp(\ttheta_{12}) = \cO(X)$.
Furthermore, one can check that 
$$
\ttheta_{12} * \ttheta_2 = -2(1_{Y_1} + 1_{Y_2} + 1_{X \oplus E_2})
$$
This implies
$$
\supp(\ttheta_{12} * \ttheta_2) = \cO(Y_1) \sqcup \cO(Y_2) \sqcup \cO(X \oplus E_2).
$$
Let $M = P_1 \oplus E_1$.
We have
$$
((\ttheta_{12} * \ttheta_2) * \ttheta_1)(M) = 
\sum_{m \in \C} m \chi(\{ U \subset M \mid M/U \cong E_1, \;
(\ttheta_{12} * \ttheta_2)(U) = m \}). 
$$
One easily sees that $M$ does not have any submodules isomorphic to 
$Y_2$ or $X \oplus E_2$.
Furthermore, one can check that we have isomorphisms of varieties
$$
\{ U \subset M \mid M/U \cong E_1,\;
(\ttheta_{12} * \ttheta_2)(U) = -2 \}
\cong
\{ U \subset M \mid M/U \cong E_1,\; U \cong Y_1 \}
\cong \C^*.
$$
Since $\chi(\C^*) = 0$, we get 
$$
((\ttheta_{12} * \ttheta_2) * \ttheta_1)(M) = 0.
$$
Note that the closure of $\cO(M)$ is a maximal irreducible component.

All three functions $\ttheta_{12}$, $\ttheta_{12} * \ttheta_2$ and 
$\ttheta_{12} * \ttheta_2 * \ttheta_1$ have non-maximal support.
However, our calculation above in a small case like $B_2$ shows that
this is a non-trivial fact which depends on the vanishing of some
Euler characteristic.

As before, we define
$$
\xymatrix@-0.7pc{
X_1\df & 2 \ar[d] &&
T_4\df & 2 \ar[d] & 2 \ar[d]
\\
1 \ar[r] & 1 &&
&1 \ar[r] & 1
}
$$
In $\tF(\Pi)$ we get
$$
1_{X_1} * 1_{E_2} =  1_{T_4} + 2 \cdot 1_{X_1 \oplus E_2}.
$$
The function $1_{X_1}$ has non-maximal support, and
$1_{E_2}$ and $1_{T_4} + 2 \cdot 1_{X_1 \oplus E_2}$ have maximal support.
(But note that $1_{X_1}$ does not belong to $\tM(\Pi)$.)
In particular, in $\tF(\Pi)$ the functions with non-maximal support
do not form an ideal.

\subsubsection{Bundle construction}\label{subsub:bundle}
We keep the notation introduced in Sections~\ref{subsub:comp}
and \ref{subsub:nonmaxsupp}.
We study the bundles
$$
\xymatrix{
& Y \ar[dl]_{p'}\ar[dr]^{p''}
\\
\Pi((2,1))^{2,(0)} \times J_0 && \Pi((2,2))^{2,(1)}.
}
$$
We have
$$
\Irr(\Pi((2,1))^{2,(0)} = Z_1 \cup Z_2
$$
where
$$
Z_1 := \ov{\cO(T_1 \oplus E_1)} \cap \Pi((2,1))^{2,(0)}
\text{\;\;\; and \;\;\;}
Z_2 := \ov{\cO(X)} \cap \Pi((2,1))^{2,(0)}.
$$
The component $Z_1$ is maximal, and $Z_2$ is non-maximal.
We have
\begin{align*}
p''(p')^{-1}(Z_1 \times J_0) &= 
\ov{\cO(P_2 \oplus E_1)} \cap \Pi((2,2))^{2,(1)} \in \Irr(\Pi((2,2))^{2,(1)}),
\\
p''(p')^{-1}(Z_2 \times J_0) &= \ov{\cO(Y_1)} \cap \Pi((2,2))^{2,(1)} \in \Irr(\Pi((2,2))^{2,(1)}).
\end{align*}
We have $\cO(Y_1) \subset \ov{\cO(P_1)}$,
thus $\ov{\cO(Y_1)}$ cannot be in $\Irr(\Pi((2,2)))$.
Furthermore, we get
\begin{align*}
\ov{\cO(P_1)} &\in \Irr(\Pi(2,2))^\mx,
\\
\ov{\cO(P_1)} \cap \Pi((2,2))^{1,(2)} &\in \Irr(\Pi((2,2))^{1,(2)})^\mx,
\\
\ov{\cO(P_1)} \cap \Pi((2,2))^{2,(1)} = \ov{\cO(Y_1)} \cap \Pi((2,2))^{2,(1)} &\in \Irr(\Pi((2,2))^{2,(1)}).
\end{align*}
Next, we study the bundles
$$
\xymatrix{
& Y \ar[dl]_{p'}\ar[dr]^{p''}
\\
\Pi((1,1))^{1,(1)} \times J_0 && \Pi((2,1))^{1,(2,1)}.
}
$$
Then $\ov{\cO(X_1)} = \ov{\cO(X_1)} \cap \Pi((1,1))^{1,(1)} \in \Irr(\Pi((1,1))^{1,(1)}$ and
$$
p''(p')^{-1}(\ov{\cO(X_1)} \times J_0) = \ov{\cO(X)} \cap \Irr(\Pi((2,1))^{1,(2,1)}).
$$
We have $\ov{\cO(X_1)} \notin \Irr(\Pi((1,1)))$ and
$\ov{\cO(X)} \in \Irr(\Pi(2,1))$.

\subsubsection{Crystal graphs and Littlewood-Richardson coefficients}
\label{subsec:LRcoeff}
In Figure~\ref{Fig:crystalB2} we display part
of the geometric crystal graph $(\cB,\tilde{e}_i) \equiv (B(-\infty),\tilde{e}_i)$ of type $B_2$.
(Each box in the figure contains a crystal module over $\Pi$.
The orbit closure of this $\Pi$-module is a maximal irreducible component.)

We have 
\begin{align*}
\alpha_1 &= 2\vpi_1 - 2\vpi_2, & \vpi_1 &= \alpha_1 + \alpha_2,
\\ 
\alpha_2 &= -\vpi_1+2\vpi_2, & \vpi_2 &= 1/2\alpha_1+\alpha_2.
\end{align*}

In Figure~\ref{Fig:crystalB2simplemoduleb} we display
the geometric crystal graph $(\cB_{\vpi_1+\vpi_2}^*,\tilde{e}_i) \equiv (B(\vpi_1+\vpi_2),\tilde{e}_i)$ of the simple representation $V(\vpi_1+\vpi_2)$
over the simple complex Lie algebra $\g$ of type $B_2$,
and we display the geometric crystal graph $(\cB_{2\vpi_2},\tilde{e}_i^*)$.

Set $\la = \vpi_1+\vpi_2$ and $\mu = 2\vpi_2$.
The possible $\nu \in P^+$ with $\la+\mu-\nu \in R^+$
are 
$$
\{ \vpi_1+3\vpi_2,\; 2\vpi_1+\vpi_2,\; 3\vpi_2,\; \vpi_1+\vpi_2,\; 
\vpi_2 \}.
$$
For $\la+\mu-\nu$ we get the elements
$$
\{ 0,\; \alpha_2,\; \alpha_1+\alpha_2,\; \alpha_1+2\alpha_2,\;
2\alpha_1+3\alpha_2 \}.
$$
The components in $\cB_\la^* \cap \cB_\mu$ have a double frame.
We get the tensor product decomposition
\begin{align*}
V(\vpi_1+\vpi_2) \otimes V(2\vpi_2) &\cong 
V(\vpi_1+3\vpi_2) \oplus 
V(2\vpi_1+\vpi_2) \oplus V(3\vpi_2)
\\
&\;\;\;\;\oplus
V(\vpi_1+\vpi_2)^2 \oplus 
V(\vpi_2).
\end{align*}
(The two copies of $V(\vpi_1+\vpi_2)$ in this decomposition come from the fact we have
two irreducible components with rank vector $\alpha_1+2\alpha_2$
in $\cB_\la^* \cap \cB_\mu$.)

\begin{figure}[!htb]
\small
\[
\xymatrix@-1.8ex{
&&&& *+[F]{\bsm 0 \esm} \ar[dl]_1\ar[dr]^2
\\
&&&*+[F]{\bsm 1\\1\esm} \ar[dl]_1\ar[d]_2 
&& *+[F]{\bsm 2\esm} \ar[d]_1\ar[drrr]^2
\\
&&*+[F]{\bsm 1\\1\esm \oplus \bsm 1\\1\esm} \ar[dl]_1\ar[d]_2
&
*+[F]{\bsm 1\\1\\2\esm} \ar[dl]_1\ar[d]_2
&&
*+[F]{\bsm 2\\1\\1\esm} \ar[d]_1\ar[dr]^2
&&&
*+[F]{\bsm 2\esm \oplus \bsm 2\esm} \ar[d]_1\ar[dr]^2
\\
&*+[F]{\bsm 1\\1\esm \oplus \bsm 1\\1\esm \oplus \bsm 1\\1\esm} &
*+[F]{\bsm 1\\1\esm \oplus \bsm 1\\1\\2\esm} &
*+[F]{\bsm &1\\2&&1\\&&2\esm} &&
*+[F]{\bsm 1\\1\esm \oplus \bsm 2\\1\\1\esm} &
*+[F]{\bsm 2\\1\\1\\2\esm}&&
*+[F]{\bsm 2\\1&&2\\&1\esm} &
*+[F]{\bsm 2\esm \oplus \bsm 2\esm \oplus \bsm 2\esm}
}
\]
\caption{The first four layers
of the geometric crystal graph $(\cB,\tilde{e}_i) \equiv 
(B(-\infty),\tilde{e}_i)$ of type $B_2$.}
\label{Fig:crystalB2}
\end{figure}
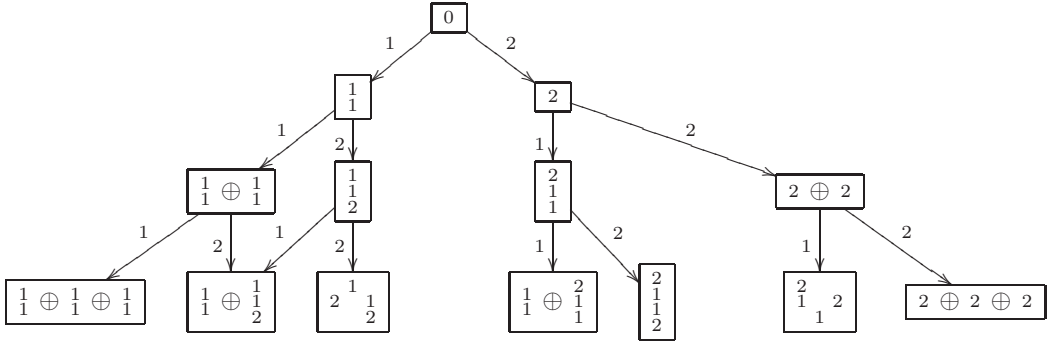

\begin{figure}[!htb]
\small
\[
\xymatrix@-1.8ex{
& \cB_{\vpi_1+\vpi_2}^* 
& *+[F=]{\bsm 0 \esm}\ar[dl]_1\ar[dr]^2 &&&
&\cB_{2\vpi_2}& *+[F=]{\bsm 0 \esm}\ar[d]^2
\\
&*+[F]{\bsm 1\\1 \esm}\ar[d]_2 && *+[F=]{\bsm 2 \esm}\ar[d]^1 &&
&& *+[F=]{\bsm 2 \esm}\ar[dl]_1\ar[dr]^2
\\
&*+[F=]{\bsm 1\\1\\2 \esm} \ar[d]_2 && 
*+[F]{\bsm 2\\1\\1 \esm}\ar[dl]_1\ar[dr]^2 &&
&*+[F=]{\bsm 1\\1\\2 \esm}\ar[d]_2 && 
*+[F]{\bsm 2 \esm \oplus \bsm 2 \esm}\ar[d]^1
\\
& *+[F=]{\bsm &1\\2&&1\\&&2 \esm}\ar[dl]_1\ar[dr]^2 & *+[F]{\bsm 1\\1 \esm \oplus \bsm 2\\1\\1 \esm}\ar[dr]^2 && *+[F=]{\bsm 2\\1\\1\\2 \esm}\ar[dl]_1
&
&*+[F=]{\bsm 2\\1\\1\\2 \esm}\ar[d]^2 && 
*+[F=]{\bsm &1\\2&&1\\&&2 \esm}\ar[d]^1
\\
*+[F]{\bsm &1\\2&&1\\1&&2\\&1 \esm}\ar[dr]^2 && *+[F]{\bsm 2 \esm \oplus \bsm &1\\2&&1\\&&2\esm}\ar[dl]_1 & *+[F]{\bsm 1\\1 \esm \oplus \bsm 2\\1\\1\\2\esm}\ar[d]^2 &&
&*+[F]{\bsm2 \esm \oplus \bsm 2\\1\\1\\2 \esm}\ar[dr]^1 &&*+[F]{\bsm 1\\1\\2 \esm \oplus \bsm 1\\1\\2\esm}\ar[dl]_2
\\
& *+[F]{\bsm 2 \esm \oplus \bsm &1\\2&&1\\1&&2\\&1\esm}\ar[d]_1 && *+[F=]{\bsm 1\\1\\2 \esm \oplus \bsm 2\\1\\1\\2\esm}\ar[d]^2 &&
&&*+[F=]{\bsm 1\\1\\2 \esm \oplus \bsm 2\\1\\1\\2\esm}\ar[d]^2
\\
& *+[F]{\bsm 2\\1\\1 \esm \oplus \bsm &1\\2&&1\\1&&2\\&1\esm}\ar[dr]^2 && *+[F]{\bsm &1\\2&&1\\&&2 \esm \oplus \bsm 2\\1\\1\\2\esm}\ar[dl]_1&&
&&*+[F]{\bsm 2\\1\\1\\2 \esm \oplus \bsm 2\\1\\1\\2 \esm}
\\
&&*+[F]{\bsm 2\\1\\1\\2 \esm \oplus \bsm &1\\2&&1\\1&&2\\&1\esm}&&
&
}
\]
\caption{The geometric crystal graphs $(\cB_{\vpi_1+\vpi_2}^*,\tilde{e}_i)$ and $(\cB_{2\vpi_2},\tilde{e}_i^*)$
for type $B_2$.}
\label{Fig:crystalB2simplemoduleb}
\end{figure}
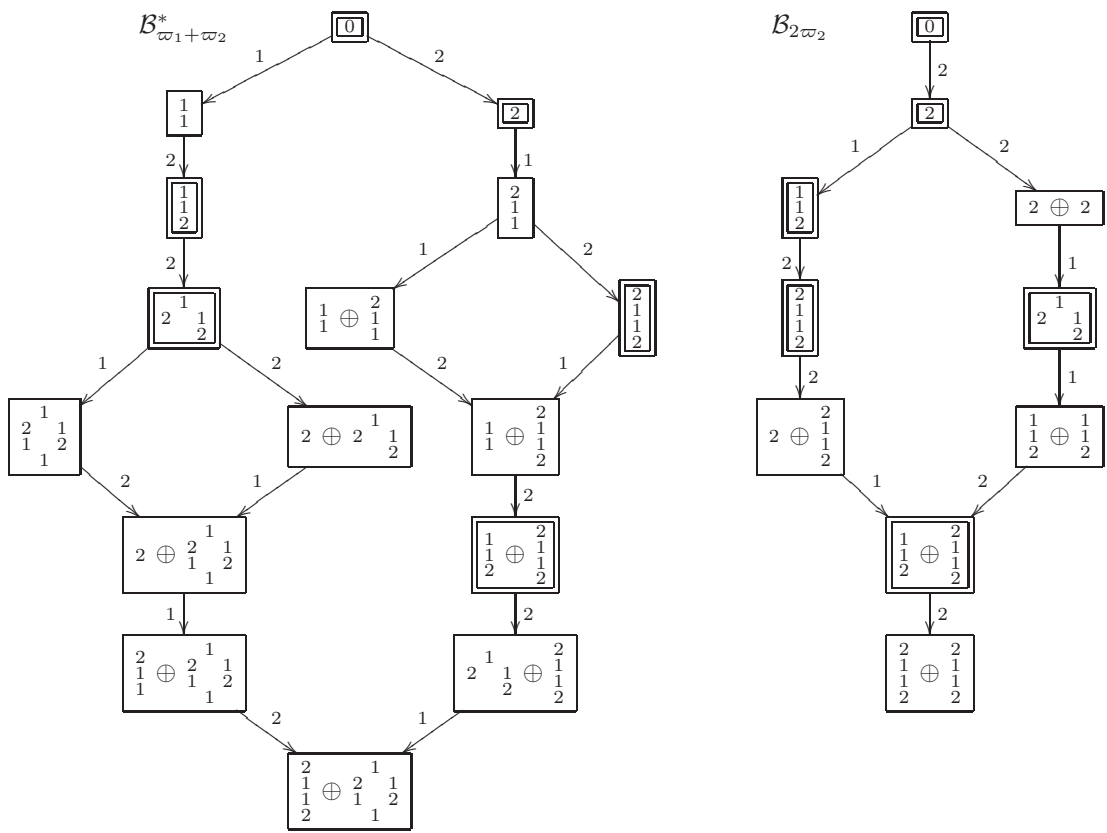

\subsection{Type $G_2$}\label{sec:exampleG2}
Let $\Pi = \Pi(C,D)$ with
\[
C = \left(\bbm 2&-1\\-3&2 \ebm\right)
\text{\;\;\; and \;\;\;}
D = \left(\bbm 3&0\\0&1 \ebm\right).
\]
Thus $C$ is a Cartan matrix of Dynkin type $G_2$, and $D$ is
minimal.
We have
$$
\Pi = K\ov{Q}/\ov{I}
$$
where $\ov{Q} = \ov{Q}(C)$ is the quiver
\[
\xymatrix{
1 \ar@(ul,dl)_{\vep_1} \ar@<0.5ex>[r]^{\alpha_{21}}& 2 \ar@<0.5ex>[l]^{\alpha_{12}}
}
\]
and $\ov{I}$ is generated by the set
$$
\{
\vep_1^3,\;
\alpha_{12}\alpha_{21}\vep_1^2  + \vep_1\alpha_{12}\alpha_{21}\vep_1+
\vep_1^2\alpha_{12}\alpha_{21}, \;
- \alpha_{21}\alpha_{12} \}.
$$

In Figure~\ref{Fig:crystalG2} we display part
of the geometric crystal graph $(\cB,\tilde{e}_i) \equiv (B(-\infty),\tilde{e}_i)$ 
of type $G_2$.
One of the components has a double frame.
This component does not have a dense orbit, but it contains a dense $K^*$-family of orbits of $\Pi$-modules $Q(\lambda)$ with $\lambda \in K^*$, 
which we define
as follows: 
\begin{align*}
Q(\la) &=
{\xymatrix@-1.0pc{
& 2\ar[dd]\ar[dl]_{\la}
\\
1 \ar[dr]
\\
& 1 \ar[dr] 
\\
&& 1 \ar[dl]
\\ 
&2 
}}
\cong
{\xymatrix@-1.0pc{
& 2\ar[dl]
\\
1 \ar[dr]
\\
& 1 \ar[dr]\ar[dd] 
\\
&& 1 \ar[dl]^\lambda
\\ 
&2 
}}
\end{align*}
Note that $Q(\la)$ is $E$-filtered. 

\begin{figure}[!htb]
\small
\[
\xymatrix@-1.8ex{
&&&& *+[F]{\bsm 0 \esm} \ar[dl]_1\ar[dr]^2
\\
&&&*+[F]{\bsm 1\\1\\1\esm} \ar[dl]_1\ar[d]_2 
&& *+[F]{\bsm 2\esm} \ar[d]_1\ar[drr]^2
\\
&&*+[F]{\bsm 1\\1\\1\esm \oplus \bsm 1\\1\\1\esm} \ar[dl]_1\ar[d]_2
&
*+[F]{\bsm 1\\1\\1\\2\esm} \ar[dl]_1\ar[d]_2
&&
*+[F]{\bsm 2\\1\\1\\1\esm} \ar[d]_1\ar[dr]^2
&&
*+[F]{\bsm 2\esm \oplus \bsm 2\esm} \ar[d]_1\ar[dr]^2
\\
&*+[F]{\bsm 1\\1\\1\esm \oplus \bsm 1\\1\\1\esm \oplus \bsm 1\\1\\1\esm} &
*+[F]{\bsm 1\\1\\1\esm \oplus \bsm 1\\1\\1\\2\esm} &
*+[F]{\bsm 1\\&1\\2&&1\\&2\esm}&&
*+[F]{\bsm 1\\1\\1\esm \oplus \bsm 2\\1\\1\\1\esm} &
*+[F=]{\bsm &2\\1\\&1\\&&1\\&2\esm}&
*+[F]{\bsm &2\\1&&2\\&1\\&&1\esm}&
*+[F]{\bsm 2\esm \oplus \bsm 2\esm \oplus \bsm 2\esm} 
}
\]
\caption{The first four layers
of the geometric crystal graph $(\cB,\tilde{e}_i) \equiv (B(-\infty),\tilde{e}_i)$ 
of type $G_2$.}
\label{Fig:crystalG2}
\end{figure}
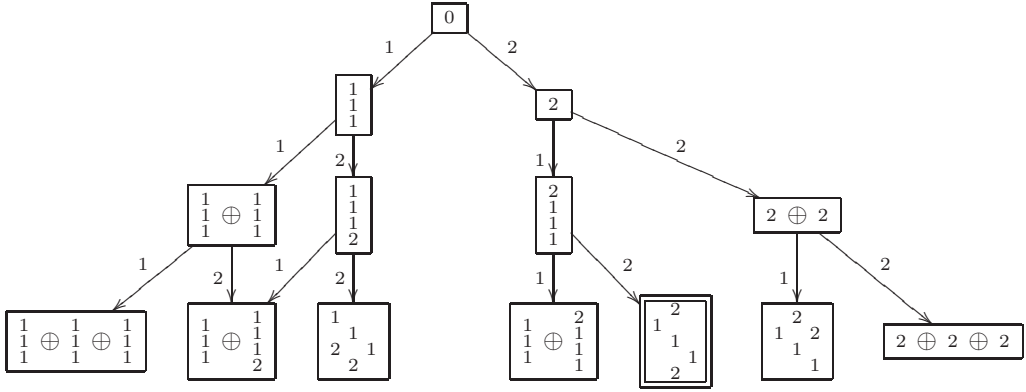

\subsection{Type $A_2$}\label{sec:exampleA2}
Let $\Pi = \Pi(C,D)$ with 
\[
C = \left(\bbm 2&-1\\-1&2 \ebm\right).
\]
Thus $C$ is a Cartan matrix of Dynkin type $A_2$.
For the minimal symmetrizer
\[
D = \left(\bbm 1&0\\0&1 \ebm\right)
\]
each irreducible component of $\nil_E(\Pi,d)$ is maximal.
This is no longer true if $D$ is non-minimal.
From now on assume that
\[
D = \left(\bbm 2&0\\0&2 \ebm\right).
\]
Thus we have $\Pi = \Pi(C,D) = K\ov{Q}/\ov{I}$ where
$\ov{Q} = \ov{Q}(C)$ is the quiver
\[
\xymatrix{
1 \ar@(ul,dl)_{\vep_1} \ar@<0.5ex>[r]^{\alpha_{21}}& 2 \ar@<0.5ex>[l]^{\alpha_{12}}\ar@(ur,dr)^{\vep_2}.
}
\]
and $\ov{I}$ is generated by the set
$$
\{ \vep_1^2,\; \vep_2^2,\; \vep_1\alpha_{12} - \alpha_{12}\vep_2,\;
\vep_2\alpha_{21}- \alpha_{21}\vep_1,\; \alpha_{12}\alpha_{21},\; 
-\alpha_{21}\alpha_{12} \}.
$$
The preprojective algebra $\Pi$ is a finite-dimensional
special biserial algebra.

Up to isomorphism there are $4$ indecomposable rigid
$\Pi$-modules, namely
\begin{align*}
{P_1 = \bsm &1\\1&&2\\&2\esm}, &&
{P_2 = \bsm &2\\2&&1\\&1\esm}, &&
{E_1 = \bsm 1\\1\esm}, && 
{E_2 = \bsm 2\\2\esm}.
\end{align*}
Let $d = (4,2)$.
We have $\dim G(d) = 20$ and $\dim G(d) - q_{DC}(d/D) = 14$.
We define an
indecomposable locally free $\Pi$-module 
$X$ as follows:
\begin{align*}
X = & 
{\xymatrix@-1.0pc{
&&1 \ar[r]\ar[d] & 1
\\
& 2 \ar[r]\ar[d] & 2
\\
1 \ar[r] & 1
}} 
\end{align*}
The module $X$ is obviously $E$-filtered.
The variety $\nil_E(\Pi,d)$ has $3$ irreducible components, namely
$$
Z_1 := \ov{\cO(P_1 \oplus E_1)}, \qquad Z_2 := \ov{\cO(P_2 \oplus E_1)},
\qquad Z_3 := \ov{\cO(X)}.
$$
We have 
$\dim(Z_1) = \dim(Z_2) = 14$ and $\dim(Z_3) = 13$.

\bigskip
{\parindent0cm \bf Acknowledgements.}\,
We thank Peter Tingley and Vinoth Nandakumar for providing us with a preliminary version of their preprint \cite{NT}.
The second and third author thank the Mittag-Leffler Institute
for kind hospitality in February/March 2015. 
The third author thanks the SFB/Transregio TR 45 for 
financial support.
The first author thanks the Mathematical Institute of the University of Bonn for one month of hospitality in June/July 2016.



\begin{thebibliography}{999}


\bibitem[BR]{BR}
M.C.R. Butler, C.M. Ringel,
\emph{Auslander-Reiten sequences with few middle terms and applications to string algebras}. 
Comm. Algebra 15 (1987), no. 1-2, 145--179. 

\bibitem[CB1]{CB1}
W. Crawley-Boevey,
\emph{Maps between representations of zero-relation algebras}.
J. Algebra 126 (1989), no. 2, 259--263.

\bibitem[CB2]{CB2}
W. Crawley-Boevey,
\emph{On the exceptional fibres of Kleinian singularities}.
Amer. J. Math. 122 (2000), no. 5, 1027--1037.

\bibitem[CBS]{CBS}
W. Crawley-Boevey, J. Schr\"oer,
\emph{Irreducible components of varieties of modules}.
J. Reine Angew. Math. 553 (2002), 201--220. 
 
\bibitem[GLS1]{GLS3}
C. Gei{\ss}, B. Leclerc, J. Schr\"oer,
\emph{Quivers with relations for symmetrizable Cartan matrices I: Foundations}.
Invent. Math. (2016). 
doi:10.1007/s00222-016-0705-1, arXiv:1410.1403

\bibitem[GLS2]{GLS4}
C. Gei{\ss}, B. Leclerc, J. Schr\"oer,
\emph{Quivers with relations for symmetrizable Cartan matrices II: Change of symmetrizer}.  
Int. Math. Res. Not. (to appear), arXiv:1511.05898

\bibitem[GLS3]{GLS5}
C. Gei{\ss}, B. Leclerc, J. Schr\"oer,
\emph{Quivers with relations for symmetrizable Cartan matrices III: Convolution algebras}. 
Represent. Theory 20 (2016), 375--413.

\bibitem[H]{H}
N. Haupt, 
\emph{Euler characteristics and geometric properties of quivers Grassmannians}. Ph.D. Thesis, University of Bonn (2011).

\bibitem[Ka]{Ka}
V. Kac, 
\emph{Infinite-dimensional Lie algebras}. 
Third edition. Cambridge University Press, Cambridge, 1990. xxii+400pp.

\bibitem[K1]{K1}
M. Kashiwara,
\emph{On crystal bases}. 
Representations of groups (Banff, AB, 1994), CMS Conf. Proc., vol. 16, Amer. Math. Soc., Providence, RI, 1995, pp. 155--197.

\bibitem[K2]{K2}
M. Kashiwara,
\emph{On crystal bases of the $Q$-analogue of universal enveloping 
algebras}. 
Duke Math. J. 63 (1991), no. 2, 465--516.

\bibitem[KS]{KS}
M. Kashiwara, Y. Saito,
\emph{Geometric construction of crystal bases}. 
Duke Math. J. 89 (1997), no. 1, 9--36.

\bibitem[L1]{L1}
G. Lusztig,
\emph{Quivers, perverse sheaves, and quantized enveloping algebras}. 
J. Amer. Math. Soc. 4 (1991), no. 2, 365--421. 

\bibitem[L2]{L2}
G. Lusztig,
\emph{Semicanonical bases arising from enveloping algebras}.
Adv. Math. 151 (2000), no. 2, 129--139.

\bibitem[NT]{NT}
V. Nandakumar, P. Tingley,
\emph{Quiver varieties and crystals in symmetrizable type via
modulated graphs}. Math. Res. Lett. (to appear), arXiv:1606.01876v2

\bibitem[R]{R}
C. M. Ringel,
\emph{The preprojective algebra of a quiver}.
Algebras and modules, II (Geiranger, 1996), 467--480, 
CMS Conf. Proc., 24, Amer. Math. Soc., Providence, RI, 1998. 

\bibitem[S]{S}
A. Schofield,
\emph{Quivers and Kac-Moody Lie algebras}.
Unpublished manuscript, 23pp.

\bibitem[TW]{TW}
P. Tingley, B. Webster,
\emph{Mirkovi{\v c}-Vilonen polytopes and Khovanov-Lauda-Rouquier 
algebras}.  
Compos. Math. 152 (2016), no. 8, 1648--1696.

\end{thebibliography}
\end{document}